\documentclass[12pt]{amsart}

\usepackage{amsmath,tikz}
\usepackage{amsthm}
\usepackage{hyperref}
\usepackage[margin=1.1in]{geometry}

\numberwithin{equation}{section}

\newtheorem{prop}{Proposition}
\newtheorem{lemma}[prop]{Lemma}

\newtheorem{thm}[prop]{Theorem}
\newtheorem{cor}[prop]{Corollary}

\numberwithin{prop}{section}

\theoremstyle{definition}
\newtheorem{defn}[prop]{Definition}

\newtheorem{rmk}[prop]{Remark}

\newcommand{\del}{\partial}
\newcommand{\dt}{\frac{\partial}{\partial t}}
\newcommand{\brs}[1]{\left| #1 \right|}

\newcommand{\gG}{\Gamma}
\renewcommand{\gg}{\gamma}
\newcommand{\gD}{\Delta}
\newcommand{\gd}{\delta}
\newcommand{\gs}{\sigma}

\newcommand{\gl}{\lambda}

\newcommand{\gw}{\omega}
\newcommand{\ga}{\alpha}
\newcommand{\gb}{\beta}
\renewcommand{\ge}{\epsilon}
\newcommand{\N}{\nabla}
\newcommand{\FF}{\mathcal F}

\newcommand{\nm}[2]{\brs{\brs{ #1}}_{#2}}
\renewcommand{\bar}[1]{\overline{#1}}

\newcommand{\TT}{\mathcal T}

\DeclareMathOperator{\Rc}{Rc}
\DeclareMathOperator{\Rm}{Rm}
\DeclareMathOperator{\inj}{inj}
\DeclareMathOperator{\tr}{tr}

\DeclareMathOperator{\grad}{grad}
\DeclareMathOperator{\Vol}{Vol}
\DeclareMathOperator{\diam}{diam}
\DeclareMathOperator{\Area}{Area}

\DeclareMathOperator{\Jac}{Jac}
\DeclareMathOperator{\dist}{dist}

\newcommand{\IP}[1]{\left<#1\right>}

\begin{document}

\title[A concentration-collapse decomposition for $L^2$ flow singularities]{A
concentration-collapse
decomposition for $L^2$ flow singularities}

\begin{abstract} We exhibit a concentration-collapse decomposition of
singularities of fourth order curvature flows, including the $L^2$ curvature
flow
and Calabi flow, in dimensions $n \leq 4$.  The proof requires the development
of several new a priori estimates.  First, we develop a smoothing result
for initial metrics with small energy and a volume growth lower
bound, in the vein of Perelman's pseudolocality result.  Next, we generalize our
technique from prior work to exhibit local smoothing estimates for the $L^2$
flow in the presence of a curvature-related bound.  A final key ingredient is a
new
local $\ge$-regularity result for $L^2$-critical metrics with possibly
nonconstant scalar curvature.  Applications of these results include new
compactness and diffeomorphism-finiteness theorems for smooth compact
four-manifolds satisfying the necessary and effectively minimal hypotheses of
$L^2$ curvature pinching and a volume noncollapsing condition.
\end{abstract}

\date{October 14th, 2013}

\author{Jeffrey Streets}

\thanks{The author gratefully acknowledges support from the National Science
Foundation and the Sloan Foundation.}

\maketitle

\section{Introduction}
\subsection{Background}

Given $M^n$ a smooth compact manifold, consider the functional of Riemannian
metrics
\begin{align*}
\FF(g) := \int_M \brs{\Rm_g}^2_g dV_g.
\end{align*}
The Euler-Lagrange equation for this functional is a fourth-order degenerate
elliptic equation for a Riemannian metric, and critical metrics for $\FF$
provide natural candidates for optimal metrics on manifolds, especially in
dimension $4$.  Indeed, in this dimension critical points for $\FF$ include
well-known classes of metrics such as Einstein metrics and scalar-flat,
half-conformally flat metrics.  The existence of critical metrics in dimension
$4$ has an interesting relationship to the smooth topology of $M$
(\cite{LeBrun1, LeBrun2}).

A natural tool for understanding the existence of critical metrics is via the
negative gradient flow of $\FF$, which we will refer to for convenience as the
\emph{$L^2$ flow}:
\begin{align*}
\dt g =&\ - \grad \FF.
\end{align*}
This is a degenerate parabolic fourth-order equation for the metric $g$, and
general short-time existence of solutions to the $L^2$ flow was established in
\cite{SL21} (cf. \cite{Yu} for the case $n=3$).  In motivating natural questions
and conjectures about the long
time behavior of this flow, it is natural to compare $\FF$ to the Yang-Mills
energy for a connection on a
principal bundle, and to compare the $L^2$ flow to the Yang-Mills flow.  The
behavior of these flows is closely related to the idea of dimensional
criticality:
the Yang-Mills functional and $\FF$ both obey scaling laws which render
dimensions $n=2,3$
``subcritical," dimension $n=4$ ``critical," and dimensions $n \geq 5$
``supercritical."  Correspondingly, Rade \cite{Rade} has shown that for $n=2,3$
the
Yang-Mills flow exists smoothly (with smooth initial data) for all time and
converges to a Yang-Mills connection.  Next, Struwe \cite{Struwe} has shown
that, in
the critical dimension $n=4$, finite time singularities to Yang-Mills flow (if
they exist, a question which is still open), occur only via ``concentration of
energy."  In particular, this implies long time existence of the flow assuming
the initial energy is sufficiently small.  Lastly, in dimension $n \geq 5$
finite time singularities have been
exhibited by Grotowski \cite{Grotowski}.

\subsection{Statement of singularity decomposition}

With these results from Yang-Mills theory as guidance, we conjectured in
\cite{SL2coll} that analogous
results hold for the $L^2$ flow, namely that one has long time existence in
dimensions $n=2,3$ and in dimension $n=4$ with sufficiently small energy.  More
generally, we can imagine that for arbitrary energy in dimension $n=4$,
singularities form via a concentration of energy phenomenon.  The
case $n=2$ of this conjecture was confirmed in \cite{SL2Surf}, and for this
reason we will ignore the case $n=2$ in this paper, although the results
apply in this case as well.  The main results of this paper, Theorems
\ref{thickthin1} and \ref{threefoldtt}, exhibit a concentration-collapse
decomposition of the metric at singular times of the $L^2$ flow in dimensions
$n=3,4$.  This theorem requires the development of several new analytic tools
for understanding $L^2$ flow solutions which are of interest themselves, and
which have interesting compactness and diffeomorphism finiteness statements as
corollaries.  In the statement below, $T_{\mu}(g)$ corresponds to a ``thin" set
of a Riemannian manifold in an appropriate sense and $d(x,T_{\mu}(g_t),t)$ means
the distance between a point and a set in the metric $g_t$.  These and other
basic definitions are made precise in \S \ref{bckgrnd}.

\begin{thm} \label{thickthin1} (Concentration-Collapse Decomposition) For any
$E,\mu > 0$
there exists $\ge(E,\mu) > 0$ so that if $(M^4, g_t)$ is a solution to the $L^2$ flow satisfying
\begin{align*}
\FF(g_0) \leq E,
\end{align*}
then for any $T \in \mathbb R_{\geq 0} \cup \{\infty\}$, and any $x \in M$ such that $\limsup_{t \to T} \brs{\Rm}(x,t) = \infty$, we have either:
\begin{enumerate}
\item For all $r > 0$, $\displaystyle \limsup_{t \to T} \int_{B_{r}(x,t)}
\brs{\Rm}^2 \geq \ge$,
\item $\liminf_{t \to T} d(x,T_{\mu}(g_t),t) = 0.$
\end{enumerate}
\end{thm}

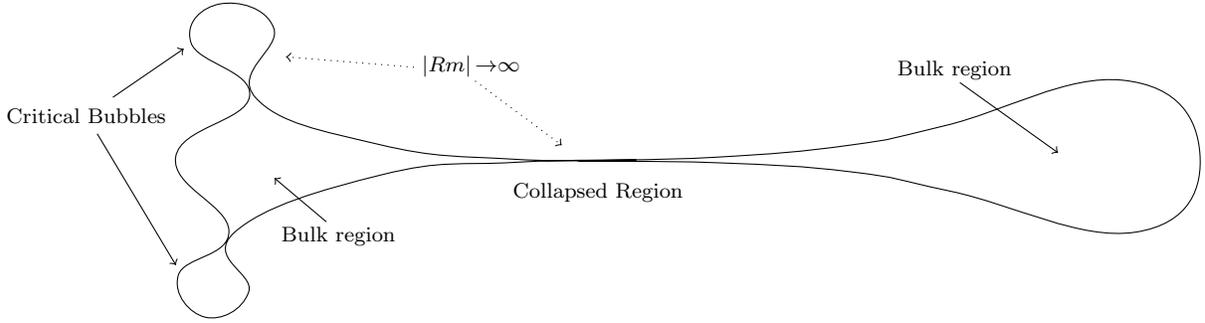
\begin{figure}
\begin{tikzpicture}[scale=2.3]

\draw  plot[smooth cycle, tension=.8] coordinates {(4.5428,4.7265)
(5.5741,4.9892)  (6.017,4.5675)      (5.6444,4.1105) (4.5779,4.3469)
(4.0015,4.4594) (3.2212,4.5122) (2.4256,4.5227)      (3.2141,4.5438)
(3.9945,4.614)};

\draw  plot[smooth cycle, tension=.8] coordinates {(2.7601,4.5286)
(2.0594,4.5321) (1.2717,4.618)  (0.5587,4.8891) (0.6638,5.2913) (0.3545,5.431)
(0.1858,5.2061) (0.5245,4.8976) (0.0972,4.5452) (0.405,4.1173) (0.1118,3.8598)
(0.2678,3.6187) (0.5189,3.7644) (0.4385,4.1108)    (1.2868,4.4422)
(2.0594,4.5145)};
\node at (2.5386,4.3423) {\tiny{Collapsed Region}};
\node (1) at (-0.4195,4.7852) {\tiny{Critical Bubbles}};
\node (2) at (0.1359,3.8572) {};
\node (3) at (0.2062,5.214) {};
\path[->] (1) edge (2);
\path[->] (1) edge (3);
\node at (1.0377,4.0892) {\tiny{Bulk region}};
\node (4) at (1.0307,4.1173) {};
\node (5) at (0.6053,4.4758) {};
\path[->] (4) edge (5);

\node at (4.5984,5.0488) {\tiny{Bulk region}};
\node (7) at (4.5667,5.0207) {};
\node (8) at (5.2592,4.5251) {};
\path[->] (7) edge (8);

\node at (1.804,5.0699) {\tiny{$\left| \displaystyle{Rm} \right|
\displaystyle{\to} \displaystyle{\infty}$}};
\node (10) at (1.5369,5.0593) {};
\node (11) at (0.6722,5.1296) {};
\node (12) at (1.7653,5.0312) {};
\node (13) at (2.391,4.5672) {};
\path[->,dotted] (10) edge (11);
\path[->,dotted] (12) edge (13);
\end{tikzpicture} \caption{Concentration-collapse decomposition of
singularities}
\label{tt}
\end{figure}

The theorem asserts two distinct behaviors which can occur for singularities:
either metric balls of a fixed radius eventually acquire a definite amount of
energy (``concentration'', or ``bubbling''), or the point eventually
becomes arbitrarily close to highly collapsed points (``collapse''). 
Figure \ref{tt} gives a rough sketch of this behavior.  It is moreover possible
to construct singularity models near the concentration points which are critical
metrics
on metric balls of a definite size.  One immediate
corollary is that for sufficiently small initial energy one can ensure the
singular region is collapsed.

\begin{cor} \label{singularitygap} Given $\mu > 0$, there exists $\ge(\mu) > 0$
so that if
$(M^4, g_0)$ satisfies $\FF(g_0) \leq \ge$ and the solution to the $L^2$ flow
with initial condition $g_0$ exists on a maximal time interval $[0,T)$, $T <
\infty$, then for any $x
\in M$ such that $\limsup_{t \to T} \brs{\Rm}(x,t) = \infty$,
\begin{align*}
\liminf_{t \to T} d(x,\mathcal T_{\mu}(g_t),t) = 0.
\end{align*}
\end{cor}

Another corollary is that, assuming the Sobolev constant remains bounded up to
the time of curvature blowup, energy concentration is the only possibility.

\begin{cor} \label{singularityconc} Given $E,\Omega > 0$, there exists
$\ge(E,\Omega) > 0$ so that if $(M^4, g_t)$ is a solution to the $L^2$ flow on
$[0,T), T \leq \infty$ such that
\begin{align*}
\FF(g_0) \leq E, \qquad \sup_{[0,T)} C_S(g_t) \leq \Omega,
\end{align*}
 then for any $x \in M$ such that $\limsup_{t \to T} \brs{\Rm}(x,t) = \infty$,
one has that for all $r > 0$,
\begin{align*}
  \limsup_{t \to T} \int_{B_r(x,t)} \brs{\Rm}^2 \geq \ge.
\end{align*}
\end{cor}

\begin{rmk} For the $L^2$ flow, the behavior described in Corollary
\ref{singularitygap} is strictly confined to low dimensions.  Indeed, for the
example of a shrinking sphere $(S^n,
\sqrt{1 - c_n t} g_{S^n})$, $n \geq 5$ we have that the curvature blows up at
every point, yet $\TT_{\mu} =
\emptyset$ for sufficiently small $\mu$.
\end{rmk}

\begin{rmk} By \cite{CFG} Theorem 1.7 we know that for $\mu$ chosen sufficiently
small, for any metric $g$ the region $T_{\mu}(g)$ admits a nearby metric with a
nilpotent Killing structure.
\end{rmk}

Next we address the case $n=3$, where we assert that for arbitrary initial
energy, the singular points eventually become arbitrarily close to
\emph{arbitrarily collapsed} points.  

\begin{thm} \label{threefoldtt} Let $(M^3, g_0)$ be a compact manifold, and
suppose the solution to the $L^2$ flow
with initial condition $g_0$ exists on a maximal time interval $[0,T)$, $T <
\infty$.  Then for any $x
\in M$ such that $\limsup_{t \to T} \brs{\Rm}(x,t) = \infty$ and any $\mu > 0$,
\begin{align*}
\liminf_{t \to T} d(x,\mathcal T_{\mu}(g_t),t) = 0. 
\end{align*}
\end{thm}

\subsection{Statement of smoothing results and their corollaries}

The first tool is a new group of short-time existence results for the $L^2$ flow
which give a lower bound on the existence time of the flow in terms of a
certain measure of volume-noncollapse defined below.  These smoothing results
have
a number of corollaries on gap phenomena and diffeomorphism finiteness which we
state below.

\begin{defn} Let $(M^n, g)$ be a Riemannian manifold.  Fix $0 < \gd < 1$, and
let $\gw_n$ denote the volume of the unit $n$-ball in $\mathbb R^n$.   Given $x
\in M$, define the \emph{$\gd$-volume radius at $x$} to be
\begin{align*}
r_{\gd} (x) := \sup \left\{ r \geq 0 \ \left| \ \forall s \leq r, \frac{\Vol
B_s(x)}{s^n} \geq \gd
\gw_n \right. \right\}.
\end{align*}
Moreover, define the \emph{$\gd$-volume radius of $M$} to be
\begin{align*}
r_{\gd} := \inf_{x \in M} r_{\gd}(x).
\end{align*}
Certainly $r_{\gd}(x) > 0$ in full generality, whereas $r_{\gd} > 0$ will hold
for compact Riemannian manifolds, but can possibly fail on complete manifolds.
\end{defn}

\begin{thm} \label{mainthm} Given $0<\gd < 1$, there exists $\ge, \iota, A > 0$
depending only on $\gd$ so that if 
$(M^4, g)$ is a compact Riemannian manifold satisfying
\begin{enumerate}
 \item $r_{\gd} \geq \rho$,
 \item $\FF(g) \leq \ge$,
\end{enumerate}
the $L^2$ flow with initial condition $g$ exists on $\left[0,\rho^4 \right]$ and
moreover satisfies the estimates
\begin{enumerate}
\item $\brs{\Rm}_{g_t} \leq A \FF^{\frac{1}{6}}(g_t) t^{-\frac{1}{2}}$,
\item $\inj_{g_t} \geq \iota t^{\frac{1}{4}}$,
\item $\diam_{g_t} \leq 2 (\rho + \diam_{g_0})$.
\end{enumerate}
\end{thm}

\begin{rmk}
\ 
\begin{enumerate} 
\item The reason for the power $\frac{1}{6}$ is technical.  A more elementary
proof can yield that the coefficient in front of the $t^{-\frac{1}{2}}$ decay of
curvature is $o(\FF)$, but we make this dependence precise.  It would be
interesting to improve this exponent on $\FF$ to $\frac{1}{2}$, in
correspondence with the elliptic estimates (\cite{Chang}, \cite{TV}).  One
possibility for attaining this would be to prove an $\ge$-regularity theorem for
sufficiently ``almost-critical" metrics, which could then be used in the proof
of
Theorem \ref{mainthm} to rule out curvature equality in the blowup sequence in
place of the argument we have given.
\item Once the curvature decay is in place, appropriate scale-invariant decay
estimates for all derivatives of curvature follow from (\cite{SL2LTB} Theorem
1.3, cf. Lemma \ref{derivest}).
\item Given $T > 0$, by an appropriately small choice of $\ge$ (depending on
$\gd$ and
$T$) one can guarantee smooth existence on $[0,T \rho^4]$.  We have chosen not
to phrase the result this way as it obscures what is the only real dependency in
the result, which is on $\rho$.
\item It is not the case that the diameter will only double in the setting we
have considered.  By considering homogeneous metrics on $S^3 \times S^1$ one can
show that the diameter can grow at the initial time like $t^{\frac{1}{4}}$,
which suggests the form of estimate (3).
\end{enumerate}
\end{rmk}

\begin{rmk} We observe that an estimate of this kind, controlling the curvature
in $C^0$ with respect to a small energy bound and the volume scale, can only
hold in dimension $n \leq 4$.  This is exhibited by the round sphere.  In
dimensions $n\geq 5$, round spheres shrink homothetically along the $L^2$ flow. 
In particular, their curvature remains a fixed multiple of some $r_{\gd}$ scale.
 But on the other hand the energy is approaching zero as the sphere shrinks, so
we do not obtain $C^0$ control in terms of energy, only in terms of the length
scale.  We will keep track of the dimension in our argument so as to clarify how
the dimensional restriction enters our proof.
\end{rmk}

\begin{rmk} The elliptic version of these estimates comes from work of 
Tian-Viaclovsky (\cite{TV,TV2,TV3} cf. Chang \cite{Chang}).  Chang's
results require an injectivity radius lower bound to obtain $\ge$-regularity. 
In \cite{TV} the authors show $\ge$-regularity for critical metrics assuming a
local Sobolev constant bound, and in \cite{TV3} this is reduced to assuming a
lower volume growth bound, as we have done here.  A crucial issue in
(\cite{TV,TV2,TV3}) is that the critical equation does not automatically
imply a Ricci curvature bound, and obtaining this bound and the attendant volume
comparison bounds requires great care.  In our parabolic case this issue is
compounded due to the fact that the metric is changing in such a way that has no
obvious $C^0$ control.  The primary difficulty is to obtain an estimate on how
the volumes of balls change along the flow, which requires significant new
technique, ``tubular averaging," which is described in detail in \S
\ref{mainthmsec}.
\end{rmk}

\begin{rmk} Despite the obvious fact that Theorem \ref{mainthm} makes global,
not local, hypotheses, it is generally speaking related to Perelman's
pseudolocality
theorem for Ricci flow.  Recall that this theorem says, roughly, that in the
presence of a local lower bound on scalar curvature and isoperimetric constant,
the Ricci flow smooths the metric to the appropriate spacetime scale.  The lower
bound on scalar curvature is a natural assumption for Ricci flow since such
bounds are preserved.  Perelman's proof involves many intricate techniques
including subtle applications of his entropy functional.

The hypotheses of Theorem \ref{mainthm} are similar, in particular we make the
natural assumption of small $L^2$ curvature, which is preserved by the flow, and
a lower volume growth bound, which can be thought of as a rough measure of how
``Euclidean" small metric balls are.  One point of particular interest is that
Perelman's proof works in arbitrary dimensions, whereas ours by necessity
cannot.
\end{rmk}

\begin{rmk} Observe that the statement implies the remarkable fact that that for
a \emph{universal} $\ge > 0$, the flow smooths the metric to the ``minimal
$r_{\gd}$ scale.''  In particular, any point where $r_{\gd}$ is roughly $\rho$
at time zero will have its curvature norm decay so that at time $\rho^4$ we have
$\brs{\Rm} \inj^2
\leq \FF^{\frac{1}{6}}$.  In other words, the metric at time $\rho^4$ already
exhibits collapsing for points at the minimal $r_{\gd}$ scale.  Thus, if one
could appropriately localize these estimates and obtain such a bound at
\emph{all} scales, this would lead to a proof that there is a uniform $\ge > 0$
so that $\FF(g) \leq \ge$ implies the existence of a collapsing structure on the
given manifold.  This is discussed more in \S \ref{mainthmsec}.
\end{rmk}

Next, we state the smoothing theorem for $3$-manifolds.

\begin{thm} \label{mainthm2} Given $0 < \gd < 1$, and $E > 0$ there exists $A =
A(\gd) > 0$ and 
$\ge = \ge(\gd) > 0$ so that if $(M^3, g)$ is a compact Riemannian
manifold satisfying
\begin{enumerate}
\item $r_{\gd} \geq \rho$,
\item $\FF(g) \leq E$,
\end{enumerate}
then the $L^2$ flow with initial condition $g$ exists on $[0, T]$, where $T :=
\min \{\ge^4 E^{-4}, \rho^4\}$,
and moreover the solution satisfies the estimates
\begin{enumerate}
\item $\brs{\Rm}_{g_t} \leq A \left[ T^{\frac{1}{4}} \FF(g_t)
\right]^{\frac{1}{5}} t^{-\frac{1}{2}},$
\item $\inj_{g_t} \geq \iota t^{\frac{1}{4}}$,
\item $\diam_{g_t} \leq 2 (\rho + \diam_{g_0})$.
\end{enumerate}
\end{thm}

As remarked above, Theorems \ref{mainthm} and \ref{mainthm2} can by exploited to
prove new
compactness results, gap theorems, and
diffeomorphism-finiteness theorems.  First we show
finiteness
theorems in dimensions three and four related to theorems of Petersen-Wei
\cite{PetersenWei}, Anderson-Cheeger \cite{ACFin}
and Gao \cite{Gao}.  In particular, in the context of small $L^2$ curvature, our
result replaces the pointwise Ricci
curvature hypothesis of these results with the weaker lower volume growth bound.
 Moreover, what
is notable in particular is that the choice of $\ge$ in our statement for $n=4$ 
is
essentially universal, in the sense that it only depends on $\gd$, and the
initial choice of $\gd$ is
arbitrary.

\begin{cor} \label{diff} Given $0 < \gd < 1$, there exists $\ge(\gd) > 0$ so
that for any $\rho,V > 0$, there are only finitely many diffeomorphism types of
compact Riemannian manifolds $(M^4, g)$ satisfying
\begin{enumerate}
\item $\Vol \leq V$,
\item $r_{\gd} \geq \rho$,
\item $\FF(g) \leq \ge$.
\end{enumerate}
\end{cor}

\noindent Observe that, in the presence of the lower volume estimate for balls
of radius
$\rho$, the upper bound on volume is equivalent to an upper bound on diameter. 

\begin{cor} \label{diff2} Given $0 < \gd < 1$ and $\rho,E,V > 0$, there are only
finitely many diffeomorphism types of
compact Riemannian manifolds $(M^3, g)$ satisfying
\begin{enumerate}
\item $\Vol \leq V$,
\item $r_{\gd} \geq \rho$,
\item $\FF(g) \leq E$.
\end{enumerate}
\end{cor}

By allowing further dependencies on
$\ge$, we obtain long time existence of the flow and convergence to a flat
metric.

\begin{cor} \label{flatcor} Given $0 < \gd < 1$ and $\rho,V > 0$ there exists
$\ge =
\ge(\gd,\rho,V)> 0$ such that given
$(M^4, g)$ a compact Riemannian manifold satisfying
\begin{enumerate}
 \item $\Vol \leq V$,
 \item $r_{\gd} \geq \rho$,
 \item $\FF(g) \leq \ge$,
\end{enumerate}
then the solution to the $L^2$ flow with initial condition $g$ exists for all
time and converges exponentially to a flat metric.
\end{cor}

Using the fact that Sobolev constants can be used to obtain a lower volume
growth estimate, we obtain a further corollary which can be interpreted as a
direct analogue of Struwe's low-energy existence theorem for Yang-Mills flow
remarked on above.  This statement was conjectured in \cite{SL2coll}.  In the
statement below, given a compact Riemannian manifold $(M^4, g)$, the Sobolev
constant $C_S$ is defined to be the smallest positive constant so that for all
$f \in C^{\infty}(M)$ we have
\begin{align*}
\nm{f}{L^4}^2 \leq C_S \left[ \nm{\N f}{L^2}^2 + V^{-\frac{1}{2}} \nm{f}{L^2}^2
\right].
\end{align*}

\begin{cor} \label{sobflatcor} Given $A, V > 0$ there exists $\ge = \ge(A,V) >
0$ so that if $(M^4, g)$ is a compact Riemannian four-manifold satisfying
\begin{enumerate}
\item $\Vol \leq V$,
\item $C_S \leq A$,
\item $\FF(g) \leq \ge$,
\end{enumerate}
then the solution to the $L^2$ flow with initial condition $g$ exists for all
time and converges exponentially to a flat metric.
\end{cor}

The crucial point of interest in this corollary is the fact that only curvature
hypothesis which is made is smallness of the scale-invariant $L^2$ norm of
curvature.  Many precursors exist for this type of result (\cite{ACFin,
DPW,Gao,Minoo,PetersenWei,Yang1,Yang2}), but all
rely on choosing the $L^{\frac{n}{2}}$ norm sufficiently small with respect to
other curvature bounds.  At the heart of these techniques is usually some form
of elliptic theory/Moser iteration, which requires a ``supercritical" $L^p$
bound to get started.  We in particular point out the results in
\cite{DPW,Minoo,Yang1,Yang2}, which
all rely on Ricci flow to smooth out metrics satisfying certain curvature bounds
in order to prove new compactness theorems.  As they rely on parabolic Moser
iteration, these results all require that one choose smallness of the
$L^{\frac{n}{2}}$ norm with respect to some supercritical $L^p$ measure of Ricci
curvature.  We point out that a statement is made in \cite{DPW} claiming to
smooth out metrics with $L^{\frac{n}{2}}$ curvature norm chosen small with
respect only to the Sobolev constant.  However, no proof is given and the author
was not able to reproduce it, as indeed the Moser iteration technique alone only
yields $t^{-1}$ decay of curvature in this setting, which cannot be integrated
over any definite time interval to ensure that the Sobolev constant remains
bounded on that interval.  While it does not appear that Moser iteration alone
can yield such a result, one wonders if by exploiting the various monotonic
quantities introduced by Perelman \cite{P1} that such a result could be
obtained.

More generally, we can obtain the existence of an $F$-structure on a compact
four-manifold with small $L^2$ curvature norm with respect to a scale-invariant
measure of the pointwise lower volume ratio.  Results of this kind begin with
the fundamental collapsing results of Cheeger-Gromov \cite{CG1,CG2}.  Since
then, several improvements have appeared which exploit $L^p$ bounds on curvature
instead of pointwise bounds (\cite{DPW, YangF}).  Our result exhibits an
$F$-structure for manifolds with small $L^2$ curvature, relying on the extra
assumption that the energy is small with respect to how much the
$\gd$-volume radius varies along the manifold.

\begin{cor} \label{Fstructcor} Given $0 < \gd < 1$ and $A > 0$ there exists
$\ge(\gd,A) > 0$ so that if $(M^4, g)$ is a compact Riemannian manifold
satisfying
\begin{enumerate}
\item{$\FF(g) \leq \ge$,}
\item{$\sup_M r_{\gd} \leq A \inf_M r_{\gd}$,}
\end{enumerate}
then $M$ admits a metric $\bar{g}$ satisfying $\sup_M \brs{\bar{\Rm}}
\inj^2_{\bar{g}} \leq C(A,\gd) \ge$.  In particular, $M$ collapses along an
$F$-structure.
\end{cor}

\begin{rmk} To emphasize the type of manifolds which are allowed under Corollary
\ref{Fstructcor}, one can take an arbitrary closed three-manifold $(N^3, g)$,
and construct $M^4 \cong N^3 \times S^1$, and consider the family of metrics
$g_{\ge} = g + \ge g_{S^1}$.  Straightforward calculations show that
$\FF(g_{\ge}) \leq C \ge$, and moreover there is a constant $A$ independent of
$\ge$ such that the pinching condition (2) in the statement of Corollary
\ref{Fstructcor} holds.  Thus, for sufficiently small $\ge$ and any perturbation
of the metric which is small in $C^0$ and yields a small perturbation of $\FF$,
the $L^2$ flow smooths the metric sufficiently to ``recover'' the $F$-structure.
\end{rmk}

The last corollary we obtain is a gap theorem for Einstein
four-manifolds.  Again, prior results of this type have appeared in for
instance \cite{ACFin, Gao, PetersenWei}, and all require bounds on a
``supercritical'' curvature quantity.

\begin{thm} \label{einv} There is a constant $0 < \eta < 1$ with the
following property: given $\rho,A,V > 0$ there exists $\ge = \ge(\rho,A,V) > 0$
such that given $(M^4, g)$ a compact Riemannian manifold satisfying
\begin{enumerate}
\item $\Vol \leq V$,
\item $\chi(M) \leq A$,
 \item $r_{1 - \eta} \geq \rho$,
\item $\nm{z}{L^2} \leq \ge$,
\end{enumerate}
then $M^4$ admits an Einstein metric.
\end{thm}

\begin{rmk} Note that we do not claim long time existence of the $L^2$ flow in
this statement.  This seems quite likely to be true, and verifying it requires a
stability analysis of the flow near Einstein metrics.  Since Einstein metrics
are global
minima for $\FF$ by the Chern-Gauss-Bonnet theorem, we know a priori that such
metrics are linearly stable.  However, information on the integrability of
solutions to the linearized equation is still required to complete the stability
analysis.  This issues are addressed in great detail in recent work of
Gursky-Viaclovsky \cite{GV}.
\end{rmk}

\subsection{Further tools required for Theorem \ref{thickthin1}}

We now return to our main line of discussion, which is the proof of Theorem
\ref{thickthin1}.  We require more a priori estimates for the $L^2$ flow which
are interesting general tools on their own, and so we state them separately. 
The first is a local smoothing estimate in the presence of bounds on curvature
and its first two derivatives.  The proof is a generalization of the
blowup/lifting argument used in \cite{SL2LTB} to show global smoothing estimates
for the $L^2$ flow in the presence of a curvature bound.  See \S \ref{bckgrnd}
for the definitions of the various quantities in this statement.

\begin{thm} \label{smoothing1} (Local Smoothing, cf. Theorem \ref{smoothing2})
Fix $m,n \in \mathbb N^+$.  There
exists a constant $C = C(m,n)$ so
that given $r, K \geq 0$, and $(M^n, g_t)$
a complete solution to $L^2$ flow on $[0,r^4]$, and $x \in M$ satisfying
\begin{align*}
\sup_{P_r(x,r^4)} f_2 \leq K,
\end{align*}
then for all $t \in (0,r^4]$,
\begin{align*}
\sup_{B_r(x,t)} \brs{\N^m \Rm} \leq C \left( (1 + r^{\frac{1}{2}}
\rho_{x,t}^{-\frac{1}{2}}) K + t^{-\frac{1}{2}} +
 \rho_{x,t}^{-2}\right)^{1 + \frac{m}{2}},
\end{align*}
where $\rho_{x,t}(y) := r - d(x,y,t)$.
\end{thm}

After proving Theorem \ref{smoothing1} we will exhibit a natural ``doubling
time" estimate for the $L^2$ flow, which have the further consequence of
showing that $L^2$ flow singularities exhibit a minimal, so-called ``Type I"
blowup rate of curvature.  The next new tool we need is an $\ge$-regularity
result for critical metrics.  

\begin{thm} \label{erthm} (cf. Theorem \ref{er}) Let $(M^4, g)$ be critical
Riemannian manifold, and fix
$r < \frac{\diam(M)}{2}$.  Given $k \geq 0$, there exists $\ge,C_k$ depending on
$C_S$ so that if
\begin{align*}
 \nm{\Rm}{L^2(B_r(p))} \leq \ge,
\end{align*}
 then
 \begin{align*}
  \sup_{B_{\frac{r}{2}}(p)} \brs{\N^k \Rm} \leq \frac{C_k}{r^{2+k}} \left[
\int_{B_r(p)} \brs{\Rm}^2  \right]^{\frac{1}{2}}.
 \end{align*}
\end{thm}

\noindent This proof is similar to related $\ge$-regularity results for Einstein
metrics
\cite{AndersonRicci}, critical metrics \cite{Chang}, Bach-flat metrics
\cite{TV}, and
extremal K\"ahler metrics \cite{CW}.  What differentiates Theorem \ref{erthm}
from the results of \cite{Chang,TV} is that we do not assume
constant scalar curvature, which follows automatically in the settings
considered in those papers.  In
some of our blowup arguments we obtain only a local limit, and thus one cannot
ensure that the scalar curvature is constant.  This causes some technical
difficulties which are
discussed further in \S \ref{ersec}.

The last tool we require is a crucial point-picking statement, Proposition
\ref{ppp}, which says roughly speaking that for points of large curvature in an
$L^2$ flow which are sufficiently noncollapsed, there are points which are
nearby in spacetime where the curvature is still large, and which lie at the
center of parabolic balls.  A crucial difficulty in obtaining this result is to
control the distance between two points which are much farther apart than their
respective curvature scales.  This requires another application of the ``tubular
averaging'' technique introduced  in the proof of Theorem \ref{mainthm}.  With
all of these tools in place we establish Theorems \ref{thickthin1} and
\ref{threefoldtt} in \S \ref{ttsec}.

\begin{rmk} Theorem \ref{thickthin1}, as well as the intermediate results
Theorem \ref{smoothing1}, Theorem \ref{erthm} apply to
solutions to
more general fourth-order curvature flows, including the Calabi flow (in
dimension 4), and the flows introduced in the work of Bour
\cite{Bour}.  Corollaries \ref{diff} and \ref{flatcor} also apply to Calabi
flow,
although in this case one already can identify the complex surfaces in question
from the Kodaira classification.  We will give the appropriate references and
sketches for Calabi flow in \S \ref{cfsec}.
\end{rmk}

Here is an outline of the rest of this paper.  In \S \ref{bckgrnd} we recall
some notation and definitions related to fourth order curvature flows.  In \S
\ref{mainthmsec} we prove Theorems \ref{mainthm} and \ref{mainthm2} and the
attendant corollaries on
rigidity and diffeomorphism finiteness.  Next in
\S \ref{smoothingsec} we prove the local smoothing property (Theorem
\ref{smoothing1}).  Section \ref{ersec} has the proof of Theorem \ref{erthm}. 
We combine these tools in \S \ref{ttsec} to establish the point-picking
algorithm and finish the proof of Theorems \ref{thickthin1} and
\ref{threefoldtt}.  We end by sketching how to modify the proofs to obtain the
results for Calabi
flow in \S \ref{cfsec}.

\vskip 0.1in

\textbf{Acknowledgements:} The author would like to express deep gratitude to
Mark Stern and Gang Tian for their encouragement and 
support of this project in its earliest stages.  Also the author would like to
thank Weiyong He for some comments on an earlier version of this manuscript.

\section{Preliminaries} \label{bckgrnd}
\subsection{Definitions and notation}

\begin{defn} \label{ballsdef} Given $(M^n, g_t)$ a smooth one-parameter family
of metrics on $[A,B]$, and given $x\in M$, $t \in [A,B]$ and $r \geq 0$, we set
\begin{align*}
d(x,y,t) :=&\ \dist_{g_t}(x,y),\\
B_r(x,t) :=&\ \left\{ (y,t) \in M \times \{t\}\ |\ d(x,y,t) \leq r \right\},\\
P_r(x,t) :=&\ \left\{ (y,s) \in M \times [A,B]\ |\  t - r^4 \leq s \leq t,\ y
\in
B(x,r,s) \right\}.
\end{align*}
\end{defn}

Several of our arguments require us to keep track of appropriately
scale-invariant bounds on curvature and a certain finite number of derivatives. 
For this reason we make a convenient definition.

\begin{defn} \label{fdef} Let $(M^n, g_t)$ be a manifold with a one-parameter
family of metrics $g_t$ on $[A,B]$, and fix $k \in \mathbb N^+$.  We define $f_k
: M \times [A,B] \to \mathbb R$ by
\begin{align*}
f_k(x,t) :&= \sum_{j=0}^k \brs{\N^j \Rm}^{\frac{2}{2+j}}_{g(t)}(x).
\end{align*}
\end{defn}

Our next definition is a natural generalization of the usual curvature radius to
include bounds on several derivatives of curvature, appropriately scaled. 
Observe that $\sigma_k$ scales like distance.

\begin{defn} Given $(M^n, g)$ a Riemannian manifold, $x \in M$ and $k \in
\mathbb N^+$, the \emph{$k$-curvature radius at $x$} is defined to be
\begin{align*}
\gs_{k}(x) := \sup \left\{ r > 0\ : \sup_{B_r(x)} f_k \leq r^{-2} \right\}.
\end{align*}
\end{defn}

In the context of a solution to the $L^2$ flow, we can define a version of
$\sigma_k$ using parabolic balls.

\begin{defn} Given $(M^n, g_t)$ a solution to the $L^2$ flow on $[A,B]$, $(x,t)
\in M \times [A,B]$ and $k \in \mathbb N^+$, the \emph{parabolic $k$-curvature
radius at $(x,t)$} is defined to be
\begin{align*}
\bar{\gs}_{k}(x,t) := \sup \left\{ r > 0\ : \sup_{P_r (x,t)} f_k \leq r^{-2}
\right\}.
\end{align*}
\end{defn}

Next, using the definition of $k$-curvature radius we can define a local measure
of ``collapse on the scale of curvature'' as the volume ratio of the ball around
$x$ whose radius is the $k$-curvature radius.
\begin{defn} Given $(M^n, g)$ a Riemannian manifold, $x \in M$ and $m \in
\mathbb N^+$, the \emph{$k$-curvature volume ratio} is defined to be
\begin{align*}
\nu_{k}(x) := \frac{\Vol(B_{\gs_k(x)}(x))}{\gs_k(x)^n}.
\end{align*}
\end{defn}

With this local measure of volume growth, we define ``thin'' points of a
Riemannian
manifold to be those whose $3$-curvature volume ratio is small.  The reason for
choosing $3$-curvature is that various technical local constructions we use to
obtain a priori estimates require knowledge of this many derivatives to succeed.

\begin{defn} \label{thinset} Let $(M^n, g)$ be a Riemannian manifold.  Given
$\mu > 0$ let
\begin{align*}
 \mathcal T_{\mu}(g) = \{x \in M |\ \nu_3(x) \leq \mu \}.
\end{align*}
\end{defn}

\subsection{Riemannian geometry lemmas}

\begin{lemma} \label{arealowerbound} Let $(M^n, g)$ be a complete Riemannian
manifold.  Suppose $\brs{\Rm} \leq K$ and $\inj \geq \iota$.  There are
constants $\mu = \mu(n)> 0$ and $c = c(n) > 0$ so that for all $r < \mu \min
\{\iota, K^{-\frac{1}{2}} \}$ and all $v \in TM$ one has
\begin{align*}
\Area \left[ \exp_{p} \left\{ B_r(0) \cap \IP{v}^{\perp} \right\} \right] \geq c
r^{n-1}.
\end{align*}
\begin{proof} The choice of $\mu$ is made to ensure that for $r$ chosen as in
the statement, $\exp_p$ is a
diffeomorphism on $B_r(0) \subset T_p M$, and moreover one has uniform control
over the differential of $\exp_p$ on this ball.  Thus the area of the disc can
be computed by diffeomorphically pulling it back via $\exp_p$, where the metric
is uniformly equivalent to the standard metric on $T_pM$, and the statement
follows.
\end{proof}
\end{lemma}

\begin{lemma} \label{coarea} \emph{(Coarea Formula)} Let $(M, g)$, $(N, h)$ be
Riemannian manifolds and let $F : M \to N$ be a smooth submersion.  Given $\phi
: M \to [0,\infty)$ a smooth function one has
\begin{align*}
\int_M \phi dV_g = \int_{y \in N} \int_{x \in F^{-1}(y)} \frac{\phi(x)}{N \Jac
F(x)} d F^{-1}(y) dV_h.
\end{align*}
\end{lemma}

\begin{lemma} \label{normchris} Let $(M^n, g)$ be a complete Riemannian manifold
satisfying $f_1 \leq K$.  There exist universal constants $C = C(n) > 0, \mu =
\mu(n) > 0$ such that, in any normal coordinate chart around a point $p \in M$,
one has
\begin{align*}
\sup_{B_{\mu K^{-\frac{1}{2}}}(p)} \brs{\gG} \leq C K^{\frac{1}{2}}.
\end{align*}
\begin{proof} First note that for $\mu$ chosen sufficiently small, normal
coordinates will always exist on $B_{\mu K^{-\frac{1}{2}}}(p)$.  By
appropriately scaling the metric and the choice of normal coordinates, it
suffices to show that statement for $K=1$.  This follows via a detailed analysis
using Jacobi fields similar to the Rauch comparison theorem, and was carried out
in (\cite{Kaul}, cf. \cite{Eich}).
\end{proof}
\end{lemma}

\begin{lemma} \label{cheegers} \emph{(Cheeger's Lemma)}:  Let $(M^n, g)$ be a
complete
Riemannian manifold.  Given $K, R, V > 0$ there exists $\iota(K,R,V) > 0$ so
that if $x \in M$ satisfies $\sup_{B_R(x)} \brs{\Rm} \leq K$ and $\Vol(B_R(x))
\geq V$ then $\inj(x) \geq \iota$.
\begin{proof} This is a direct corollary of \cite{CGT} Theorem 4.7.
\end{proof}
\end{lemma}

\subsection{\texorpdfstring{$L^2$}{L2} flow background}

Here we recall some fundamental facts about solutions to the $L^2$ flow.  As
exhibited in \cite{SL21}, for any dimension the $L^2$ flow is a fourth-order
degenerate parabolic equation, with degeneracies arising from the action of the
diffeomorphism group.  Using the DeTurck gauge-fixing technique it is possible
to prove general short-time existence of solutions on compact manifolds.  
By deriving $L^2$ smoothing estimates for higher derivatives of curvature, one
can establishing a general result ensuring long time existence of the flow as
long as the curvature and Sobolev constant remain bounded along the flow
(\cite{SL21} Theorem 6.2).  By the use of blowup/compactness/covering arguments,
in \cite{SL2LTB} Theorem 1.3 we established that a bound on curvature suffices
to ensure long time existence of the flow.  
In the course of the proof we established global pointwise smoothing estimates
for derivatives of curvature in the presence of a curvature bound.  We record
here an easy corollary of this estimate which 
will be needed below.

\begin{lemma} (Corollary of \cite{SL2LTB} Theorem 1.3) \label{derivest} Fix $m,n
\geq 0$. 
There exists a constant $C = C(m,n) > 0$ so that if $(M^n, g_t)$ is a complete
solution to the $L^2$ flow on $[0,T]$ satisfying
\begin{align*}
\sup_{M \times [0,T]} t^{\frac{1}{2}} \brs{\Rm}_{g(t)} \leq A,
\end{align*}
then for all $t \in(0,T]$,
\begin{align*}
\sup_{M} \brs{\N^m \Rm}_{g(t)} \leq C \left((A + 1)
t^{-\frac{1}{2}}
\right)^{1 + \frac{m}{2}}.
\end{align*}
\begin{proof} Fix a time $t \in (0,T]$ and apply (\cite{SL2LTB} Theorem 1.3) on
the interval $\left[\frac{t}{2}, t \right]$.
\end{proof}
\end{lemma}

\section{Smoothing Results and their corollaries} \label{mainthmsec}

In this section we prove Theorem \ref{mainthm}, and the corollaries on
compactness and diffeomorphism finiteness stated in the introduction.  Before we
begin let us describe the strategy of the proof.  We want to use a
blowup/compactness/contradiction argument to establish the smoothing estimates
of
the theorem.  As long as the injectivity radius estimate holds, one can
construct blowup limits at the scale of curvature and use smoothing estimates to
establish the curvature bound.  Much more difficult is to establish the
injectivity radius
lower bound.  Roughly speaking, by rescaling at a point where the injectivity
radius inequality becomes equality, since there
is a curvature bound at the end time, Cheeger's Lemma (cf. Lemma \ref{cheegers})
shows that the volume of a metric ball around that point in spacetime must be
quite small.  Thus it suffices to show
that the volume of a unit ball cannot get too small too quickly under the
hypotheses of the theorem, since there
will be a lower bound on the volume of unit balls at the initial time of the
flow.

On the one hand, the flow changes the volume form in an easily controlled
manner, but on the other hand it is difficult to see how the flow changes the
distance between points, which is required to establish volume control over
metric balls along the flow.  A direct estimate of the distance using the
smoothing estimates which come from the curvature decay is not possible, as this
requires an integration of $t^{-1}$.  Our method to obtain this estimate is to
use the strong control over $\grad \FF$ coming from the fundamental energy
monotonicity for solutions of the $L^2$ flow.  In particular, by averaging
$\grad \FF$ over a tubular neighborhood around a curve, one can exchange the
integral of $\grad \FF$ over a curve, which one doesn't know too much about,
with one over an open set in $M$, which one has strong control over.  This of
course introduces error terms which must be controlled as well.  The most
technical part of this estimate is that, as the metric is changing over time, we
cannot continue to 
measure how the distance between points is changing by measuring the length of
some fixed curve.  In particular, any such choice of curve which is say a
geodesic at the initial time will quickly acquire geodesic curvature, and then
even the construction of tubular neighborhoods is not clear.  For this reason we
break up the time interval into very small segments on which we can choose a
curve which is very close to being a geodesic, and on which we can carry out the
construction of a controlled tubular neighborhood.

\subsection{Quasi-geodesics}

\begin{defn} Let $(M^n, g_t)$ be a one-parameter family of complete Riemannian
manifolds defined for $t \in [t_1,t_2]$.  Given a constant $\gb > 0$ and $x,y
\in M$ we say that a one-parameter family of curves $\gg(t,s) : [t_1,t_2] \times
[0,1] \to M$ is a \emph{$\gb$-quasi-geodesic connecting $x$ and $y$} if there
is a constant $S > 0$ so that:
\begin{enumerate}
 \item For all $t \in [t_1,t_2]$ one has $\gg(t,0) = x, \gg(t,1) = y$.
 \item For all $n \in \mathbb N^+$ such that $t_1 + n S \leq t_2$, $\gg_{t_1 +
nS}$ is a minimizing geodesic.
 \item For all $n \in \mathbb N^+$ such that $t_1 + n S \leq t_2$, and all $t
\in [t_1 + nS, t_1 + (n+1) S] \cap [t_1,t_2]$, one has $\gg_{t} = \gg_{t_1 +
nS}$.
\item For all $t \in [t_1,t_2]$ one has
\begin{gather} \label{qge}
 \begin{split}
d(x,y,t) \leq L(\gg_t,g_t) \leq&\ d(x,y,t) + \gb, \qquad
\brs{\N_{\dot{\gg}} \dot{\gg}} \leq \gb.
\end{split}
\end{gather}
\end{enumerate}
\end{defn}

Note that the first condition of (\ref{qge}) says intuitively that $\gg$ is
close to a geodesic in the
global sense of nearly minimizing distance, and the second condition says that
the curve is close to being a geodesic in the infinitesimal sense.  The
construction of a $\gb$-quasi-geodesic is as follows: given a
one-parameter family of metrics, given a small constant $\gb > 0$ one can pick
a sufficiently small parameter $S$ so that on time intervals of length $S$ one
can control the distortion of lengths
and second-fundamental forms of curves.  Then one chooses a minimizing geodesic
at each time $t_1 +
nS$, and this defines the piece-wise constant one-parameter family of curves. 
This is made precise in the next lemma.

\begin{lemma} \label{almgeodlemma} Let $(M^n, g_t)$ denote a smooth
one-parameter family of complete
Riemannian manifolds on $[t_1,t_2]$.  Given $\gb > 0$ and $(x,y) \in M$ there
exists a $\gb$-quasi-geodesic connecting $x$ and $y$.
\begin{proof} It suffices to show that for $S$ chosen sufficiently small with
respect to the family $\{g_t\}$ we can obtain the estimate (\ref{qge}) on an
interval of length $S$.  In particular, fix a time $T \in [t_1,t_2]$ and let
$\gg$ denote a minimizing 
geodesic connecting $x$ to $y$ in the metric $g_{T}$, parameterized with
constant speed on $[0,1]$.  We will show that
estimates (\ref{qge}) hold on $[T,T+S]$ for $S$ sufficiently small.  To begin,
let
\begin{align*}
 A := \sup_{M \times [t_1,t_2]} \brs{\frac{\del g}{\del t}} + \brs{\N \frac{\del
g}{\del t}}.
\end{align*}
First, we may estimate
\begin{align*}
\dt L(\gg_t) \leq&\ \int_{\gg} \brs{\frac{\del g}{\del t}} d \gs \leq A
L(\gg_t),
\end{align*}
and therefore
\begin{align*}
L(\gg_t) \leq L(\gg_{T}) e^{A (t - T)} \leq L(\gg_T) e^{AS}.
\end{align*}
Note that in particular this type of estimate implies a bound for $d(x,y)$ along
the whole time
interval $[t_1,t_2]$.  Choosing $S$ small with respect to this upper bound, $A$
and $\gb$ yields
\begin{align*}
L(\gg_t) \leq d(x,y,g_{T}) \left(1 + 2C S \right)) \leq d(x,y,g_{T}) + \gb.
\end{align*}
Next, a direct calculation yields
\begin{align*}
\dt \brs{\N_{\dot{\gg}} \dot{\gg}}^2 \leq&\ \brs{\frac{\del g}{\del t}}
\brs{\N_{\dot{\gg}} \dot{\gg}}^2 + C \brs{\dot{\gg}}^2 \brs{\N_{\dot{\gg}}
\dot{\gg}} \brs{\N \frac{\del g}{\del t}} \leq C (A,L(\gg_{T}))\left(1 +
\brs{\N_{\dot{\gg}}
\dot{\gg}}^2 \right).
\end{align*}
Since $\brs{\N_{\dot{\gg}} \dot{\gg}}_{g_{T}} \equiv 0$ by construction, a
direct estimate of this ODE yields the required estimate for $S$ chosen
sufficiently small with respect to $\gb$.  
With this final choice of $S$ one can build the quasi-geodesic by enforcing the
first
three conditions to hold, and then the estimates above show that (\ref{qge})
will hold.
\end{proof}
\end{lemma}

\subsection{Tubular neighborhoods of quasi-geodesics}

\begin{defn} Let $(M^n, g)$ be a Riemannian manifold, and let $\gg : [a,b] \to
M$ be a smooth curve.  Given $r > 0$, and $t \in [a,b]$, let
\begin{align*}
D_r(\gg(t)) := \exp_{\gg(t)} \left\{ B_r(0) \cap \IP{\dot{\gg}(p)}^{\perp}
\right\},
\end{align*}
and let
\begin{align*}
D_r(\gg) := \bigcup_{t \in [a,b]} D_r(\gg(t)).
\end{align*}
\end{defn}

\begin{lemma} \label{tubularfocal} Given $n,K,\iota > 0$ there exists a constant
$\gb = \gb(n,K,\iota) > 0$ and a universal constant $\mu > 0$ so that if $(M^n,
g)$ is a complete Riemannian manifold satisfying $f_3 \leq K$,
$\inj \geq
\iota$, and $\gg : [0,L] \to M$, is a unit speed curve satisfying
\begin{align*}
\brs{\N_{\dot{\gg}} \dot{\gg}} \leq \gb, \qquad L(\gg) \leq d(\gg(0),\gg(L)) +
\gb,
\end{align*}
then $D_R(\gg)$ is foliated by $D_R(\gg(t))$ for  $R = \mu \min \left\{\iota,
K^{-\frac{1}{2}} \right\}$.  Furthermore, if $\pi : D_R(\gg) \to \gg$ is the
projection map sending a point $p \in D_R(\gg(t))$ to $\gg(t)$, which is
well-defined by the foliation property, then
\begin{align} \label{Jacobianestimate}
\brs{d \pi} \leq 2 \mbox {\quad for all \quad } x \in D_R(\gg).
\end{align}
\begin{proof} First we obtain a local estimate which says that discs arising
from sufficiently close points along $\gg$ cannot intersect.  Consider $s_0,s_1
\in [0,L]$ and suppose there exists $p \in D_R(\gg(s_0)) \cap D_R(\gg(s_1))$. 
Since $R \leq \mu K^{-\frac{1}{2}}$, we can choose a normal coordinate chart
around $p$ of radius $20 R$ around $p$ for $\mu $ chosen sufficiently small. 
Since $\gg(s_0)$ is in this normal coordinate chart, and geodesics originating
at $p$ are radial, it follows that
\begin{align*}
\IP{\frac{\del}{\del r}, \dot{\gg}}_{\gg(s_0)} = 0.
\end{align*}
We will show that, for $\ge$ sufficiently small, this inner product cannot be
zero at any other points of $\gg$ inside our normal coordinate chart.  However,
this will yield a contradiction since there is a geodesic in this chart
connecting $\gg(s_1)$ to $p$, which implies $\IP{\frac{\del}{\del r},
\dot{\gg}}_{\gg(s_1)} = 0$.  To establish this claim we compute the derivative
along $\gg$,
\begin{align*}
\frac{\del}{\del t} \IP{ \frac{\del}{\del r}, \dot{\gg}} =&\ \IP{D_t
\frac{\del}{\del r}, \dot{\gg}} + \IP{\frac{\del}{\del r}, D_t \dot{\gg}}.
\end{align*}
We will estimate the right hand side using various properties of normal
coordinates.  Let $\gg(t) = (\gg^1(t),\dots,\gg^n(t))$ denote the coordinates of
the path $\gg$ in this chart.  Also, let $\del_i$ denote the coordinate vector
fields for this normal coordinate chart, and let $r(x)$ denote the radius
(distance from $p$) of a point $x$.  Using the expression for the radial vector
field in normal coordinates we have
\begin{align*}
\left. \frac{\del}{\del r} \right|_{\gg(t)} = \frac{\gg^i(t)}{r(\gg(t))} \del_i.
\end{align*}
Thus
\begin{align*}
D_t \frac{\del}{\del r} = \frac{\dot{\gg}^i}{r} \del_i - \frac{\gg^i
\IP{\dot{\gg}, \frac{\del}{\del r}}}{r^2} \del_i + \frac{\gg^i}{r} D_t \del_i.
\end{align*}
Since the radial vector field is a unit vector with respect to $g$ we can
estimate
\begin{align*}
\IP{ \gg^i \del_i, \dot{\gg}} = \IP{ r \frac{\del}{\del r}, \dot{\gg}} \leq r
\brs{\dot{\gg}} = r,
\end{align*}
and so
\begin{align*}
\IP{D_t \frac{\del}{\del r}, \dot{\gg}} \geq&\ \frac{1}{r} \brs{\dot{\gg}}^2 -
\frac{C}{r^2} \IP{\dot{\gg}, \frac{\del}{\del r}} r -
\brs{\dot{\gg}} \brs{D_t \del_i} \geq \frac{1 - C \IP{\dot{\gg},\frac{\del}{\del
r}} - r \brs{\gG}}{r}.
\end{align*}
Next, again using that $\frac{\del}{\del r}$ is a unit vector we estimate
\begin{align*}
\brs{\IP{\frac{\del}{\del r}, D_t \dot{\gg}}} \leq \brs{\frac{\del}{\del r}}
\brs{\N_{\dot{\gg}} \dot{\gg}} \leq \gb.
\end{align*}
Now choose $\mu$ sufficiently small so that Lemma \ref{normchris} applies, and
still smaller so that the estimate $\mu K^{-\frac{1}{2}} \brs{\gG} \leq
\frac{1}{4}$ holds.  Next, if we choose $\gb$ sufficiently small with respect to
$R = \mu \min \{\iota, K^{-\frac{1}{2}} \}$ we yield $\gb r \leq C \gb R \leq
\frac{1}{4}$.
 Putting these estimates together, and 
applying Lemma \ref{normchris} yields
\begin{align*}
\frac{\del}{\del t} \IP{\frac{\del}{\del r}, \dot{\gg}} \geq&\ \frac{1 - C
\IP{\dot{\gg},\frac{\del}{\del r}} - r \brs{\gG} - \gb r}{r} \geq \frac{1}{r}
\left[ \frac{1}{2} - C \IP{\frac{\del}{\del r},\dot{\gg}}
\right].
\end{align*}
Thus a straightforward estimate shows that $\IP{\frac{\del}{\del r}, \dot{\gg}}$
cannot vanish other than at $\gg(s_0)$, in the domain of validity of this
estimate, which is inside the normal coordinate chart of radius $20R$.  Since
this estimate holds at every point along $\gg$, we conclude that $D_R(\gg(s_0))
\cap D_R(\gg(s_1)) = \emptyset$ for $\brs{s_1 - s_0} \leq 10 R$.

To conclude that no two discs can intersect for $\brs{s_1 - s_0} > 10 R$, we
turn to the property that $\gg$ nearly minimizes length.  In particular, suppose
there exist $s_0,s_1 \in [0,d]$ such that $\brs{s_1-s_0} > 10 R$ but $p \in
D_R(\gg(s_0)) \cap D_R(\gg(s_1))$.  Construct a piecewise smooth curve $\ga$ as
follows: Follow $\gg$ from $s=0$ to $s = s_0$, then follow the geodesic
connecting $\gg(s_0)$ to $p$, then the geodesic connecting $p$ to $\gg(s_1)$,
then follow $\gg$ from $\gg(s_1)$ to $\gg(L)$.  See Figure \ref{globalcurve}. 
\begin{figure}[ht]
\begin{tikzpicture}

\draw  [color=blue,dashed] plot[smooth, tension=.7] coordinates
{(-6.1367,-0.9343) (-3.6177,-0.8054) (-0.5279,-1.5904) (2.3083,-1.089)        
};

\draw [color=white,fill=white] (-0.4606,-2.8878) rectangle (3.9682,0.0296);

\draw  [color=blue,dashed] plot[smooth, tension=.7] coordinates {(3.4566,1.6463)
(0.668,2.2119) (-2.1205,3.2663) (-4.4555,2.8211) (-6.2447,2.7743)};

\draw  [color=white,fill=white] (0.4181,0.5393) rectangle (3.6343,3.4918);

\draw  plot[smooth, tension=.7] coordinates {(-6.1593,-0.9719) (-3.6403,-0.843)
(-0.5505,-1.628) (2.3083,-1.089)         };
\draw  plot[smooth, tension=.7] coordinates {(3.4566,1.7463) (0.668,2.2619)
(-2.1205,3.3163) (-4.4555,2.8711) (-6.2447,2.8243)};

\draw [fill=black] (2.6833,-0.9719) circle (0.04);
\draw [fill=black] (3.48,-0.7376) circle (0.04);
\draw [fill=black] (3.1051,-0.8547) circle (0.04);
\draw [fill=black] (4.511,1.2308) circle (0.04);
\draw [fill=black] (4.1595,1.4417) circle (0.04);
\draw [fill=black] (3.7846,1.6057) circle (0.04);

\draw [fill=black] (-6.1743,-0.9897) node (v1) {} circle (0.07);
\draw [fill=black] (-6.2364,2.8299) circle (0.07);
\draw [fill=black] (-0.589,0.6272) node (v3) {} circle (0.07);
\draw [fill=black] (-0.5656,-1.599) circle (0.07);
\draw [fill=black] (0.3249,2.3612) circle (0.07);

\draw  (-0.5739,-1.5343) node (v2) {} ellipse (0.4687 and 2.6479);

\draw  [rotate = -30] (-0.8584,2.2303) node (v4) {} ellipse (0.4687 and 2.6479);

\node at (-6.3111,3.1931) {$\gamma(L)$};
\node at (-6.2603,-1.3588) {$\gamma(0)$};
\node at (-0.5642,-1.9915) {$\gamma(s_0)$};
\node at (0.6309,2.7889) {$\gamma(s_1)$};
\node at (-0.6715,0.346) {$p$};

\draw [color=blue,dashed] (v2) edge (v3);

\draw [color=blue,dashed] (v3) edge (v4);

\node at (-3.7647,3.6852) {\textcolor{blue}{$\approx L - s_1$}};
\node at (-0.9527,1.875) {\textcolor{blue}{$\approx R$}};
\node at (-3.6416,-1.2709) {\textcolor{blue}{$\approx s_0$}};
\node at (0.4005,-0.3219) {\textcolor{blue}{$\approx R$}};
\end{tikzpicture}
\caption{Construction of the test curve $\ga$, in blue.}
\label{globalcurve}
\end{figure}
 Since $\gg$ is parameterized by arclength, a straightforward calculation yields
that
\begin{align*}
d(\gg(0), \gg(L)) \leq&\ L(\ga) \leq s_0 + R + R + (L - s_1)\\
\leq&\ d + 2 R - (s_1 - s_0)\\
\leq&\ L(\gg) - 8 R\\
\leq&\ d(\gg(0),\gg(L)) + \gb - 8 R.
\end{align*}
Thus, if $\gb$ is chosen sufficiently small with respect to $R$, which recall
depends on $K$ and $\iota$, we obtain a contradiction here, finishing the claim
that the discs $D_R(\gg(s))$ form a foliation.

We now establish (\ref{Jacobianestimate}).  This can be proved by direct
methods, but we give an argument using blowups.  If the statement were false, we
have a sequence of constants $\mu_i \to 0$ and a sequence of complete Riemannian
manifolds $(M^n_i, g_i)$ satisfying $f_3 \leq K_i$, $\inj \geq \iota_i$, and
curves $\gg_i : [0,L_i] \to M_i$ satisfying
\begin{align*}
\brs{\N_{\dot{\gg_i}} \dot{\gg_i}} \leq \gb_i, \qquad L(\gg) \leq
d(\gg_i(0),\gg_i(L_i)) + \gb_i,
\end{align*}
where $\gb_i \leq 1$, such that $D_R(\gg_i)$ is foliated by $D_R(\gg_i(t))$ for
$r_i := \mu_i \min \{ \iota_i, K_i^{-\frac{1}{2}} \}$, but on the other hand
there exists $p_i = \gg_i(t)$ and $y_i \in D_{r} (\gg_i(t))$ such that
\begin{align} \label{largejac}
\brs{d \pi_i}(y_i) > 2.
\end{align}
Now consider the pointed sequence of Riemannian manifolds $\{(M^n_i, r_i^{-2}
g_i, p_i)\}$.  Observe that in the rescaled manifold, $f_3 \to 0$ pointwise, and
the injectivity radius goes to infinity.  Therefore one can show that a
subsequence converges in the $C^{2,\ga}$ Cheeger-Gromov sense to $(\mathbb R^n,
g_{\mbox{Eucl}}, 0)$.  

However, we observe that, since $\gb_i \leq 1$, after reparameterizing $\gg_i$
to be unit speed with respect to the blowup metric, $\brs{\N_{\dot{\gg_i}}
\dot{\gg_i}} \to 0$ in the limit.  That is, the curves $\gg_i$ converge to a
geodesic through the origin, which we may take without loss of generality to be
the line $\gg(t) = (t,0,\dots,0)$.  Moreover, the points $y_i$ are contained in
the unit ball around $p_i$ in each blowup manifold, and thus can be though of as
points in $B_2(0) \subset \mathbb R^n$ in the convergence process, and these
will converge to a point $y \in D_1(0)$, i.e. the unit disc orthogonal to the
geodesic $\gg$ in $\mathbb R^n$, which is just the intersection of the unit ball
with the plane $x_1 = 0$.  Furthermore, the inequality (\ref{largejac}) is
scaling invariant due to the fact that it involves both the inverse metric on
$M_i$ and the metric induced on $\gg_i$ by $M_i$, and so the scaling factors
cancel out.  Due to the $C^{2,\ga}$ convergence of the Riemannian metrics, and
the 
fact that the projection map $\pi$ sends a point to the point on $\gg$ to which
it is closest, it follows that the maps $\pi_i$ converge in $C^{1,\ga}$ to the
corresponding map $\pi$ on the limit space.  It thus follows that $\brs{d
\pi}(y) > 2$.  But the map $\pi$ on $\mathbb R^n$ is given explicitly as
$(x_1,\dots,x_n) \to (x_1,0\dots,0)$, and certainly $\brs{d \pi} \equiv 1$
everywhere.  This is a contradiction, and the lemma is finished.
\end{proof}
\end{lemma}

\subsection{Proofs of Smoothing Theorems}

\begin{proof}[Proof of Theorem \ref{mainthm}] Fix $0 < \gd < 1$.  We first
observe that the estimates on curvature decay and injectivity radius growth in
the conclusion of theorem are parabolically scaling invariant.  In particular,
it suffices to exhibit the statement corresponding to $\rho = 1$.  In
particular, we claim that there exist $\ge, \iota,$ and $A$ depending only on
$\gd$ so that if $(M^4, g)$ is a compact Riemannian
manifold satisfying
\begin{enumerate}
\item $r_{\gd} \geq 1$
 \item $\FF(g) \leq \ge$,
\end{enumerate}
then the solution to $L^2$ flow with
initial condition $g$ exists on $[0,1]$ and moreover satisfies
\begin{gather} \label{ncsmoothing}
 \begin{split}
 \brs{\Rm}_{C^0(g_t)} \leq&\ A \FF^{\frac{1}{6}}(g_t) t^{-\frac{1}{2}}, \\
 \inj_{g_t} \geq&\ \iota t^{\frac{1}{4}}.
 \end{split}
\end{gather}

Suppose we have such a solution, but that the estimates of
(\ref{ncsmoothing}) do not hold on $[0,1]$.  Then there exists a point $(x,\tau)
\in M
\times [0,1]$ so that $[0,\tau]$ is the maximal interval containing $0$ on
which (\ref{ncsmoothing}) holds.  In particular, one of the inequalities of
(\ref{ncsmoothing}) is an equality at $(x,\tau)$.  Consider the solution to
$L^2$ flow defined by
\begin{align*}
 \bar{g}_t := \tau^{-\frac{1}{2}} g \left( \tau t \right).
\end{align*}
The solution $(M^4, \bar{g}_t)$ exists on $[0,1]$, and moreover satisfies
(\ref{ncsmoothing}) on $[0,1]$ due to the scaling invariance of those estimates.
 Also, one of (\ref{ncsmoothing}) is an equality at $(x,1)$.  First, suppose
that the curvature
inequality of
(\ref{ncsmoothing}) is an equality at $(x,1)$.  Observe that there is a uniform
curvature bound on $[\frac{1}{2},1]$, and moreover a
uniform lower bound on injectivity radius on this time interval.  Without loss
of generality we assume that $x$ satisfies
\begin{align*}
\brs{\bar{\Rm}}(x) = \sup_{M}
\brs{\bar{\Rm}} = A \FF^{\frac{1}{6}}(\bar{g}_1) =: K \leq 1,
\end{align*}
where the last iniequality holds provided $\ge$ is chosen
small with respect to $A$, which ultimately will only depend on $\gd$.  It
follows from (\cite{SL2LTB} Theorem 1.3)
that, for some universal constant $C$ we have
$\sup_{M} \brs{\bar{\N \Rm}} \leq C$.  Note that the solution has not existed
on a sufficiently long time interval to ensure $\brs{\bar{\N \Rm}} \leq C
K^{\frac{3}{2}}$, as one might expect.  In particular, on $B_{\mu K}(x)$ we
obtain $\brs{\bar{\Rm}} \geq \frac{K}{2}$, for a small universal constant $\mu$.
 If $\ge$ is furthermore chosen small with respect to
$\iota$ and $A$ we have that $\mu K \leq \iota$ and so
$\Vol B_{\mu K} \geq  c \mu^4 K^4$ for a small universal constant
$c$.  We conclude
that
\begin{align*}
\FF(\bar{g}_1) \geq \int_{B_{\mu K}(x)} \brs{\bar{\Rm}}^2 \geq
\frac{K^2}{4} \Vol(B_{\mu K_i}(x_i)) \geq c \mu^4 K^6 = c \mu^4  A^6
\FF(\bar{g}_1).
\end{align*}
Thus if $A$ is chosen sufficiently large with universal constants this yields a
contradiction.

Now suppose that the
injectivity
radius inequality of (\ref{ncsmoothing}) is an equality at $(x,1)$.  Our goal
is to derive a contradiction using the volume noncollapsing hypothesis on the
initial metrics.  Recall our noncollapsing hypothesis, which, after
parabolically rescaling by $\rho$ as we have done, says that for all $r
\leq 1$ and $x \in M$,
\begin{align} \label{ncsmoothing10}
 \frac{\Vol B_r(x,\bar{g}_0)}{r^4} \geq \gd.
\end{align}
On the other hand, by Lemma \ref{cheegers}, since we have a curvature bound
$\sup_M \brs{\bar{\Rm}} \leq A \FF^{\frac{1}{6}} \leq 1$, if we choose $\iota$
sufficiently small with respect to $\gd$ we obtain
\begin{align} \label{ncsmoothing20}
 \Vol B_{1}(x,\bar{g}_1) \leq \frac{\gd}{64}.
\end{align}

The strategy of the rest of the proof is to derive a contradiction from these
two inequalities, by showing that the flow can not in fact dissipate volume too
quickly under these hypotheses.  To simplify notation, at this stage of the
proof we remove the bar and refer to the blowup solution
as $(M, g_t)$, defined for $t \in [0,1]$.  The most difficult part, addressed
below, is to show that for $\ge$ sufficiently small one has
\begin{align} \label{ballcont}
U := B_{\frac{1}{2}} (x,g_0) \subset B_1(x,g_1).
\end{align}
Once this is established we can control the volume of $U$ along the flow to
derive a contradiction.  First note that by (\ref{ncsmoothing10})
\begin{align*}
\Vol_{g_0}(U) \geq \frac{\gd}{16}.
\end{align*}
On the other hand, we can estimate
\begin{align*}
\frac{d}{dt} \Vol_{g_t} (U) =&\ \frac{d}{dt} \int_U dV_g\\
=&\ - \frac{1}{2} \int_U \tr \grad \FF dV_g\\
\geq&\ - C \left(\int_U \brs{\grad \FF}^2 dV_g \right)^{\frac{1}{2}}
\Vol_{g_t}(U)^{\frac{1}{2}}.
\end{align*}
A direct calculation then yields
\begin{align*}
\Vol_{g_1}(U)^{\frac{1}{2}} \geq&\ \Vol_{g_0}(U)^{\frac{1}{2}} - C \int_0^1
\left[\int_U \brs{\grad \FF}^2 dV_g \right]^{\frac{1}{2}} dt\\
\geq&\ \left(\frac{\gd}{16} \right)^{\frac{1}{2}} - C \left[ \int_0^1 \int_M
\brs{\grad \FF}^2 dV_g dt \right]^{\frac{1}{2}}\\
\geq&\ \left(\frac{\gd}{16} \right)^{\frac{1}{2}} - C \ge^{\frac{1}{2}}\\
\geq&\ \left( \frac{\gd}{32} \right)^{\frac{1}{2}},
\end{align*}
for $\ge$ chosen sufficiently small with respect to $\gd$ and universal
constants.  Combining this with (\ref{ncsmoothing20}) and (\ref{ballcont}) we
obtain
\begin{align*}
\frac{\gd}{32} \leq \Vol_{g_1}(U) \leq \Vol (B_1(x,g_1)) \leq \frac{\gd}{64},
\end{align*}
a contradiction.

The remainder of the proof  is devoted to establishing (\ref{ballcont}).  The
direct path to showing an estimate on how distance changes is to control the
metric tensor in $C^0$ along a curve, which in turn
requires estimating
\begin{align*}
 \int_0^1 \int_{\gg} \brs{\grad \FF} d\gs dt.
\end{align*}
Observe that a direct estimate of this integral is not possible since naively
one only has $\brs{\grad \FF} \leq C t^{-1}$, just on the borderline of
integrability.  To overcome this problem we need to more cleverly exploit the
fundamental energy estimate for $L^2$ flow solutions.  In particular, since we
have very good control over the ``bulk" integral of $\brs{\grad \FF}^2$, we
essentially replace $\brs{\grad \FF}$ at each point of the curve by its average
over a small disc orthogonal to the curve.  As it turns out this yields
quantities which are integrable, and moreover controllable using the small
energy hypothesis, provided $n \leq 4$.

Fix constants $0 < R < 1, \ga > \frac{1}{4}$ and let $r_t := R t^{\ga}$.  These
constants will be chosen more precisely in the course of the proof.  First,
observe that if $\ge \leq 1$ and $R$ is chosen sufficiently small with respect
to $A$, $\iota$, and the universal constant $\mu$ of Lemma \ref{tubularfocal},
we have via (\ref{ncsmoothing}) that $r_t \leq \mu \max \left\{\inj_{g_t},
\brs{\Rm}_{g_t}^{-\frac{1}{2}} \right\}$ for all $t \in [0,1]$.  It follows from
Lemma \ref{tubularfocal} that we may choose a small parameter $\beta > 0$ so
that a unit speed curve  $\gg : [0,L] \to M$, satisfying
\begin{align*}
\brs{\N_{\dot{\gg}} \dot{\gg}} \leq \gb, \qquad L(\gg) \leq d(\gg(0),\gg(L)) +
\gb,
\end{align*}
has the property that $D_{r_t}(\gg)$ is foliated by the discs $D_{r_t}(\gg(s))$,
and moreover that the map $\pi : D_{r_t}(\gg) \to \gg$ has a bounded normal
differential.  With this choice of $\gb$, fix a point $y \in U$ and choose a
$\gb$-quasi-geodesic from $x$ to $y$ on the time interval $[0,1]$, which is
possible by Lemma \ref{almgeodlemma}.  

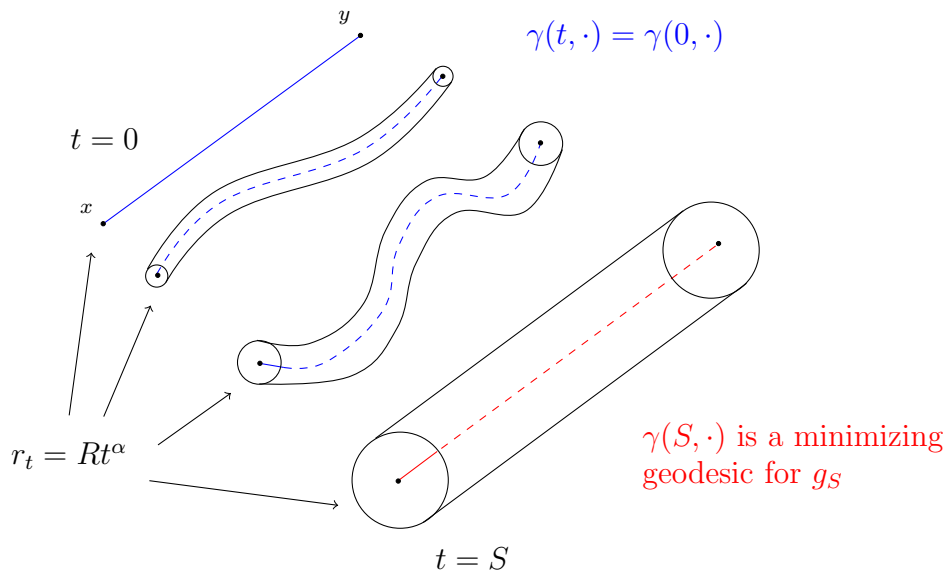
\begin{figure}[ht]
\begin{tikzpicture} [scale=1.4]

\node (v1) at (-6.1978,3.1962) {};
\node (v2) at (-8.8442,1.259) {};
\draw [color=blue] (v1) edge (v2);

\draw  plot[smooth, tension=.7] coordinates {(-8.1478,0.777) (-7.5321,1.4348)
(-6.1964,1.9269) (-5.4363,2.6801)};
\draw  [dashed,color=blue] plot[smooth, tension=.7] coordinates {(-8.235,0.8547)
(-7.6199,1.5402) (-6.2842,2.0323) (-5.5217,2.7403)};
\draw  plot[smooth, tension=.7] coordinates {(-8.3236,0.9025) (-7.6727,1.6457)
(-6.337,2.1378) (-5.5819,2.8106)};

\draw  plot[smooth, tension=.7] coordinates {(-7.261,-0.1878) (-6.3873,-0.1118)
(-5.9515,0.3451) (-5.7722,0.9603) (-5.4032,1.4348) (-4.7353,1.4524)
(-4.3838,2.0393)};
\draw  plot[smooth, tension=.7] coordinates {(-7.2821,0.227) (-6.7759,0.1989)
(-6.2749,0.5701) (-6.064,1.266) (-5.6844,1.737) (-5.1149,1.7828)
(-4.7424,2.2362)};
\draw  [color=blue,dashed] plot[smooth, tension=.7] coordinates
{(-7.0281,-0.0315) (-6.6113,0.0288) (-6.0815,0.5209) (-5.9409,1.1184)
(-5.5719,1.5929) (-4.904,1.6105) (-4.5877,2.1026)};

\node (v3) at (-6.2608,-0.7816) {};
\node (v4) at (-3.1325,1.5332) {};
\draw  (v3) edge (v4);
\node (v5) at (-5.8079,-1.581) {};
\node (v6) at (-2.5311,0.8443) {};
\draw  (v5) edge (v6);
\node (v7) at (-5.6351,-0.8781) {};
\node (v8) at (-2.781,1.2309) {};
\draw  [color=red,dashed] (v7) edge (v8);

\draw [color=blue](-7.2641,0.0087) -- (-7.0281,-0.0315);
\draw [color=red] (-5.9535,-1.1161) -- (-5.5719,-0.8299);

\draw  (-8.2351,0.8372) circle (0.1054);

\draw  (-5.5166,2.7353) circle (0.0954);
\draw  (-7.261,0.0161) circle (0.2071);
\draw  (-4.5877,2.0956) circle (0.2071);
\draw  (-5.9234,-1.103) node (v9) {} circle (0.4569);
\draw  (-2.9708,1.0833) circle (0.4569);

\node at (-8.7294,2.1418) {$t =0$};
\node at (-5.2379,-1.8418) {$t=S$};
\node at (-9.0665,-0.8734) {$r_t = R t^{\alpha}$};

\draw [->] (-9.0665,-0.4867) -- (-8.8556,1.0598);
\draw [->] (-8.7384,-0.4985) -- (-8.2932,0.556);
\draw [->] (-8.2229,-0.7679) -- (-7.5316,-0.2758);
\draw [->] (-8.3049,-1.1077) -- (-6.5123,-1.342);

\draw  [fill=black] (-5.9415,-1.1111) circle (0.02);
\draw  [fill=black] (-4.5908,2.1026) circle (0.02);
\draw  [fill=black] (-7.2551,0.0077) circle (0.02);
\draw  [fill=black] (-5.5187,2.7353) circle (0.02);
\draw  [fill=black] (-8.2252,0.8443) circle (0.02);
\draw  [fill=black] (-6.3001,3.1238) circle (0.02);
\draw  [fill=black] (-8.7433,1.3345) circle (0.02);
\draw  [fill=black] (-2.9001,1.1459) circle (0.02);

\node at (-8.9024,1.4816) {\tiny{$x$}};
\node at (-6.4537,3.2977) {\tiny{$y$}};
\node at (-3.785,3.1142) {\textcolor{blue}{$\gamma(t,\cdot) =
\gamma(0,\cdot)$}};
\node [text width = 4.5cm] at (-2.0159,-0.8693) {\textcolor{red}{$\gamma(S,
\cdot)$ is a minimizing geodesic for $g_S$}};
\end{tikzpicture}
\caption{Time-dependent disc neighborhood around $\gb$-quasi geodesic}
 \label{tube}
\end{figure}

Recall that this construction comes with
a decomposition of $[0,1]$ into finitely many subintervals $I_j := [jS,(j+1)S)$
on which $\gg_t$ is constant.  Figure \ref{tube} is a picture of the situation
in the first interval $[0,S]$.  The indicated blue curves are the same curve
which is a geodesic in the metric $g_0$, but which acquires geodesic curvature
as time
progresses.  The curve is ``tightened" at time $S$ by replacing it by a new
geodesic in the metric $g_S$.  Let us restrict attention to one interval $I_j$. 
Fix a time $t \in I_j$ and give $\gg$ a $g_t$-unit speed parameterization
$\gg:[0,L] \to M$.  Fix an arbitrary $p \in \gg$.  

To begin our estimate, first observe that if we choose $\ge$ small enough to
ensure $A \ge^{\frac{1}{6}} \leq 1$ then, using (\ref{ncsmoothing}),
it follows from Lemma \ref{derivest} that we have the estimates
\begin{align*}
\sup_{M} \brs{\N^m \Rm}_{g_t} \leq C(m,n) t^{-\frac{2+m}{4}}.
\end{align*}
In particular, since $\grad \FF = L(\N^2 \Rc) + \Rm^{*2}$ (\cite{Besse}
Proposition 4.70), we conclude an
estimate of the form $\brs{\N \grad \FF} \leq C t^{-\frac{5}{4}}$, and thus
there is a universal constant $C$ such that, for any $q \in D_{r_t}(p)$,
\begin{align*}
\brs{\grad \FF}(p) \leq \brs{\grad \FF}(q) + C r_t t^{-\frac{5}{4}}.
\end{align*}
Thus we obtain the pointwise estimate
\begin{gather} \label{pntest}
\begin{split}
\brs{\grad \FF}(p,t) =&\ \Area(D_{r_t}(p))^{-1} \int_{D_{r_t}(p)} \brs{\grad
\FF}(p) dA(q)\\
\leq&\ \Area( D_{r_t}(p))^{-1} \int_{D_{r_t}(p)} \left[ \brs{\grad \FF}(q) + C
r_t t^{-\frac{5}{4}} \right] dA(q)\\
=&\ \Area(D_{r_t}(p))^{-1} \int_{D_{r_t}(p)} \brs{\grad \FF}(q) dA(q) + C R
t^{\ga -\frac{5}{4}}\\
\leq&\ \Area(D_{r_t}(p))^{-\frac{1}{2}} \left[ \int_{D_{r_t}(p)} \brs{\grad
\FF}^2(q) dA(q) \right]^{\frac{1}{2}} + C R t^{\ga - \frac{5}{4}}.
\end{split}
\end{gather}
Next, observe that by Lemma \ref{arealowerbound} we have that, for all $s \in
[0,L]$,
\begin{align*}
\Area(D_{r_t}(\gg(s))) \geq c r_t^{n-1} = c R^{n-1} t^{\ga(n-1)}.
\end{align*}
for a small but universal constant $c$.  Putting these estimates together and
applying the coarea formula (Lemma \ref{coarea}), we yield for $t \in I_j$,
\begin{gather} \label{lengthestimate}
\begin{split}
\frac{d}{dt} L(\gg) \leq&\ \int_{\gg} \brs{\grad \FF} d \gs\\
\leq&\ C R^{\frac{1-n}{2}} t^{\frac{\ga(1-n)}{2}} \int_{\gg} \left[
\int_{D_{r_t}(p)} \brs{\grad \FF}^2(q) dA(q) \right]^{\frac{1}{2}} d \gs + C R
t^{\ga - \frac{5}{4}} L(\gg)\\
\leq&\ C R^{\frac{1-n}{2}} t^{\frac{\ga(1-n)}{2}} \left[ \int_{\gg}
\int_{D_{r_t}(p)} \brs{\grad \FF}^2 (q) dA(q) d \gs \right]^{\frac{1}{2}}
L(\gg)^{\frac{1}{2}} + C R t^{\ga - \frac{5}{4}} L(\gg)\\
\leq&\ C R^{\frac{1-n}{2}} t^{\frac{\ga(1-n)}{2}} \sup_{D_{r_t}(\gg)}
\brs{d\pi}^{\frac{1}{2}} \left[\int_M \brs{\grad \FF}^2 dV_g
\right]^{\frac{1}{2}}
L(\gg)^{\frac{1}{2}} + C R t^{\ga - \frac{5}{4}} L(\gg)\\
\leq&\ C R^{\frac{1-n}{2}} t^{\frac{\ga(1-n)}{2}} \left[\int_M \brs{\grad \FF}^2
dV_g \right]^{\frac{1}{2}} L(\gg)^{\frac{1}{2}} + C R t^{\ga - \frac{5}{4}}
L(\gg).
\end{split}
\end{gather}
Recall that, initially, $L(\gg_0) = d(x,y,0) \leq \frac{1}{2}$.  Provided
$d(x,y,jS) < 1$, this differential inequality can be integrated, as long as
$L(\gg) \leq 1$, along the interval $t \in [j S,(j+1) S]$, to yield
\begin{align*}
d(x,y,g_t) \leq d(x,y,g_{jS}) + C R^{\frac{1-n}{2}} \int_{jS}^{jS+1}
t^{\frac{\ga(1-n)}{2}} \left[\int_M \brs{\grad \FF}^2 dV_g \right]^{\frac{1}{2}}
dt + C R \int_{jS}^{jS+1} t^{\ga - \frac{5}{4}} dt.
\end{align*}
Thus after an inductive application of this estimate we obtain for any $t \in
[0,1]$,
\begin{align*}
d(x,y,g_t) \leq&\ d(x,y,g_0) + C R^{\frac{1-n}{2}} \int_0^1
t^{\frac{\ga(1-n)}{2}} \left[\int_M \brs{\grad \FF}^2 dV_g \right]^{\frac{1}{2}}
dt + C R \int_0^1 t^{\ga - \frac{5}{4}} dt\\
\leq&\ \frac{1}{2} + C R^{\frac{1-n}{2}} \left[ \int_0^1 t^{\ga(1-n)} dt
\right]^{\frac{1}{2}} \left[\int_0^1 \int_M \brs{\grad \FF}^2 dV_g dt
\right]^{\frac{1}{2}} + C R \int_0^1 t^{\ga - \frac{5}{4}} dt\\
=:&\ \frac{1}{2} + I + II.
\end{align*}
It is at this point that one can clearly see the dimensional restriction of this
proof.  Controlling the integral in the second term certainly requires a choice
of $\ga > \frac{1}{4}$.  But on the other hand controlling the first integral of
the first term requires $\ga(1-n) > -1$, or $\ga < \frac{1}{n-1}$.  There is an
appropriate choice of $\ga$ satisfying both conditions if and only if $n \leq
4$.  

Now we set $n=4$ in the above estimate and make the final choices of constants. 
First, pick any
$\frac{1}{4} < \ga < \frac{1}{3}$.  For concreteness we choose $\ga =
\frac{7}{24}$.  Then
\begin{align*}
II \leq 24 C R \leq \frac{1}{8},
\end{align*}
as long as $R$ is chosen sufficiently small with respect to the universal
constant $C$.  With this choice we then estimate
\begin{align*}
I \leq C R^{-\frac{3}{2}} \sqrt{8} \left[ \int_0^1 \int_M \brs{\grad \FF}^2
\right]^{\frac{1}{2}} \leq C R^{-\frac{3}{2}} \ge^{\frac{1}{2}} \leq
\frac{1}{8},
\end{align*}
for $\ge$ chosen sufficiently small with respect to $R$, which recall implicitly
depends on $\iota$, which in turn depends on $\gd$.  This finishes the proof of
(\ref{ballcont}), and so finishes the proof of the uniform short-time existence
claim, together with the claimed properties (1) and (2).

It remains to furnish property (3), the diameter estimate, on the
requisite time
interval. We observe that the length estimate (\ref{lengthestimate}) in fact
applies to arbitrary curves which are $\gb$-quasi geodesics for
sufficiently small $\gb$, the existence of which is guaranteed by Lemma
\ref{almgeodlemma}.  Applying this estimate and the arguments above to a
$\gb$-quasi
geodesic connecting two points realizing the diameter of $g_1$ yields the
estimate
\begin{align*}
\frac{d}{dt} L(\gg) \leq&\ C R^{- \frac{3}{2}} t^{- \frac{3 \ga}{2}}
\left[\int_M \brs{\grad \FF}^2 dV \right]^{\frac{1}{2}} L(\gg)^{\frac{1}{2}} + C
R t^{\ga - \frac{5}{4}} L(\gg).
\end{align*}
Integrating the resulting ODE over $[jS,(j+1)S]$ yields, as long as $L(\gg_t) <
2 (1 + \diam(g_0))$,
\begin{align*}
d(x,y,g_{(j+1)S}) \leq&\ d(x,y,g_{jS}) + C R^{-\frac{3}{2}}
(1 + \diam_{g_0})^{\frac{1}{2}} \int_{jS}^{jS+1} t^{-\frac{3 \ga}{2}} \left[
\int_M \brs{\grad \FF}^2 dV_g
\right]^{\frac{1}{2}} dt \\
&\ + C R (1 + \diam_{g_0}) \int_{jS}^{jS+1} t^{\ga - \frac{5}{4}} dt.
\end{align*}
Applying this inductively yields, for all $t \in [0,r_0^4]$,
\begin{align*}
\diam_{g_t} =&\ d(x,y,g_t)\\
\leq&\ d(x,y,g_0) + C R^{-\frac{3}{2}} (1 + \diam_{g_0})^{\frac{1}{2}} \int_0^1
t^{-\frac{3 \ga}{2}} \left[ \int_M \brs{\grad \FF}^2 dV_g \right]^{\frac{1}{2}}
dt \\
&\ + C R (1 + \diam_{g_0})\int_0^1 t^{\ga - \frac{5}{4}} dt.
\end{align*}
As above, choosing $\ga = \frac{7}{24}$, choosing $R$ small with respect to
universal constants (note it has also earlier been chosen small with respect to
$\gd$), and then choosing $\ge$ small with respect to this choice of $R$ we can
obtain
\begin{align*}
\diam_{g_t} \leq&\ d(x,y,g_0) + (1 + \diam_{g_0})^{\frac{1}{2}} + \frac{1}{2} (1
+ \diam_{g_0})\\
\leq&\ \diam_{g_0} + \frac{1}{2} + \frac{1}{2} \left(1 + \diam_{g_0} \right) +
\frac{1}{2}(1 + \diam_{g_0})\\
<&\ 2 (1 + \diam_{g_0}),
\end{align*}
as required.
\end{proof}

\begin{proof} [Proof of Theorem \ref{mainthm2}] The proof is very similar to the
proof of Theorem \ref{mainthm}, and we indicate the required changes here.  Fix
$0 < \gd < 1$ and $\rho, E > 0$.  We claim that there exists $\ge = \ge(\gd)$ so
that if we set
\begin{align*}
 T := \min \left\{ \ge^{4} E^{-4}, \rho^4 \right\},
\end{align*}
then if $(M^3, g_0)$ is a compact Riemannian manifold satisfying $r_{\gd} \geq
\rho$, $\FF(g) \leq E$, then the solution to the $L^2$ flow with this initial
condition exists on $[0,T]$ and satisfies the estimates of the statement of
Theorem \ref{mainthm2}.

As in the proof of Theorem \ref{mainthm}, we begin with a rescaling.  Consider
the solution to $L^2$ flow defined by
\begin{align*}
 \bar{g}_t := T^{-\frac{1}{2}} g_{T t}.
\end{align*}
Observe that, by construction, we have
\begin{align} \label{3folds10}
 \FF(\bar{g}_0) = \FF(T^{-\frac{1}{2}} g_0) = T^{\frac{1}{4}} \FF(g_0) \leq \ge.
\end{align}
Also, 
\begin{align} \label{3folds20}
 r_{\gd}(\bar{g}) = T^{-\frac{1}{4}} r_{\gd}(g) \geq 1.
\end{align}
This is now exactly the setup of the proof of Theorem \ref{mainthm}, and the
proof is almost identical.  We claim that for metrics $\bar{g}_0$ 
satisfying (\ref{3folds10}) and (\ref{3folds20}), the solution to the $L^2$ flow
with this condition exists on $[0,1]$ and moreover satisfies
\begin{align*}
 \brs{\Rm}_{C^0(\bar{g}_t)} \leq&\ A \FF^{\frac{1}{5}}(\bar{g}_t)
t^{-\frac{1}{2}},\\
 \inj_{\bar{g}_t} \geq&\ \iota t^{\frac{1}{4}},\\
\diam_{\bar{g}_t} \leq&\ 2 \left(1 + \diam_{\bar{g}_0} \right),
\end{align*}
where the constants $A$ and $\iota$ depend on $\gd$.  Observe that
the power on $\FF$ has changed from $\frac{1}{6}$ to $\frac{1}{5}$, in
correspondence with the change in dimension.  The argument ruling out equality
of the curvature inequality now proceeds exactly as in the proof of Theorem
\ref{mainthm}.  Moreover, ruling out the injectivity radius equality proceeds in
analogy with the proof of Theorem \ref{mainthm}.  As observed there, since $n
\leq 4$, the method of averaging over disc neighborhoods of radius $r_t = R
t^{\ga}$ can succeed for appropriately chosen $\ga$.  Thus the solution
$\bar{g}_t$ exists on $[0,1]$ with the appropriate curvature decay estimates,
and rescaling back by $T$ yields the statement of the theorem for the solution
$g_t$.  Note in particular that, when rescaling the curvature bound one has
\begin{align*}
\brs{\Rm}_{C^0(g_t)} =&\ \brs{\Rm (T^{\frac{1}{2}} \bar{g}_{T^{-1} t})}_{C^0}\\
=&\ T^{-\frac{1}{2}} \brs{\Rm(\bar{g}_{T^{-1} t})}\\
\leq&\ A \FF(\bar{g}_{T^{-1} t})^{\frac{1}{5}} t^{-\frac{1}{2}}\\
=&\ A \FF^{\frac{1}{5}}(T^{-\frac{1}{2}} g_t) t^{-\frac{1}{2}}\\
=&\ A \left[ T^{\frac{1}{4}} \FF(g_t) \right]^{\frac{1}{5}} t^{-\frac{1}{2}}.
\end{align*}
\end{proof}

\subsection{Rigidity and diffeomorphism finiteness}

We finish this section with the various corollaries on rigidity and finiteness
stated in the introduction.  To begin we use the smoothing effect of Theorem
\ref{mainthm} to establish Corollary \ref{diff}.

\begin{proof}[Proof of Corollary \ref{diff}] Fix $0 < \gd < 1$ and choose
$\ge(\gd)$ as in Theorem
\ref{mainthm}.  Given $\rho,V$ as in the statement, let $\{(M_i, g^i)\}$ be a
sequence of compact manifolds satisfying the hypotheses.  It suffices to show
that there exists a subsequence of $\{M_i\}$ which are diffeomorphic.  Let
$g^i_t$ denote the solution to $L^2$ flow with initial condition $g^i$.  By
Theorem \ref{mainthm} these solutions all exist on $[0,\rho^4]$, and moreover
satisfy the estimates
\begin{align*}
\brs{\Rm}_{g^i_{\rho^4}} \leq C(\rho,\gd), \qquad \inj_{g^i_{\rho^4}} \geq
\iota(\rho,\gd).
\end{align*}
Moreover, since the flow preserves volume we have $\Vol(g^i_{\rho^4}) =
\Vol(g^i)
\leq V$.  It follows from Cheeger-Gromov compactness that there exists a
subsequence of $\left\{(M_i, g^i_{\rho^4}) \right\}$ converging in the
$C^{\infty}$ Cheeger-Gromov sense to a smooth compact four-manifold
$(M_{\infty}, g_{\infty})$, finishing the proof of the corollary.
\end{proof}

\begin{proof} [Proof of Corollary \ref{diff2}] The proof follows the same
argument as the proof of Corollary \ref{diff}, using Theorem \ref{mainthm2} in
place of Theorem \ref{mainthm}.
\end{proof}

\begin{proof}[Proof of Corollary \ref{flatcor}]  Again we argue by
contradiction.  If the statement were false, then given $0 < \gd < 1$ and
$\rho,V > 0$ we could choose a sequence of
compact Riemannian manifolds $\left\{ (M_i,g^i) \right\}$ satisfying
conditions $(1)$ and $(2)$ of the Corollary, and also $\FF(g^i) \leq \ge_i \to
0$, but for which the solutions to the $L^2$ flow with this initial condition
does not exist for all time and converge to a flat metric.  By Theorem
\ref{mainthm}, for sufficiently large $i$ the solutions exist on $[0,\rho^4]$,
and moreover the sequence of Riemannian manifolds $\left\{(M_i, g^i_{\rho^4})
\right\}$ have uniform bounds on curvature and all covariant derivatives of
curvature, a uniform lower bound on injectivity radius, and a uniform upper
bound on volume.  By \cite{CheegerComp} we can choose a subsequence of these
manifolds converging in the $C^{k,\ga}$-Cheeger-Gromov topology for any $k,\ga$,
and the limit is necessarily flat since the $L^2$ norm of curvature approaches
zero along the sequence.  We can now 
invoke \cite{SL22} Theorem 1.6, which is a stability result for solutions to the
$L^2$ flow starting $C^{k}$-close to flat metrics, to conclude that eventually
our solutions do in fact exist for all time and converge to flat metrics,
contradicting the original assumption, and finishing the proof.
\end{proof}

\begin{proof}[Proof of Corollary \ref{sobflatcor}]
Using an argument in \cite{Car2} (cf. \cite{Hebey} Lemma 3.2) one has that if
$(M^n, g)$ is a Riemannian manifold with $C_S \leq A$, then for
any $r > 0$, $x \in M$ one has
\begin{align*}
\Vol B_r(x) \geq \min \left\{ \frac{1}{2 A}, \frac{r}{C(n,q) A} \right\}^n
\end{align*}
Choosing $\rho$ and $\gd$ sufficiently small with respect to $A$ we verify
condition (2) of Corollary (\ref{flatcor}).  If we choose $\ge$ now with respect
to this choice of $\rho,\gd$, which in turn only depends on universal constants
and $A$, we may apply Corollary \ref{flatcor} to obtain the result.
\end{proof}

Next we establish Corollary \ref{Fstructcor}.  Before we begin we will need to
recall a certain notion of weak injectivity radius introduced in \cite{YangF}.

\begin{defn} Given $(M^n, g)$, $0 < \mu < 1$ and $x \in M$, the \emph{$\mu$-weak
injectivity radius at $x$}, denoted $\inj^{\mbox{wk}}(\mu,x,g)$, is the largest
radius $R$ such that
\begin{enumerate}
 \item $C_S(B_{R}(x)) \leq \mu^{-2} C_S(\mathbb R^n)$
 \item For any $B_r(y) \subset B_R(x)$, one has $\Vol(B_r(y)) \leq \mu^{-n}
\gw_n$.
\end{enumerate}
\end{defn}

\begin{proof} [Proof of Corollary \ref{Fstructcor}] Fix constants $\gd, A$ and a
compact Riemannian manifold $(M^4, g)$ as in the statement.  Let $\rho =
r_{\gd}$.  By Theorem \ref{mainthm}, for $\ge$ chosen sufficiently small with
respect to $\gd$ we obtain that the solution to $L^2$ flow with this initial
condition exists on $[0,\rho^4]$ and moreover satisfies the smoothing estimates
claimed in that theorem.  We claim that for $\ge$ and $\mu$ chosen small with
respect to the constants $\gd$ and $A$ that the $\mu$-weak injectivity radius is
bounded \emph{above} at the time $\rho^4$.  More
specifically, we claim
\begin{align} \label{wkinjub}
\inj^{\mbox{wk}}(\mu, x,g_{\rho^4}) \leq C(A,\gd) \rho.
\end{align}
Assuming this is the case,
observe that for the metric $g_{\rho^4}$ we have
\begin{align*}
\left[ \sup_M \inj^{\mbox{wk}}(\mu,x) \right]^2 \int_M \brs{\Rc}^3 \leq \left[
C(A,\gd) \rho \right]^2 \sup_M \brs{\Rm} \int_M \brs{\Rm}^2 \leq C(A,\gd) \ge.
\end{align*}
In particular, for $\ge$ chosen sufficiently small with respect to $C(A,\gd)$
and $\mu$
we can verify the hypotheses of (\cite{YangF} Theorem 0.1) to conclude the
inequality for the metric $g_{\rho^4}$.  The existence of the $F$-structure
follows from \cite{CG2}.

It remains to establish (\ref{wkinjub}).  Fix some point $x \in M$.  We claim
that it suffices to establish a local volume upper bound for the metric at time
$t = \rho^4$.  In particular, we claim that it suffices to show that for $\ge$
chosen sufficiently small with respect to $\gd$ and $A$, if we set $\bar{\gd} =
\frac{1 + \gd}{2}$, we have for all $x \in M$ that
\begin{align} \label{Fstruct10}
r_{\bar{\gd}}(x,g_{\rho^4}) \leq C(A,\gd) \rho.
\end{align}
Indeed, since a local Sobolev constant bound implies a local lower volume growth
bound, (\ref{Fstruct10}) implies an upper bound on the $\mu$-weak injectivity
radius at $x$ for $\mu$ chosen sufficiently small with respct to $C(A,\gd)$.  

To establish (\ref{Fstruct10}) fix $x \in M$, and fix some $\bar{A} \leq A$ such
that
\begin{align*}
r_{\gd}(x,g) = \bar{A} r_{\gd}(g) = \bar{A} \rho.
\end{align*}
In particular, this implies that
\begin{align*}
\Vol(B_{\bar{A} \rho}(x,g)) = \gd \left(\bar{A} \rho \right)^4 \gw_4.
\end{align*}
Let $U = B_{\bar{A} \rho}(x,g)$.  Given some $\eta > 1$, we can fix $\ge$
sufficiently small with respect to
$\eta$, $\gd$, and $A \geq \bar{A}$ so that by following the proof of the volume
estimate in Theorem
\ref{mainthm}, we can estimate
\begin{align*}
\Vol(U,g_{\rho^4}) \leq&\ \eta \gd \bar{A}^4 \rho^4 \gw_4.
\end{align*}
Next we claim that for any choice of $\eta > 1$ we can choose $\ge$ sufficiently
small so that 
\begin{align} \label{Fstructballcont}
B_{\eta^{-1} \bar{A} \rho}(x,g_{\rho^4}) \subset U 
\end{align}
Note that, if this holds, then
\begin{align*}
\Vol B_{ \eta^{-1} \bar{A} \rho}(x,g_{\rho^4}) \leq \Vol(U,g_{\rho^4}) \leq \eta
\gd
A^4
\rho^4 \gw_4 = \eta^{5} \gd \left( \eta^{-1} \bar{A} \rho \right)^4 \gw_4.
\end{align*}
Now if we choose $\eta$ sufficiently close to $1$ so that $\eta^{5} \gd <
\frac{1 + \gd}{2} = \bar{\gd}$, then this finishes the proof of
(\ref{Fstruct10}).  

The last step is to establish (\ref{Fstructballcont}).  This proceeds precisely
as in the proof of the distance estimate (\ref{ballcont}) of Theorem
\ref{mainthm}.  The key difference is that the containment is the reverse of the
one established there. That is, this time we are trying to show that distances
cannot decay too quickly along the flow, which is the same as saying that they
do not grow too fast under the time-reversed flow.  One observes that the
argument for establishing (\ref{ballcont}) only relies on the energy identity
for the flow and the decay properties of curvature.  Thus the argument applies
also if we reverse time, provided choosing $\ge$ small with respect to $A$ and
$\eta$, which depend only on $\gd$.  This establishes (\ref{Fstructballcont}),
finishing the proof.   
\end{proof}

Lastly we establish Theorem \ref{einv}, which requires a slight reworking of the
proof of Theorem \ref{mainthm}.

\begin{proof}[Proof of Theorem \ref{einv}]  Following the proof of Theorem
\ref{mainthm}, we first establish a uniform short time existence statement for
the $L^2$ flow with the given initial conditions.  Observe that by rescaling it
suffices to consider the case $\rho = 1$.  Fix $\rho=1$ and $A,V > 0$.   We will
show that there exists a universal constant $0 < \eta < 1$ and constants
$\ge,\tau > 0$ depending on $A$ and $V$ so that the solution to $L^2$ flow with
initial condition $g$ satisfying $\Vol \leq V, \chi(M) \leq A, r_{1 - \eta}
\geq 1$ and $\nm{z}{L^2} \leq \ge$ exists on $[0,\tau]$ and moreover satisfies
\begin{gather} \label{esmoothing}
 \begin{split}
 \brs{\Rm}_{C^0(g_t)} \leq&\ K t^{-\frac{1}{2}}, \\
\inj_{g_t} \geq&\  t^{\frac{1}{4}}.
 \end{split}
\end{gather}
where $K$ is some universal constant to be determined below.   To prove this we
proceed by blowup/contradiction as in Theorem
\ref{mainthm}.  Suppose that there exists a sequence of manifolds $(M_i,
g^i)$, satisfying $\chi(M_i) \leq A$, $\Vol \leq V$, $r_{1 - \eta} \geq 1$ and
$\nm{z}{L^2}
\leq \ge_i \to 0$, but that along the corresponding solutions to the $L^2$ flow,
the estimates (\ref{esmoothing}) hold only on a maximal time interval
$[0,\tau_i],  \tau_i \to 0$.   As in the proof of Theorem \ref{mainthm}, we
construct the sequence of blowup solutions
$\bar{g}_t^i := \tau_i^{-\frac{1}{2}} g_i(\tau_i t)$.  We aim to derive a
contradiction from this setup.

First, recall the Chern-Gauss-Bonnet formula (cf. \cite{Besse})
\begin{align*}
\FF(g) = 8 \pi^2 \left( \chi(M) + 4 \nm{z}{L^2}^2 \right).
\end{align*}
It follows that the $L^2$ flow in dimension $4$ is also the gradient flow of the
$L^2$ norm of $z$.  Therefore this norm is also nonincreasing along the flow,
and we conclude
\begin{align} \label{e20}
\lim_{i \to \infty} \nm{z}{L^2(\bar{g}^i_{1})} = 0.
\end{align}
To derive a contradiction from this blowup sequence, first suppose that we have
infinitely many metrics $\bar{g}_1^i$ such that the curvature inequality of
(\ref{esmoothing}) is an equality at $(x_i,1)$.  Applying (\cite{SL2LTB}
Corollary 1.5) we may take a limiting solution $(M_{\infty}, g^{\infty}_t,
x_{\infty})$
which exists on $[\frac{1}{2},1]$ and satisfies
\begin{align*}
\sup_{M_{\infty} \times \{1\}} \brs{\Rm} = \brs{\Rm}(x_{\infty}, 1) =  K, \qquad
\inj_{g_1^{\infty}} \geq 1.
\end{align*}
On the other hand, by (\ref{e20}) we have that $g^{\infty}_1$ is an Einstein
metric.  Moreover, since we assume a uniform upper bound on $\chi(M_i)$, it
follows from the Chern-Gauss-Bonnet formula that the $L^2$ norm of the full
Riemann curvature tensor is bounded.  Since the volume of the blowup limit is
infinite and the scalar curvature is constant, it follows that the scalar
curvature must vanish.   Thus $g^{\infty}_1$ is Ricci flat. It follows from the
local harmonic coordinate estimates of Einstein metrics (\cite{Anderson}, cf.
\cite{Petersenbook} Lemma 52, Theorem 76) that there is a constant $C$ depending
ultimately on the dimension (which is $4$) and the injectivity radius lower
bound (which is $1$) such that $\brs{\Rm}_{g^{\infty}_1} \leq C$.  Choosing $K$
larger than this constant $C$ yields a contradiction.

We now turn to the case where there are infinitely many elements of our sequence
where the injectivity radius estimate of (\ref{esmoothing}) is an equality at
$(x^i, 1)$.  We want to follow the same argument as in Theorem \ref{mainthm}. 
However, due to the fact that our curvature estimate in this case depends on the
injectivity radius estimate, we cannot apply Cheeger's Lemma to ensure that
small injectivity radius implies smallness of the volume of geodesic balls. 
Instead we will exploit Anderson's Gap Lemma for complete Ricci flat manifolds
with nearly maximal volume growth.  The crucial point remains the same: to
estimate the volume of metric balls along the flow.

Our main goal is to show that for all sufficiently large metric balls, one has
nearly Euclidean volume.  As in the proof of Theorem \ref{mainthm}, for
notational convenience we will refer to one element of the blowup sequence as
$(M, g_t)$.  Fix some radius $\gs \geq 1$.  Since our blowup factor $\tau_i \to
0$ it follows that for sufficiently large $i$ in our sequence, we have
\begin{align*}
\Vol_{g_0}( B_{\gs}(x,0)) \geq (1-\eta) \gw_4 \gs^4.
\end{align*}
Fix a constant $\mu < 1$ to be determined, but very close to $1$, and let $U :=
B_{\mu \gs}(x, g_0)$.  As in the proof of Theorem \ref{mainthm}, a direct
integration of the evolution of volume yields the estimate
\begin{align} \label{evolest}
\Vol_{g_1}(U)^{\frac{1}{2}} \geq&\ \Vol_{g_0}(U)^{\frac{1}{2}} - C
\ge^{\frac{1}{2}} \geq \left[ (1-\eta) \mu^{4} \gs^4 \gw_4
\right]^{\frac{1}{2}} - C \ge^{\frac{1}{2}} \geq \left[(1 - 2 \eta) \mu^4
\gs^4 \gw_4 \right]^{\frac{1}{2}},
\end{align}
where the last inequality holds for $\ge$ chosen sufficiently small with respect
to universal constants, $\mu$, and $\eta$, but NOT $\gs$.  

We next claim that we can find $\ge$ sufficiently small
with respect to $\mu$ and universal constants so that for blowup metrics in our
sequence satisfying $\nm{z}{L^2(g_0)} \leq \ge$, in analogy with
(\ref{ballcont}) we have
\begin{align} \label{ballcont2}
U := B_{\mu \gs}(x,g_0) \subset B_{\gs}(x,g_1).
\end{align}
Note that by combining (\ref{ballcont2}) with (\ref{evolest}) we obtain for any
$\gs \geq 1$ the volume estimate
\begin{align*}
\Vol_{g_1} B_{\gs}(x,g_1) \geq \Vol_{g_1} U \geq (1 - 2 \eta) \mu^4 \gs^4 \gw_4.
\end{align*}

Assuming the claim of (\ref{ballcont2}), let us finish the proof of the
short-time existence claim. 
The arguments from above ruling out the curvature equality
along the blowup sequence still apply here to conclude that the sequence of
blowups $(M_i, \bar{g}^i_1,x_i)$ admits a limit which is a complete Ricci flat
manifold
with
$\inj_{x_{\infty}} = 1$.  Observe that for any choice of $\zeta > 0$ we can
choose $\eta$
sufficiently small and $\mu$ sufficiently close to $1$ to guarantee $(1 - 2
\eta) \mu^4
\geq (1-\zeta)$.  Assuming $\ge$ is chosen small to guarantee the
validity of (\ref{evolest}) and (\ref{ballcont2}), we conclude moreover that
this limit space satisfies
\begin{align} \label{ein20}
\frac{\Vol B_{\gs}(x,g_{\infty})}{\gs^4} \geq (1-\zeta) \gw_4
\end{align}
for all $\gs \geq 1$.  It follows from the Bishop-Gromov inequality that
(\ref{ein20}) will hold for all $\gs > 0$.  Thus, if $\zeta$ is chosen
sufficiently small we conclude from
(\cite{Anderson} Lemma 3.1) that the limiting manifold is in fact isometric to
$(\mathbb R^4, g_{\mbox{Eucl}},0)$, contradicting the fact that
$\inj_{x_{\infty}} = 1$.

We now proceed to establish (\ref{ballcont2}), following closely the proof of
(\ref{ballcont}) in Theorem \ref{mainthm}.  Again we simplify notation and let
$g_t$ refer to any element of our blowup sequence.  Fix $y \in U$, and choose a
small constant $\gb > 0$ so that $\gb$-quasi-geodesics admit foliated disc
neighborhoods with appropriately chosen radii (cf. Lemma \ref{tubularfocal}). 
With the notation $r_t = R t^{\ga}$ as in Theorem \ref{mainthm}, the preliminary
discussion leading to line (\ref{lengthestimate}) is justified in this setting,
so we may conclude the estimate
\begin{align*}
\frac{d}{dt} L(\gg) \leq C R^{- \frac{3}{2}} t^{- \frac{3 \ga}{2}} \left[\int_M
\brs{\grad \FF}^2 dV_g \right]^{\frac{1}{2}} L(\gg)^{\frac{1}{2}} + C R t^{\ga -
\frac{5}{4}} L(\gg).
\end{align*}
Thus, for all times such that the inequality $L(\gg) \leq \gs$ holds, since $\gs
\geq 1$ we conclude for
$t \in [jS,(j+1)S]$ that
\begin{align*}
d(x,y,g_t) \leq d(x,y,g_{jS}) + \gs \left[ C R^{-\frac{3}{2}} \int_{jS}^{jS+1}
t^{- \frac{3\ga}{2}} \left[\int_M \brs{\grad \FF}^2 dV_g \right]^{\frac{1}{2}}
dt + C R \int_{jS}^{jS+1} t^{\ga - \frac{5}{4}} dt \right].
\end{align*}
Thus after an inductive application of this estimate we obtain for any $t \in
[0,1]$,
\begin{align*}
d(x,y,g_t) \leq&\ d(x,y,g_0) + \gs \left[ C R^{- \frac{3}{2}} \int_0^1
t^{-\frac{3 \ga}{2}} \left[\int_M \brs{\grad \FF} dV_g \right]^{\frac{1}{2}}
dt + C R \int_0^1 t^{\ga - \frac{5}{4}} dt\right]\\
\leq&\ \mu \gs + \gs \left[ C R^{- \frac{3}{2}} \left[ \int_0^1 t^{-3\ga} dt
\right]^{\frac{1}{2}} \left[\int_0^1 \int_M \brs{\grad \FF}^2 dV_g dt
\right]^{\frac{1}{2}} + C R \int_0^1 t^{\ga - \frac{5}{4}} dt \right]\\
=:&\ \gs \left[ \mu + I + II \right],
\end{align*}
where $I$ and $II$ are defined by the final equality.  As in Theorem
\ref{mainthm}, we choose $\ga = \frac{7}{24}$.  Then we obtain
\begin{align*}
II \leq 24 C R < \frac{1-\mu}{2},
\end{align*}
for $R$ chosen sufficiently small with respect to universal constants and $\mu$.
 With this choice we then estimate
\begin{align*}
I \leq C R^{-\frac{3}{2}} \sqrt{8} \left[ \int_0^1 \int_M \brs{\grad \FF}^2
\right]^{\frac{1}{2}} \leq C R^{-\frac{3}{2}} \ge^{\frac{1}{2}} <
\frac{1-\mu}{2},
\end{align*}
as long as $\ge$ is chosen sufficiently small with respect to universal
constants, $\mu$, and $R$ (which already depends on $\mu$ as well).  Combining
these estimates yields, for any $t \in [0,1]$,
\begin{align*}
d(x,y,g_t) < \gs \left[ \mu + \frac{1-\mu}{2} + \frac{1-\mu}{2} \right] = \gs.
\end{align*}
This completes the proof of (\ref{ballcont2}), and so finishes the claim of
uniform short-time existence.

With this smoothing result in hand, the proof proceeds in analogy with the proof
of Corollary \ref{flatcor}.  In particular, if the statement were false, we
could take a sequence of compact Riemannian four-manifolds $(M^4_i, g^i)$
satisfying the hypotheses of the Theorem, and moreover $\lim_{i \to \infty}
\nm{z}{L^2(g_i)} = 0$, but for which no $M^4_i$ admits an Einstein metric.  By
the above discussion, there exists a uniform $\tau > 0$ so that the $L^2$ flow
with initial condition $g^i$ exists on $[0,\tau]$ and moreover satisfies the
estimates (\ref{esmoothing}).  It follows that the sequence of Riemannian
manifolds $\{(M^4_i, g^i_{\tau}) \}$ is precompact, and so some subsequence of
these manifolds converges to a limiting space $(M^4_{\infty}, g_{\infty})$,
which
necessarily is an Einstein manifold.  By the nature of the convergence, $M_i
\cong M_{\infty}$ for sufficiently large $i$, and so we have arrived at a
contradiction.
\end{proof}

\section{Local smoothing property} \label{smoothingsec}

The goal of this section is to establish a local smoothing property
for solutions to the $L^2$ flow in arbitrary dimension.  First we record a lemma
comparing parabolic balls as we have defined them with product regions in
spacetime along an $L^2$ flow.  Then we state a local
compactness theorem for $L^2$ flow solutions, implicit in \cite{SL2LTB},
as Theorem \ref{locacomp} below.  The local smoothing estimates then appear in
Theorem \ref{smoothing2}.  We also record a global curvature
doubling estimate, which in turn implies an upper bound on the blowup rate of
$L^2$ flow solutions, i.e. solutions blow up no faster than the so-called Type I
rate.  Recall that, for the Ricci flow, local smoothing can
be obtained by combining the maximum principle with the use of cutoff functions.
 Moreover, a global doubling time estimate for
the curvature follows directly from the maximum principle applied to the
evolution equation for $\brs{\Rm}^2$ (cf. \cite{CLN} Lemma 6.1).  In the case of
fourth order curvature flows we require the technique of blowup/compactness
arguments.

\begin{lemma} \label{parballcomp} Suppose $(M^n, g_t)$ is a complete solution to
$L^2$ flow, and suppose there exists $x \in M$ such that
\begin{align*}
\sup_{P_r(x,t)} f_2 \leq K.
\end{align*}
There exists a constant $c = c(n)$ such that, setting $\Lambda = e^{-c r^4
K^2}$,
we have $B_{\Lambda r}(x,t) \times [t - r^4,t]
\subset P_{r}(x,t)$.
\begin{proof} One needs to show that for $y \in B_{\Lambda r}(x,t)$ and $s \in
[t
- r^4,t]$ one has $d(x,y,s) \leq r$.  Directly integrating the evolution
equation for $g$ and using the hypothesis on $f_2$ yields a uniform bound for
the metric, and hence the distance, along the flow.
\end{proof}
\end{lemma}

\begin{thm} \label{locacomp} (Local compactness theorem) Let $ \{(M_i^n, g^i_t,
p_i) \}$, $t \in
[0,T]$ be a sequence of complete pointed solutions to the $L^2$ flow.  Fix $m
\geq 4$ and $r > 0$.  Suppose
\begin{align} \label{locacomp10}
\sup_{B_{g^i_0(p_i,r)} \times [0,T]} f_m(g^i) \leq K,
\end{align}
and suppose $\bar{B_{g^i_0}(p_i,r)}$ is compact for all $i$.  There exists a
constant $\ga(n) > 0$ so that setting $\bar{r} = e^{-\ga K^2 T} r$, a
subsequence
of
\begin{align*}
\{(B_{\bar{r}}(p_i,g^i_0), g^i_t,p_i)\}, \qquad t \in [0,T],
\end{align*}
converges to a pointed solution of the $L^2$ flow
\begin{align*}
(B_{\bar{r}}(p_{\infty}, g^{\infty}_0), g^{\infty}_t, p_{\infty}), \qquad t \in
(0,T].
\end{align*}
in the sense of $C^{\infty}$-local submersions, and
$\bar{B_{s}(p_{\infty},g^{\infty}_0)}$ is compact for all $s < \bar{r}$.  If
moreover
\begin{align*}
\inj_{g^i_0}(p_i) \geq \gd > 0
\end{align*}
for all $i$, then the convergence is in the $C^{\infty}$ Cheeger-Gromov sense,
and moreover
\begin{align*}
\inj_{g^{\infty}_0}(p_{\infty}) \geq \gd > 0.
\end{align*}
\begin{proof} We only give a brief sketch, as the proof uses standard ideas
which have already been established for solutions to the $L^2$ flow.  Using
techniques for establishing compactness of one-parameter families of metrics
established in \cite{Hamilton}, we can obtain a limiting space as described in
the theorem, with the sequence converging in the $C^{m-4}(g^{\infty}_0)$
topology.  Moreover, the $L^2$ flow equation passes to this limit.  To improve
the convergence to $C^{\infty}(g^{\infty}_0)$ we exploit local $L^2$ smoothing
estimates of the flow obtained by pulling back to the tangent space via the
exponential map to obtain estimates on higher derivatives of curvature.
\end{proof}
\end{thm}

\begin{thm} \label{smoothing2} Fix $m,n \in \mathbb N^+$.  There
exists a constant $C = C(m,n)$ so
that given $r, K \geq 0$, and $(M^n, g_t)$
a complete solution to $L^2$ flow on $[0,r^4]$, and $x \in M$ satisfying
\begin{align*}
\sup_{P_r(x,r^4)} f_2 \leq K,
\end{align*}
then for all $t \in (0,r^4]$,
\begin{align*}
\sup_{B_r(x,t)} \brs{\N^m \Rm} \leq C \left( (1 + r^{\frac{1}{2}}
\rho_{x,t}^{-\frac{1}{2}}) K + t^{-\frac{1}{2}} +
 \rho_{x,t}^{-2}\right)^{1 + \frac{m}{2}},
\end{align*}
where $\rho_{x,t}(y) := r - d(x,y,t)$.
\begin{rmk} The reason for the seemingly ``extra" term $r^{\frac{1}{2}}
\rho_{x,t}^{-\frac{1}{2}} K$ in this estimate is technical, and is exploited in
Lemma \ref{blowuplemma}.
\end{rmk}
\begin{proof} 
By construction it suffices to show that, for each $m \geq 4$, there exists a
constant $C_m$ such that for a solution satisfying the hypotheses of the
theorem, for all $t \in (0,r^4]$ one has
\begin{align*}
\sup_{B_r(x,t)} f_m \leq C_m \left( (1 + r^{\frac{1}{2}}
\rho_{x,t}^{-\frac{1}{2}} )K + t^{-\frac{1}{2}} + \rho_{x,t}^{-2}
\right).
\end{align*}

Suppose that this claim is false.  Choose a sequence of complete solutions to
the $L^2$ flow $(M^n_i, g^i_t)$ and fix points $x_i \in M^n_i$, and $r_i, K_i >
0$ such that the flow $g^i_t$ exists on $[0,r_i^4]$ and 
\begin{align*}
\sup_{P_{r_i}(x_i,r_i^4)} f_2 \leq K_i.
\end{align*}
Moreover, suppose that there exist points $(z_i,t_i) \in P_{r_i}(x_i,r_i^4)$
such
that
\begin{align} \label{lambnd}
\gl_i := f_m(z_i,t_i) = C_i \left( (1 + r_i^{\frac{1}{2}}
\rho_{x_i,t}^{-\frac{1}{2}} ) K_i + t_i^{-\frac{1}{2}} +
\rho_{x_i,t_i}^{-2} \right), \qquad \lim_{i \to \infty} C_i = \infty.
\end{align}
We aim to derive a contradiction from such a sequence.  Note that since each
solution is smooth and the parabolic balls in question are compact, we may
assume without loss of generality that the points $(z_i,t_i)$ satisfy
\begin{align} \label{supbnd}
\frac{f_m(z_i,t_i)}{(1 + r_i^{\frac{1}{2}}
\rho_{x,t}^{-\frac{1}{2}} ) K_i + t_i^{-\frac{1}{2}} + \rho_{x_i,t_i}^{-2}(z_i)}
=
\sup_{(y,s) \in P_{r_i}(x_i,r_i^4)} \frac{f_m(y,s)}{(1 + r_i^{\frac{1}{2}}
\rho_{x_i,t}^{-\frac{1}{2}} )K_i + s^{-\frac{1}{2}} +
\rho_{x_i,t_i}(y)^{-2}}.
\end{align}
Observe that, since each solution is smooth and the denominator of the terms
above goes to infinity at the boundary of the parabolic ball, the point
realizing the supremum occurs in the interior of the parabolic ball.  

We will construct a blowup sequence around this sequence of points.  Before
that,
we begin with a lemma showing the existence of a parabolic ball with a
controlled radius on which one has a priori control of $f_m$.

\begin{lemma} \label{blowuplemma} For sufficiently large $i$, we have the
estimate 
\begin{align*}
\sup_{P_{\gl_i^{-\frac{1}{2}}}(z_i,t_i)} f_m \leq 16 \gl_i.
\end{align*}
\begin{proof} To begin, we make some observations which follow directly from
(\ref{lambnd}).  In particular, we may choose $N$ sufficiently large so that for
all $i \geq N$ one has
\begin{gather} \label{lambnd2}
 \begin{split}
 \gl_i^{-2} \leq&\ \frac{3}{4} t_i\\
 \gl_i^{-\frac{1}{2}} \leq&\ \frac{1}{16} \rho_{x_i,t_i}(z_i),\\
 \gl_i^{-2} \leq&\ \frac{1}{C K_i^2},\\
 \gl_i^{-2} \leq&\ \frac{\rho_{x_i,t_i}(z_i)}{32 C r_i K_i^2},
 \end{split}
\end{gather}
where the constant $C$ above is a universal constant to be determined below.

First note that for all
points $(y,s) \in P_{r_i}(x_i,r_i^4)$ satisfying
\begin{align} \label{bul1}
d(x_i,y,s) \leq \frac{3}{4} r_i +
\frac{1}{4} d(x_i,z_i,t_i)
\end{align}
one has that
\begin{align*}
\rho_{x_i,s}(y) = r_i - d(x_i,y,s) \geq r_i - \left( \frac{3}{4} r_i +
\frac{1}{4} d(x_i,z_i,t_i) \right) = \frac{1}{4} \rho_{x_i,t_i}(z_i).
\end{align*}
By the estimate (\ref{lambnd2}), for $s \geq t_i
- \gl_i^{-2}$ one has $s^{-\frac{1}{2}} \leq 2
t_i^{-\frac{1}{2}}$.  Thus for
$(y,s)$ satisfying (\ref{bul1}) and $s \geq t_i - \gl_i^{-2}$ one has by
(\ref{supbnd}) that
\begin{gather} \label{supbnd3}
\begin{split}
f_m(y,s) \leq&\ \gl_i \frac{(1 + r_i^{\frac{1}{2}}
\rho_{x_i,t_i}^{-\frac{1}{2}}(y) ) K_i + s^{-\frac{1}{2}} +
\rho_{x_i,s}(y)^{-2}}{(1
+ r_i^{\frac{1}{2}}
\rho_{x_i,t_i}^{-\frac{1}{2}}(z_i) K_i +
t_i^{-\frac{1}{2}} + \rho_{x_i,t_i}(z_i)^{-2}}\\
\leq&\ \gl_i \frac{(1 + 2 r_i^{\frac{1}{2}}
\rho_{x_i,t_i}^{-\frac{1}{2}}(z_i)) K_i + 2 t_i^{-\frac{1}{2}} + 16
\rho_{x_i,t_i}(z_i)^{-2}}{(1 + r_i^{\frac{1}{2}}
\rho_{x_i,t_i}^{-\frac{1}{2}}(z_i)) K_i
+ t_i^{-\frac{1}{2}} + \rho_{x_i,t_i}(z_i)^{-2}}\\
\leq&\ 16 \gl_i.
\end{split}
\end{gather}
Next we claim that for $i$ chosen sufficiently large, we can ensure that for all
$s
\geq t_i - \gl_i^{-2}$ one has
\begin{align} \label{distbnd10}
d(x_i,z_i,s) \leq d(x_i,z_i,t_i) + \frac{1}{16} \rho_{x_i,t_i}(z_i).
\end{align}

\begin{figure} \label{fig1}
\begin{tikzpicture}[scale=0.8]

\draw  (2.7401,1.1791) ellipse (1.3357 and 5.6942);
\draw  (-10.4717,1.1239) ellipse (1.3357 and 5.6942);

\draw  [fill=white,color=white] (-8.9642,7.0207) rectangle (-10.4088,-4.9921);

\draw [dashed] (-10.4717,1.1239) ellipse (1.3357 and 5.6942);
\draw  plot[smooth, tension=.7] coordinates {(-10.4639,6.815) (-7.3942,6.3563)
(-4.2307,6.6844) (-0.7392,6.1689) (2.6585,6.8484)};

\draw  plot[smooth, tension=.7] coordinates {(-10.4405,-4.56) (-7.6754,-4.0245)
(-3.9964,-4.3994) (-0.7099,-4.0245) (2.682,-4.4984)};
\draw  [fill=black](2.7801,1.1791) circle (0.0537);
\node (v1) at (2.7623,1.2011) {};
\node (v2) at (2.7623,6.9015) {};
\draw [<->] (v1) -- (v2);
\node at (3.1213,3.849) {$r$};
\node at (2.9456,0.7383) {$(x,r^4)$};

\draw  [fill=black](-4.0667,3.7084) node (v3) {} circle (0.0537);

\draw  (v3) ellipse (0.4921 and 1.2708);
\draw  (-6.8787,3.6557) ellipse (0.4921 and 1.6708);
\draw [color=white,fill=white] (-6.8787,1.822) rectangle (-6.2225,5.5362);
\draw  [dashed] (-6.8787,3.6557) ellipse (0.4921 and 1.6708);
\draw  plot[smooth, tension=.7] coordinates {(-6.8505,1.979) (-6.1827,2.2532)
(-5.3391,2.2391) (-4.7275,2.4079) (-4.0667,2.43)};
\draw  plot[smooth, tension=.7] coordinates {(-6.8435,5.328) (-6.1686,5.1003)
(-5.4867,5.1214) (-4.7346,4.8824) (-4.0808,4.9838)};

\draw  [fill=black](-4.0467,1.1791) node(v4) {} circle (0.0537);

\draw  [color=red] plot[smooth, tension=.7] coordinates {(-8.2261,5.3838)
(-7.5465,5.4776)(-6.8318,5.3253) (-6.1171,5.5479) (-5.0275,5.4073)
(-4.0784,5.7002) (-3.0825,5.4073)};
\node at (-8.3901,5.8349) {$\textcolor{red}{\rho_{x,s} = \frac{15}{16}
\rho_{x,t}(z)}$};
\draw [<->] (v3) -- (v4);
\node (v5) at (-4.0667,4.9855) {};
\draw [<->,blue] (v3) -- (v5);

\node at (-3.3872,0.8964) {$(x,t)$};
\node (v6) at (-2.3795,3.31) {$(z,t)$};
\draw[<-,dotted] (v3) -- (v6);

\node (v7) at (-1.5218,4.5051) {$\textcolor{blue}{\lambda \leq \epsilon
\rho_{x,t}(z)}$};
\node (v8) at (-4.0667,4.3411) {};
\draw[<-,dotted] (v8) -- (v7);
\node (v9) at (-4.0667,1.9978) {};
\node (v10) at (-2.0733,1.9041) {$d(x,z,t)$};
\draw[<-,dotted] (v9) -- (v10);
\node (v12) at (-6.8318,5.3253) {};
\node (v13) at (-5.9882,7.7623) {Rule out this intersection};
\draw [->,dotted] (v13) -- (v12);
\end{tikzpicture}
\caption{A parabolic neighborhood around the blowup point}
\end{figure}
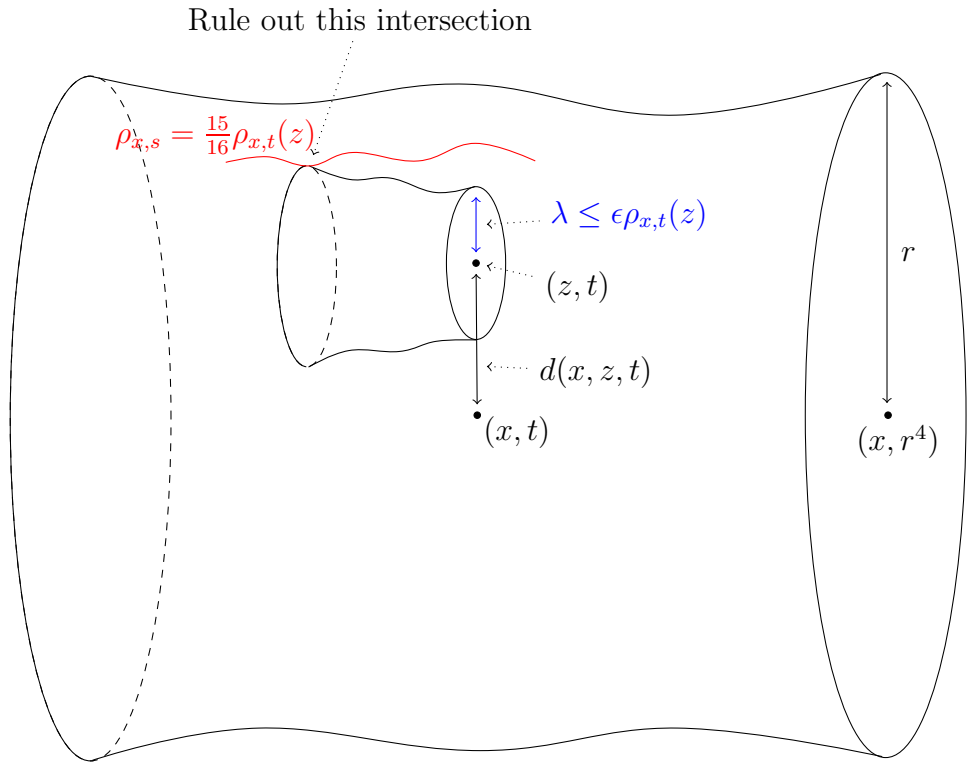
\noindent For such $i$ that (\ref{bul1}) and (\ref{distbnd10}) hold, observe
that for any point
$(y,s) \in
P_{\gl_i^{-\frac{1}{2}}}(z_i,t_i)$, one has that
\begin{align*}
d(x_i,y,s) \leq&\ d(x_i,z_i,s) + d(z_i,y,s)\\
\leq&\ d(x_i,z_i,t_i) + \frac{1}{16} \rho_{x_i,t_i}(z_i) +
\gl_i^{-\frac{1}{2}}\\
\leq&\ d(x_i,z_i,t_i) + \frac{1}{8} \rho_{x_i,t_i}(z_i)\\
=&\ \frac{1}{8} r_i + \frac{7}{8} d(x_i,z_i,t_i)\\
\leq&\ \frac{3}{4} r_i + \frac{1}{4} d(x_i,z_i,t_i).
\end{align*}
Thus by the above discussion, estimate (\ref{supbnd3}) applies to finish the
proof.  

It remains to show (\ref{distbnd10}).  We recall the basic estimate for
a one-parameter family of metrics,
\begin{align*}
g_t \leq \exp \left[ \int_s^t \brs{\frac{\del g}{\del t}}_{g_p} dp \right] g_s.
\end{align*}
Using the uniform estimate on $f_2$ we obtain that for all $s \geq t_i -
\gl_i^{-2}$ one has the pointwise estimate
\begin{align*}
 g_s \leq \exp \left[ \int_{t_i - \gl_i^{-2}}^{t_i} C K_i^2 \right] g_{t_i}.
\end{align*}
for all points $y$ such that $\{y\} \times [s,t] \subset P_{r_i}(x_i,r_i^4)$. 
Let $\gg : [0,d(x_i,z_i,t_i)] \to M$ denote a unit-speed minimizing geodesic in
the metric $g_{t_i}$ connecting $x_i$ to $z_i$.  For all times $s \leq t_i$ such
that the length of $\gg$ is bounded above by $r_i$, we can ensure that all
points along the curve do indeed lie in $P_{r_i}(x_i,r_i^4)$, and so the
pointwise estimate of $f_2$ applies at such points.  Arguing implicitly with
this continuity method we obtain for $s \geq t_i - \gl_i^{-2}$,
\begin{align*}
 d(x_i,z_i,s) \leq&\ d(x_i,z_i,t) e^{C K_i^2 \gl_i^{-2}}.
\end{align*}
Now using this constant $C$ as the definition for the third inequality of
(\ref{lambnd2}), we have $C K_i^2 \gl_i^{-2} \leq 1$ and hence 
$e^{C K_i ^2
\gl_i^{-2}} \leq 1 + 2 C K_i^2
\gl_i^{-2}$.  Using this inequality and the final inequality of (\ref{lambnd2})
we obtain
\begin{align*}
 d(x_i,z_i,s) \leq&\ d(x_i,z_i,t_i) \left[1 + 2 C K_i^2 \gl_i^{-2} \right]\\
 \leq&\ d(x_i,z_i,t_i) + 2 C r_i K_i^2\gl_i^{-2}\\
 \leq&\ d(x_i,z_i,t_i) + \frac{1}{16} \rho_{x_i,t_i}(z_i).
\end{align*}
This finishes the proof that (\ref{distbnd10}) holds for sufficiently large $i$,
finishing the proof.
\end{proof}
\end{lemma}

We now construct a blowup sequence around these points.  In particular, let
\begin{align*}
\bar{g}^i_t(x) := \gl_i g^i \left(x, t_i + \frac{t}{\gl_i^2} \right).
\end{align*}
Let us make some observations about the family of solutions
$\bar{g}^i_t$.  First, for each $i$, by construction we certainly have
\begin{align} \label{blowup30}
\bar{f}_m(z_i,0) = 1.
\end{align}
Next, observe that the solution $(M^i, \bar{g}^i_t)$ exists by construction on
$(-\gl_i^2 t_i, 0]$.  But also
\begin{align*}
\lim_{i \to \infty} \gl_i t_i^{\frac{1}{2}} \geq \lim_{i \to \infty} C_i =
\infty.
\end{align*}
Thus for sufficiently large $i$ the solution exists on $[-1,0]$.  Thirdly, by
Lemma \ref{blowuplemma} we have a uniform estimate on $\bar{f}_m$ on
$\bar{P}_{1}(z_i,0)$.  By Lemma \ref{parballcomp} this implies a uniform bound
for $f_m$ on a product neighborhood $B_{c}(z_i,0) \times [-c^4,0]$.  Theorem
\ref{locacomp} implies that there exists a subsequence of the pointed spaces
$\left\{  \left( B_c(z_i,0) \times [-c^4,0], z_i \right) \right\}$ converging in
the sense of local $C^{\infty}$-submersions to a limiting space
$\{\bar{B}_{\bar{c}}(z_{\infty},\bar{c}), \bar{g}^{\infty}_t, z_{\infty}\}$. 
Observe that, by (\ref{blowup30}), one has that
$\bar{f}_m(z_{\infty},0) = 1$.  On the other hand, observe that
\begin{align*}
\sup \bar{f}_2 \leq&\ \lim_{i \to \infty} \gl_i^{-1} K_i \leq \lim_{i \to
\infty}
C_i^{-1} = 0.
\end{align*}
Thus $\bar{f}_2 \equiv 0$, and so $\bar{g}_{\infty}$ is flat, contradicting that
$\bar{f}_m(z_{\infty},0) = 1$, and finishing the proof.
\end{proof}
\end{thm}

\begin{cor} \label{globaldoubling} Let $(M^n, g_t)$ be a complete solution to
$L^2$ flow which exists on a maximal time interval $[0,T)$.  Given $\ga > 0$,
there exists a constant $C = C(\ga,n) > 0$ so that if $\tau \in [0,T)$
satisfies 
\begin{align*}
K := \sup_{M \times [0,\tau]} \brs{\Rm} \geq \frac{\ga}{\tau^{\frac{1}{2}}},
\end{align*}
then the solution exists smoothly on $[0,\tau + \frac{C}{K^2}]$, and moreover
\begin{align*}
\sup_{M \times [0,\tau + \frac{C}{K^2}]} \leq 2 K.
\end{align*}
\begin{proof} By hypothesis the flow exists smoothly on $[\tau -
\frac{\ga^2}{K^2},\tau]$ with a global curvature bound of $K$.  Let $T_0 = \sup
\left\{ t \geq \tau | \sup_{M \times [0,t]} \brs{\Rm} \leq 2 K \right\}$.  It
suffices to show a uniform lower bound for $T_0$ since the curvature bound
implies smooth existence of the flow by \cite{SL2LTB} Corollary 1.9.  Observe
that for all $t \leq T_0$ the flow exists smoothly on $[t -
\frac{\ga^2}{K^2},t]$ with a global curvature bound of $2 K$.  By \cite{SL2LTB}
Theorem 1.3 it follows that there exists a constant $C = C(n,\ga)$ such that
\begin{align*}
\sup_{M \times \{t\}} f_4 \leq C K.
\end{align*}
A direct calculation using \cite{SL21} Proposition 4.2 shows that for a solution
to $L^2$ flow one has the pointwise ODE
\begin{align*}
\dt \brs{\Rm} \leq C f_4^3,
\end{align*}
for a universal constant $C$.  Thus on $[0,T_0]$ we have $\dt \brs{\Rm}
\leq C K^3$, and then it follows directly that $T_0 \geq \tau + \frac{C}{K^2}$
for a constant $C$ depending only on $n$ and $\ga$, as required.
\end{proof}
\end{cor}

\begin{cor} \label{typeI} Let $(M^n, g_t)$ be a complete solution to $L^2$ flow
which exists on a maximal finite time interval $[0,T)$.  There exists a constant
$C = C(n) > 0$ such that
\begin{align*}
\limsup_{t \to T} (T-t)^{\frac{1}{2}} \brs{\Rm}_{C^0(g_t)} \geq C.
\end{align*}
\begin{proof} Given $n \in \mathbb N$, let $C = C(1,n)$ be the constant from
Corollary \ref{globaldoubling}.  Suppose $(M^n, g_t)$ is a complete solution to
$L^2$ flow as in the statement, but
\begin{align} \label{typeI10}
\limsup_{t \to T} (T-t)^{\frac{1}{2}} \brs{\Rm}_{C^0(g_t)} < C^{\frac{1}{2}}.
\end{align}
Since $T$ is the maximal existence time, it follows from \cite{SL2LTB} Corollary
1.9 that $\displaystyle{\limsup_{t \to T}} \brs{\Rm} = \infty$, thus we may
choose a sequence of times $t_i \to T$ such that
\begin{align*}
\sup_{M \times [0,t_i]} \brs{\Rm} \geq \frac{1}{t_i^{\frac{1}{2}}}.
\end{align*}
Also, by (\ref{typeI10}) we know that
\begin{align*}
\gl_i := \sup_{M \times[0,t_i]} \brs{\Rm} < C^{\frac{1}{2}} \left(T - t_i
\right)^{-\frac{1}{2}}.
\end{align*}
It follows from Corollary \ref{globaldoubling} that the solution $(M^n, g_t)$
exists smoothly on $\left[0,t_i + \frac{C}{\gl_i^2} \right]$.  However, by
construction we have
\begin{align*}
t_i + \frac{C}{\gl_i^2} > t_i + C \left( \frac{(T -
t_i)^{\frac{1}{2}}}{C^{\frac{1}{2}}} \right)^2 = T,
\end{align*}
contradicting maximality of $T$.
\end{proof}
\end{cor}

\section{\texorpdfstring{$\ge$}{Epsilon}-regularity for critical metrics}
\label{ersec}

In this section we prove a local $\ge$-regularity theorem for critical metrics. 
Two direct precursors to this result appear in \cite{CW} and \cite{TV}. 
Specifically, \cite{TV} Theorem 3.1 asserts scale-invariant pointwise estimates
on curvature and its derivatives for metrics satisfying the equation
\begin{align} \label{er10}
 \gD \Rc = \Rm * \Rc.
\end{align}
This system of equations holds automatically for critical metrics on compact
manifolds.  However, in our setting we obtain only \emph{local} limits, and so
we do not automatically obtain that our limits have constant scalar curvature,
and thus we only have the coupled system of curvature equations given in Lemma
\ref{criteqns}.  This is similar in some ways to the extra terms which arise in
showing $\ge$-regularity for extremal K\"ahler metrics as in \cite{CW}.  Roughly
speaking, the method in \cite{TV} is to couple (\ref{er10}) to the general
equation satisfied by any Riemannian metric,
\begin{align} \label{er20}
 \gD \Rm = L(\N^2 \Rc) + \Rm * \Rm,
\end{align}
where $L$ denotes a universal linear expression.  This system of equations is
used in a two-step iteration process, first obtaining estimates on Ricci
curvature then exploiting this to control the term $L(\N^2\Rc)$ and obtain
bounds on the Riemann curvature tensor.  Our method is similar, in that we
consider the coupled system of three equations given by (\ref{er20}),
(\ref{crit1}) and (\ref{crit2}), and then employ a three-step iteration
process, first using (\ref{crit2}) to obtain estimates on scalar curvature,
which are used to control the $\N^2 s$ term in (\ref{crit1}) to obtain estimates
on Ricci curvature, which are in turn used in (\ref{er20}) to obtain estimates
on the Riemann curvature.

Moreover, due to the lack of a priori Ricci curvature bounds, in the course of
the corresponding arguments in \cite{TV} a preliminary Ricci curvature bound
must be
obtained, which is essential in controlling volume growth using Bishop-Gromov
volume comparison, and in turn obtaining the higher order estimates.  In our
setting this is not possible, and so instead we first obtain an $L^p$ bound on
Ricci curvature which allows the application of volume growth estimates due to
Petersen-Wei \cite{PetersenWei}.

\begin{lemma} \label{criteqns} (\cite{Besse} pg. 134) Let $(M^4, g)$ be a smooth
manifold with critical Riemannian metric.  The curvature of $g$ satisfy the
system of equations
\begin{align} 
\label{crit1}
  \gD \Rc =&\ \frac{1}{2} \N^2 s + \Rc * \Rm,\\ \label{crit2}
 \gD s =&\ 0.
\end{align} 
\end{lemma}

\begin{thm} \label{er} Given $(M^4, g)$ a solution to (\ref{crit1}),
(\ref{crit2}), and let
$r < \frac{\diam(M)}{2}$.  Given $k \geq 0$, there exists $\ge,C_k$ depending on
$C_S$ so that if
\begin{align*}
 \nm{\Rm}{L^2(B_r(p))} \leq \ge,
\end{align*}
 then
 \begin{align*}
  \sup_{B_{\frac{r}{2}}(p)} \brs{\N^k \Rm} \leq \frac{C_k}{r^{2+k}} \left[
\int_{B_r(p)} \brs{\Rm}^2  \right]^{\frac{1}{2}}.
 \end{align*}
\end{thm}

To prepare for the proof of this theorem we need a series of preliminary
estimates given in the lemmas below.  One observes that, by scaling, it suffices
to prove the theorem on balls of unit size (i.e. $r=1$), a reduction which will
be used in all the proofs below without further comment. In particular, the
statements will refer to balls of arbitrary size but will only be proved for
balls of unit size.

Also, we will state each intermediary estimate in the form of estimating a
higher order term on a ball of radius $\frac{1}{2}$ in terms of the energy on a
unit ball.  However, the overall proof is inductive and requires repeated
applications of prior estimates.  Thus strictly speaking the $n$-th ``layer" of
the proof should yield estimates on a ball of radius $\frac{1}{2^n}$ in terms of
the total energy.  This point is ignored in the proofs below in the interest of
a cleaner presentation, as it has no effect on the form of the final estimates. 
Lastly, for notational simplicity we will drop the volume form terms $dV_g$ in
all the integrals below.

\begin{lemma} \label{erlem1} There exists $C > 0$ so that
 \begin{align}
\left[\int_{B_{\frac{r}{2}}(p)} s^4  \right]^{\frac{1}{2}} \leq&\
\frac{C}{r^2} \int_{B_r(p)} s^2,\\
 \int_{B_{\frac{r}{2}}(p)} \brs{\N s}^2  \leq&\ \frac{C}{r^2}
\int_{B_r(p)} s^2,\\
\int_{B_{\frac{r}{2}}(p)} \brs{\N^2 s}^2  \leq&\ \frac{C}{r^4}
\int_{B_r(p)} s^2.
\end{align}
\begin{proof} Let $\phi$ denote a smooth function supported in $B_1(p)$.  We
integrate by parts to yield
\begin{align*}
 \int_{B_1(p)} \phi^2 \brs{\N s}^2 \leq \int_M \phi \brs{\N \phi} \brs{\N s}
\brs{s} \leq \gd \int_M \phi^2 \brs{\N s}^2 + \gd^{-1} \brs{\N \phi}^2 s^2.
\end{align*}
If $\phi$ is chosen to be a cutoff function for $B_{\frac{1}{2}}(p)$, the first
inequality follows.  For the second we again set $\phi$ to be arbitrary and
estimate
\begin{align*}
 \int_{B_1(p)} \phi^2 \brs{\N^2 s}^2 =&\ \int_{B_1(p)} \phi \N \phi * \N s *
\N^2 s + \phi^2 \N s * \gD \N s\\
\leq&\ \int_{B_1(p)} \gd \phi^2 \brs{\N^2 s}^2 + \gd^{-1} \brs{\N \phi}^2
\brs{\N s}^2 + \phi^2 \brs{\Rc} \brs{\N s}^2.
\end{align*}
Now estimate using the Sobolev inequality
\begin{align*}
 \int_M \phi^2 \brs{\Rc} \brs{\N s}^2 \leq&\ \nm{\Rc}{L^2(B_1(p))} \nm{\phi
\brs{\N s}}{L^4}^2\\
 \leq&\ \ge C_S \left[ \int_{B_1(p)} \brs{\N (\phi \brs{\N s})}^2 + \phi^2
\brs{\N s}^2 \right]\\
 \leq&\ \ge C_S \left[ \int_{B_1(p)} \phi^2 \brs{\N^2 s}^2 + (\phi^2 + \brs{\N
\phi}^2) \brs{\N s}^2 \right].
\end{align*}
Combining the two estimates above, choosing $\gd$ small and choosing $\ge$
sufficiently small with respect to $C_S$ yields
\begin{align*}
 \int_{B_1(p)} \phi^2 \brs{\N^2 s}^2 \leq&\  C \int_{B_1(p)} (\phi^2 + \brs{\N
\phi}^2) \brs{\N s}^2.
\end{align*}

\end{proof}
\end{lemma}

\begin{lemma} \label{erlem2} There exists a constant $C$ so that
\begin{align} \label{erlem21}
\left[ \int_{B_{\frac{r}{4}}(p)} \brs{\Rc}^4 \right]^{\frac{1}{2}} \leq
\frac{C}{r^2} \int_{B_r(p)} \brs{\Rc}^2\\ \label{erlem22}
\int_{B_{\frac{r}{8}}(p)} \brs{\N \Rc}^2 \leq \frac{C}{r^2} \int_{B_r(p)}
\brs{\Rc}^2.
\end{align}
\begin{proof} Using the Sobolev inequality and the Kato inequality $\brs{\N
\brs{\Rc}} \leq \brs{\N \Rc}$ we obtain, for $\phi$ a smooth function,
\begin{align*}
\left[\int_{B_1(p)} \left( \phi \brs{\Rc}^4 \right) \right]^{\frac{1}{2}} \leq&\
C_S \int_{B_1(p)} \brs{\N (\phi \brs{\Rc})}^2 \leq C_S \int_{B_1(p)} \left[
\brs{\N \phi}^2 \brs{\Rc}^2 + \phi^2 \brs{\N
\Rc}^2 \right].
\end{align*}
Integrating by parts yields
\begin{gather} \label{erlem210}
\begin{split}
\int_{B_1(p)} \phi^2 \brs{\N \Rc}^2 =&\ \int_{B_1(p)} \phi^2 \IP{\Rc, - \gD \Rc}
+ \phi \N \phi * \Rc * \N \Rc\\
=&\ \int_{B_1(p)} \phi^2 \IP{\Rc, -\frac{1}{2} \N^2 s + \Rc * \Rm} + \phi \N
\phi * \Rc * \N \Rc\\
\leq&\ C \nm{\phi \brs{\Rc}}{L^2} \nm{\phi \brs{\N^2 s}}{L^2}\\
&\ + C \int_{B_1(p)} \left[ \phi^2 \brs{\Rc}^2 \brs{\Rm} + \gd \phi^2 \brs{\N
\Rc}^2 + \gd^{-1} \brs{\N \phi}^2 \brs{\Rc}^2 \right]\\
\leq&\ C \nm{\phi \brs{\Rc}}{L^2} \nm{\phi \brs{\N^2 s}}{L^2} +
\left(\int_{B_1(p)} (\phi \brs{\Rc})^4 \right)^{\frac{1}{2}} \left(\int_{B_1(p)}
\brs{\Rm}^2 \right)^{\frac{1}{2}}\\
&\ + C \gd^{-1} \int_{B_1(p)} \brs{\N \phi}^2 \brs{\Rc}^2,
\end{split}
\end{gather}
where in the last line we choose $\gd$ sufficiently small with respect to
universal constants to absorb one term.  Now choosing $\phi$ to be a cutoff
function for $B_{\frac{1}{4}}(p)$ and applying Lemma \ref{erlem1} yields
\begin{align*}
\nm{\phi \brs{\Rc}}{L^2} \nm{\phi \brs{\N^2 s}}{L^2} \leq \nm{\phi
\brs{\Rc}}{L^2}^2.
\end{align*}
Combining the above three estimates and choosing $\ge$ sufficiently small with
respect to the Sobolev constant yields equation (\ref{erlem21}).  Now choosing
$\phi$ to be a cutoff function for $B_{\frac{1}{2}}(p)$ and applying
(\ref{erlem21}) in line (\ref{erlem210}) yields (\ref{erlem22}).
\end{proof}
\end{lemma}

\begin{lemma} \label{erlem3} There exists a constant $C$ such that
\begin{align} \label{erlem31}
\left[ \int_{B_{\frac{r}{2}}(p)} \brs{\Rm}^4 \right]^{\frac{1}{2}} \leq
\frac{C}{r^2} \int_{B_r(p)} \brs{\Rm}^2\\ \label{erlem32}
\int_{B_{\frac{r}{2}}(p)} \brs{\N \Rm}^2 \leq \frac{C}{r^2} \int_{B_r(p)}
\brs{\Rm}^2.
\end{align}
\begin{proof} Using the Sobolev inequality and the Kato inequality $\brs{\N
\brs{\Rm}} \leq \brs{\N \Rm}$ we obtain, for $\phi$ a smooth function,
\begin{gather} \label{erlem305}
\begin{split}
\left[\int_{B_1(p)} \left( \phi \brs{\Rm}^4 \right) \right]^{\frac{1}{2}} \leq&\
C_S \int_{B_1(p)} \brs{\N (\phi \brs{\Rm})}^2 + \phi^2 \brs{\Rm}^2\\
\leq&\ C_S \int_{B_1(p)} \left[ \phi^2 \brs{\N \Rm}^2 + (\phi^2 + \brs{\N
\phi}^2 ) \brs{\Rm}^2 \right].
\end{split}
\end{gather}
Now observe the general identity for a Riemannian metric,
\begin{align*}
\int \IP{ \gD \Rm, \phi^2 \Rm} =&\ \int \IP{\N^2 \Rc + \Rm * \Rm, \phi^2 \Rm}\\
=&\ \int \phi \N \phi * \N \Rc * \Rm + \phi^2 \N \Rc * \N \Rm + \phi^2 \Rm^{*3}.
\end{align*}
Integrating by parts and using this identity and (\ref{erlem22}) yields
\begin{gather} \label{erlem310}
\begin{split}
\int_{B_1(p)} \phi^2 \brs{\N \Rm}^2 =&\ \int_{B_1(p)} \phi^2 \IP{\Rm, - \gD \Rm}
+ \phi \N \phi * \Rm * \N \Rm\\
=&\ \int_{B_1(p)} \phi^2 \N \Rc * \N \Rm + \phi^2 \Rm^{*3} + \phi \N \phi * \N
\Rc * \Rm\\
\leq&\ C \int_{B_1(p)} \left[ \gd \phi^2 \brs{\N \Rm}^2 + \gd^{-1} \brs{\N
\Rc}^2 + \phi^2 \brs{\Rm}^3 + \brs{\N \phi}^2 \brs{\Rm}^2 \right]\\
\leq&\ C \int_{B_1(p)} \left[ \gd^{-1} \brs{\N \Rc} + \brs{\N \phi}^2
\brs{\Rm}^2 \right]\\
&\ + C \left[ \int_{B_1(p)} (\phi \brs{\Rm})^4 \right]^{\frac{1}{2}} \left[
\int_{B_1(p)} \brs{\Rm}^2 \right]^{\frac{1}{2}}\\
\leq&\ C \int_{B_1(p)} \brs{\Rm}^2 + C \ge \left[ \int_{B_1(p)} (\phi
\brs{\Rm})^4 \right]^{\frac{1}{2}}.
\end{split}
\end{gather}
Now choosing $\phi$ to be a cutoff function for $B_{\frac{1}{2}}(p)$ and
combining
this with (\ref{erlem305}) yields (\ref{erlem31}).  With this established one
obtains (\ref{erlem32}) by plugging (\ref{erlem31}) into (\ref{erlem310}).
\end{proof}
\end{lemma}

\begin{lemma} \label{ervol} There exist constants $\ge, C$ so that if
$\nm{\Rm}{L^2(B_r(p))} \leq \ge$ then
\begin{align*}
 \Vol(B_{\frac{r}{2}}(p)) \leq C r^4.
\end{align*}
\begin{proof} Rescaling the ball to unit size and applying Lemma \ref{erlem3} we
have a uniform $L^4$ bound on $\Rm$ on $B_{\frac{1}{2}}(p)$.  Applying
(\cite{PetersenWei} Theorem 1.1) we obtain a uniform upper bound on
$\Vol(B_{\frac{1}{2}}(p))$, and then scaling back yields the result.
\end{proof}
\end{lemma}

\begin{prop} \label{erind} For any $k \geq 0$ there exist constants $\ge,C$ so
that if
$\nm{\Rm}{L^2(B_r(p))} \leq \ge$, then
\begin{align}
\label{erind1} \left[ \int_{B_{\frac{r}{2}}(p)} \brs{\N^k s}^4
\right]^{\frac{1}{2}} \leq&\
\frac{C}{r^{2(k+1)}} \int_{B_r(p)} \brs{\Rm}^2,\\
\label{erind2} \left[ \int_{B_{\frac{r}{2}}(p)} \brs{\N^{k+1} s}^2 \right]
\leq&\
\frac{C}{r^{2(k+1)}} \int_{B_r(p)} \brs{\Rm}^2,\\
\label{erind3} \left[ \int_{B_{\frac{r}{2}}(p)} \brs{\N^k \Rc}^4
\right]^{\frac{1}{2}} \leq&\
\frac{C}{r^{2(k+1)}} \int_{B_r(p)} \brs{\Rm}^2,\\
\label{erind4} \left[ \int_{B_{\frac{r}{2}}(p)} \brs{\N^{k+1} \Rc}^2
\right] \leq&\ \frac{C}{r^{2(k+1)}} \int_{B_r(p)} \brs{\Rm}^2,\\
\label{erind5} \left[ \int_{B_{\frac{r}{2}}(p)} \brs{\N^k \Rm}^4
\right]^{\frac{1}{2}} \leq&\
\frac{C}{r^{2(k+1)}} \int_{B_r(p)} \brs{\Rm}^2,\\
\label{erind6} \left[ \int_{B_{\frac{r}{2}}(p)} \brs{\N^{k+1} \Rm}^2 \right]
\leq&\
\frac{C}{r^{2(k+1)}} \int_{B_r(p)} \brs{\Rm}^2.
\end{align}
\begin{proof} The proof is by induction.  We have already verified case $k=0$ of
each inequality in the above lemmas.  Now assume the lemma is true for all
$i=0,\dots,k-1$.  We proceed in three steps: scalar curvature, Ricci curvature,
Riemann curvature.
 
To obtain the scalar curvature estimates, first observe that since $\gD s = 0$
commutation of derivatives yields
 \begin{align} \label{scalarcommute}
  \gD (\N^k s) = \sum_{l=0}^{k-1} \N^l \Rm * \N^{k-l} s.
 \end{align}
Let $\phi$ denote a cutoff function for $B_{\frac{1}{2}}(p)$.  Using the Sobolev
inequality we
first observe
\begin{align*}
 \left[ \int_{B_1(p)} (\phi \brs{\N^k s})^4 \right]^{\frac{1}{2}} \leq&\ C
\int_{B_1(p)} \left[
\brs{\N (\phi \brs{\N^k s})}^2 + \phi^2 \brs{\N^k s}^2 \right]\\
 \leq&\ C \int_{B_1(p)} \left[ \phi^2 \brs{\N^{k+1} s}^2 + (\phi^2 + \brs{\N
\phi}^2)
\brs{\N^k s}^2 \right]
\end{align*}
Integrating by parts and applying (\ref{scalarcommute}) yields
\begin{align*}
 \int_{B_1(p)} \phi^2 \brs{\N^{k+1} s}^2 =&\ \int_{B_1(p)} \left[ \phi \N \phi *
\N^k s * \N^{k+1}
s - \phi^2 \IP{\N^k s, \gD \N^k s} \right]\\
 =&\ \int_{B_1(p)} \left[ \phi \N \phi * \N^k s * \N^{k+1} s + \phi^2 \N^k s *
\sum_{l=0}^{k-1} \N^l \Rm * \N^{k-l} s \right]\\
\leq&\ C \int_{B_1(p)} \left[ \gd \phi^2 \brs{\N^{k+1} s}^2 + \gd^{-1} \brs{\N
\phi}^2
\brs{\N^k s}^2 + \phi^2 \brs{\N^k s} \sum_{l=0}^{k-1} \brs{\N^l \Rm}
\brs{\N^{k-l} s} \right]\\
\leq&\ C \int_{B_1(p)} \left[ \gd^{-1} \brs{\N \phi}^2 \brs{\N^k s}^2 + \phi^2
\brs{\N^k s} \sum_{l=0}^{k-1} \brs{\N^l \Rm} \brs{\N^{k-l} s} \right],
\end{align*}
where the last line follows by choosing $\gd$ sufficiently small and absorbing
one term.  We next
estimate the terms in the sum above.  For $l=0$, we apply H\"older's inequality
and the Sobolev inequality to yield
\begin{align*}
 \int_{B_1(p)} \phi^2 \brs{\Rm} \brs{\N^k s}^2 \leq&\ \left[ \int_{B_1(p)}
\brs{\Rm}^2 \right]^{\frac{1}{2}} \left[\int_{B_1(p)}
(\phi \brs{\N^k s})^4\right]^{\frac{1}{2}}\\
 \leq&\ C \ge  \left[ \int_{B_1(p)} \brs{\N (\phi \brs{\N^k s})}^2 + \phi^2
\brs{\N^k
s}^2 \right]\\
 \leq&\ C \ge \left[ \int_{B_1(p)} \phi^2 \brs{\N^{k+1} s}^2 + (\phi^2 + \brs{\N
\phi}^2)
\brs{\N^k s}^2 \right].
\end{align*}
Next, applying H\"older's inequality and the inductive hypothesis we can
estimate, for any $1 \leq l \leq k-1$,
\begin{align*}
 \int_{B_1(p)} \phi^2 \brs{\N^{k} s} \brs{\N^{k-l} s} \brs{\N^l \Rm} \leq&\
\nm{\phi \N^k
s}{L^2} \nm{\phi^{\frac{1}{2}} \N^l \Rm}{L^4} \nm{\phi^{\frac{1}{2}}
\N^{k-l}s}{L^4}\\
\leq&\ C \left( \int_{B_1(p)} \brs{\Rm}^2 \right)^{\frac{1}{2}} \left(
\int_{B_1(p)} \brs{\Rm}^2
\right)^{\frac{1}{4}} \left( \int_{B_1(p)} \brs{\Rm}^2
\right)^{\frac{1}{4}}\\
=&\ C \int_{B_1(p)} \brs{\Rm}^2.
\end{align*}
Combining these estimates and choosing $\ge$ sufficiently small yields estimates
(\ref{erind1}) and (\ref{erind2}).

Next we prove (\ref{erind3}) and (\ref{erind4}).  To begin we observe that by
commuting derivatives and applying (\ref{crit1}) we obtain that
\begin{align} \label{riccicommute}
 \gD (\N^k \Rc) = \sum_{l=0}^k \N^l \Rm * \N^{k-l} \Rc + \N^{k+2} s.
\end{align}
Let $\phi$ again denote a cutoff function for $B_{\frac{1}{2}}(p)$.  Using the
Sobolev
inequality we first observe
\begin{align*}
  \left[ \int_{B_1(p)} (\phi \brs{\N^k \Rc})^4 \right]^{\frac{1}{2}} \leq&\ C
\int_{B_1(p)} \left[
\brs{\N (\phi \brs{\N^k \Rc})}^2 + \phi^2 \brs{\N^k \Rc}^2 \right]\\
 \leq&\ C \int_{B_1(p)} \left[ \phi^2 \brs{\N^{k+1} \Rc}^2 + (\phi^2 + \brs{\N
\phi}^2)
\brs{\N^k \Rc}^2 \right].
\end{align*}
Integrating by parts and applying (\ref{riccicommute}) and H\"older's inequality
yields
\begin{align*}
  \int_{B_1(p)} \phi^2 \brs{\N^{k+1} \Rc}^2 =&\ \int_{B_1(p)} \phi \N \phi *
\N^k \Rc * \N^{k+1}
\Rc - \phi^2 \IP{\N^k \Rc, \gD \N^k \Rc}\\
 =&\ \int_{B_1(p)} \phi \N \phi * \N^k \Rc * \N^{k+1} \Rc\\
 &\ \qquad + \phi^2 \N^k \Rc *
\sum_{l=0}^{k} \N^l \Rm * \N^{k-l} \Rc + \phi^2 \N^k \Rc * \N^{k+2} s\\
\leq&\ C \int_{B_1(p)} \gd \phi^2 \brs{\N^{k+1} \Rc}^2 + \gd^{-1} \brs{\N
\phi}^2
\brs{\N^k \Rc}^2 + \phi^2 \brs{\N^k \Rc} \sum_{l=0}^k \brs{\N^l \Rm}
\brs{\N^{k-l} \Rc}\\
&\ + \phi \N \phi \N^k \Rc * \N^{k+1} s + \phi^2 \N^{k+1} \Rc * \N^{k+1} s\\
\leq&\ C \int_{B_1(p)} \gd \phi^2 \brs{\N^{k+1} \Rc}^2 + \gd^{-1} \left( \brs{\N
\phi}^2
\brs{\N^k \Rc}^2 + \phi^2 \brs{\N^{k+1} s}^2 \right)\\
&\ + \phi^2 \brs{\N^k \Rc} \sum_{l=0}^k \brs{\N^l \Rm} \brs{\N^{k-l} \Rc}\\
\leq&\ C \int_{B_1(p)} \brs{\Rm}^2 + \phi^2 \brs{\N^k \Rc} \sum_{l=0}^k
\brs{\N^l \Rm} \brs{\N^{k-l} \Rc},
\end{align*}
where the last line follows by applying the inductive hypothesis as well as
estimate (\ref{erind2}) for the given $k$.  We now estimate the sum above. 
First, for $l=0$ we apply H\"older's inequality
and the Sobolev inequality to yield
\begin{align*}
 \int_{B_1(p)} \phi^2 \brs{\Rm} \brs{\N^k \Rc}^2 \leq&\ \left[ \int_{B_1(p)}
\brs{\Rm}^2 \right]^{\frac{1}{2}} \left[\int_{B_1(p)}
(\phi \brs{\N^k \Rc})^4\right]^{\frac{1}{2}}\\
 \leq&\ C \ge  \left[ \int_{B_1(p)} \brs{\N (\phi \brs{\N^k \Rc})}^2 + \phi^2
\brs{\N^k
\Rc}^2 \right]\\
 \leq&\ C \ge \left[ \int_{B_1(p)} \phi^2 \brs{\N^{k+1} \Rc}^2 + (\phi^2 +
\brs{\N
\phi}^2)
\brs{\N^k \Rc}^2 \right].
\end{align*}
Next, applying H\"older's inequality and the inductive hypothesis we can
estimate, for any $1 \leq l \leq k-1$,
\begin{align*}
 \int_{B_1(p)} \phi^2 \brs{\N^{k} \Rc} \brs{\N^{k-l} \Rc} \brs{\N^l \Rm} \leq&\
\nm{\phi
\N^k
\Rc}{L^2} \nm{\phi^{\frac{1}{2}} \N^l \Rm}{L^4}
\nm{\phi^{\frac{1}{2}}\N^{k-l}\Rc}{L^4}\\
\leq&\ C \left( \int_{B_1(p)} \brs{\Rm}^2 \right)^{\frac{1}{2}} \left(
\int_{B_1(p)} \brs{\Rm}^2
\right)^{\frac{1}{4}} \left( \int_{B_1(p)} \brs{\Rm}^2
\right)^{\frac{1}{4}}\\
=&\ C \int_{B_1(p)} \brs{\Rm}^2.
\end{align*}
Now consider the term $l=k$.  We may estimate
\begin{align*}
 \int_{B_1(p)} & \phi^2 \N^k \Rc * \N^k \Rm * \Rc\\
 =&\ \int_{B_1(p)} \phi * \N \phi * \N^k \Rc * \N^{k-1} \Rm * \Rc + \phi^2
\N^{k+1} \Rc *
\N^{k-1} \Rm * \Rc + \phi^2 \N^{k} \Rc * \N^{k-1} \Rm * \N \Rc.
\end{align*}
For the first term above we estimate using the inductive hypotheses
\begin{align*}
 \int_{B_1(p)} \phi * \N \phi * \N^k \Rc * \N^{k-1} \Rm * \Rc \leq&\ C
\nm{\phi^{\frac{1}{2}} \N^k \Rc}{L^2}
\nm{\phi^{\frac{1}{4}} \N^{k-1} \Rm}{L^4} \nm{\phi^{\frac{1}{4}}\Rc}{L^4}\\
 \leq&\ C \int_{B_1(p)} \brs{\Rm}^2.
\end{align*}
For the next term we estimate using the Cauchy-Schwarz inequality and H\"older's
inequality
\begin{align*}
 \int_{B_1(p)} \phi^2 \N^{k+1} \Rc * \N^{k-1} \Rm * \Rc \leq&\ C \int_{B_1(p)}
\phi^2 \left[ \gd
\brs{\N^{k+1} \Rc}^2 + \gd^{-1} \brs{\N^{k-1} \Rm}^2 \brs{\Rc}^2 \right]\\
 \leq&\ C \gd \int_{B_1(p)} \phi^2 \brs{\N^{k+1} \Rc}^2 + C \gd^{-1}
\nm{\phi^{\frac{1}{2}} \N^{k-1}
\Rm}{L^4}^2 \nm{\phi^{\frac{1}{2}} \Rc}{L^4}^2\\
 \leq&\ C \gd \int_{B_1(p)} \phi^2 \brs{\N^{k+1} \Rc}^2 + C \gd^{-1}
\left(\int_{B_1(p)}
\brs{\Rm}^2 \right)^2.
\end{align*}
Lastly we estimate, if $k=1$,
\begin{align*}
 \int_{B_1(p)} \phi^2 \N^k \Rc * \N^{k-1} \Rm * \N \Rc \leq&\ C \int_{B_1(p)}
\phi^2 \brs{\Rm}
\brs{\N \Rc}^2\\
 \leq&\ C \left(\int_{B_1(p)} \brs{\Rm}^2 \right)^{\frac{1}{2}} \left[
\int_{B_1(p)}
(\phi \brs{\N^k \Rc} )^4 \right]^{\frac{1}{2}}.
\end{align*}
And if $k \geq 2$, we have
\begin{align*}
 \int_{B_1(p)} \phi^2 \N^k \Rc * \N^{k-1} \Rm * \N \Rc \leq&\ C \int_{B_1(p)}
\phi^2 \left[
\brs{\N^{k-1} \Rm}^4 + \brs{\N \Rc}^4 + \brs{\N^k \Rc}^2 \right].
\end{align*}
Combining these estimates and choosing $\ge$ sufficiently small yields estimates
(\ref{erind3}) and (\ref{erind4}).

It remains to show (\ref{erind5}) and (\ref{erind6}).  To begin we observe that
by
commuting derivatives and applying (\ref{crit1}) we obtain that
\begin{align} \label{riemanncommute}
 \gD (\N^k \Rm) = \sum_{l=0}^k \N^l \Rm * \N^{k-l} \Rm + \N^{k+2} \Rc.
\end{align}
Let $\phi$ again denote a cutoff function for $B_{\frac{1}{2}}(p)$.  Using the
Sobolev
inequality we first observe
\begin{align*}
  \left[ \int_{B_1(p)} (\phi \brs{\N^k \Rm})^4 \right]^{\frac{1}{2}} \leq&\ C
\int_{B_1(p)} \left[
\brs{\N (\phi \brs{\N^k \Rm})}^2 + \phi^2 \brs{\N^k \Rm}^2 \right]\\
 \leq&\ C \int_{B_1(p)} \left[ \phi^2 \brs{\N^{k+1} \Rm}^2 + (\phi^2 + \brs{\N
\phi}^2)
\brs{\N^k \Rm}^2 \right].
\end{align*}
Integrating by parts and applying (\ref{riemanncommute}) and H\"older's
inequality
yields
\begin{align*}
  \int_{B_1(p)} \phi^2 \brs{\N^{k+1} \Rm}^2 =&\ \int_{B_1(p)} \phi \N \phi *
\N^k \Rm * \N^{k+1}
\Rm - \phi^2 \IP{\N^k \Rm, \gD \N^k Rm}\\
 =&\ \int_{B_1(p)} \phi \N \phi * \N^k \Rm * \N^{k+1} \Rm\\
 &\ \qquad + \phi^2 \N^k \Rm *
\sum_{l=0}^{k} \N^l \Rm * \N^{k-l} \Rm + \phi^2 \N^k \Rm * \N^{k+2} \Rc\\
\leq&\ C \int_{B_1(p)} \gd \phi^2 \brs{\N^{k+1} \Rm}^2 + \gd^{-1} \brs{\N
\phi}^2
\brs{\N^k \Rm}^2\\
&\ + \phi^2 \brs{\N^k \Rm} \sum_{l=0}^k \brs{\N^l \Rm}
\brs{\N^{k-l} \Rm}\\
&\ + \phi \N \phi \N^k \Rm * \N^{k+1} \Rc + \phi^2 \N^{k+1} \Rm * \N^{k+1} \Rc\\
\leq&\ C \int_{B_1(p)} \gd \phi^2 \brs{\N^{k+1} \Rm}^2 + \gd^{-1} \left( \brs{\N
\phi}^2
\brs{\N^k \Rm}^2 + \phi^2 \brs{\N^{k+1} \Rc}^2 \right)\\
&\ + \phi^2 \brs{\N^k \Rm} \sum_{l=0}^k \brs{\N^l \Rm} \brs{\N^{k-l} \Rm}\\
\leq&\ C \int_{B_1(p)} \brs{\Rm}^2 + \phi^2 \brs{\N^k \Rc} \sum_{l=0}^k
\brs{\N^l \Rm} \brs{\N^{k-l} \Rc},
\end{align*}
where the last line follows by applying the inductive hypothesis as well as
estimate (\ref{erind4}) for the given $k$.  We now estimate the sum above. 
First, for $l=0$ we apply H\"older's inequality
and the Sobolev inequality to yield
\begin{align*}
 \int_{B_1(p)} \phi^2 \brs{\Rm} \brs{\N^k \Rm}^2 \leq&\ \left[ \int_{B_1(p)}
\brs{\Rm}^2 \right]^{\frac{1}{2}} \left[\int_{B_1(p)}
(\phi \brs{\N^k \Rm})^4\right]^{\frac{1}{2}}\\
 \leq&\ C \ge  \left[ \int_{B_1(p)} \brs{\N (\phi \brs{\N^k \Rm})}^2 + \phi^2
\brs{\N^k
\Rm}^2 \right]\\
 \leq&\ C \ge \left[ \int_{B_1(p)} \phi^2 \brs{\N^{k+1} \Rm}^2 + (\phi^2 +
\brs{\N
\phi}^2)
\brs{\N^k \Rm}^2 \right].
\end{align*}
Next, applying H\"older's inequality and the inductive hypothesis we can
estimate, for any $1 \leq l \leq k-1$,
\begin{align*}
 \int_{B_1(p)} \phi^2 \brs{\N^{k} \Rm} \brs{\N^{k-l} \Rm} \brs{\N^l \Rm} \leq&\
\nm{\phi
\N^k
\Rm}{L^2} \nm{\phi^{\frac{1}{2}} \N^l \Rm}{L^4} \nm{\phi^{\frac{1}{2}}
\N^{k-l}\Rm}{L^4}\\
\leq&\ C \int_{B_1(p)} \brs{\Rm}^2.
\end{align*}
Now consider the term $l=k$.  We may estimate
\begin{align*}
 \int_{B_1(p)} & \phi^2 \N^k \Rm * \N^k \Rm * \Rm\\
 =&\ \int_{B_1(p)} \phi * \N \phi * \N^k \Rm * \N^{k-1} \Rm * \Rm + \phi^2
\N^{k+1} \Rm *
\N^{k-1} \Rm * \Rm\\
&\ \quad + \phi^2 \N^{k} \Rm * \N^{k-1} \Rm * \N \Rm.
\end{align*}
For the first term above we estimate using the inductive hypotheses
\begin{align*}
 \int_{B_1(p)} \phi * \N \phi * \N^k \Rm * \N^{k-1} \Rm * \Rm \leq&\ C
\nm{\phi^{\frac{1}{2}} \N^k \Rm}{L^2}
\nm{\phi^{\frac{1}{4}} \N^{k-1} \Rm}{L^4} \nm{\phi^{\frac{1}{4}} \Rm}{L^4}\\
\leq&\ C \left( \int_{B_1(p)} \brs{\Rm}^2 \right)^{\frac{1}{2}} \left(
\int_{B_1(p)} \brs{\Rm}^2
\right)^{\frac{1}{4}} \left( \int_{B_1(p)} \brs{\Rm}^2
\right)^{\frac{1}{4}}\\
=&\ C \int_{B_1(p)} \brs{\Rm}^2.
\end{align*}
For the next term we estimate using the Cauchy-Schwarz inequality and H\"older's
inequality
\begin{align*}
 \int_{B_1(p)} \phi^2 \N^{k+1} \Rm * \N^{k-1} \Rm * \Rm \leq&\ C \int_{B_1(p)}
\phi^2 \left[ \gd
\brs{\N^{k+1} \Rm}^2 + \gd^{-1} \brs{\N^{k-1} \Rm}^2 \brs{\Rm}^2 \right]\\
 \leq&\ C \gd \int_{B_1(p)} \phi^2 \brs{\N^{k+1} \Rm}^2 + C \gd^{-1}
\nm{\phi^{\frac{1}{2}} \N^{k-1}
\Rm}{L^4}^2 \nm{\phi^{\frac{1}{2}} \Rm}{L^4}^2\\
 \leq&\ C \gd \int_{B_1(p)} \phi^2 \brs{\N^{k+1} \Rm}^2 + C \gd^{-1}
\left(\int_{B_1(p)}
\brs{\Rm}^2 \right)^2.
\end{align*}
Lastly we estimate, if $k=1$,
\begin{align*}
 \int_{B_1(p)} \phi^2 \N^k \Rm * \N^{k-1} \Rm * \N \Rm \leq&\ C \int_{B_1(p)}
\phi^2 \brs{\Rm}
\brs{\N \Rm}^2\\
 \leq&\ C \left(\int_{B_1(p)} \brs{\Rm}^2 \right)^{\frac{1}{2}} \left[
\int_{B_1(p)}
(\phi \brs{\N^k \Rm} )^4 \right]^{\frac{1}{2}}.
\end{align*}
And if $k \geq 2$, we have
\begin{align*}
 \int_{B_1(p)} \phi^2 \N^k \Rm * \N^{k-1} \Rm * \N \Rm \leq&\ C \int_{B_1(p)}
\phi^2 \left[
\brs{\N^{k-1} \Rm}^4 + \brs{\N \Rm}^4 + \brs{\N^k \Rm}^2 \right].
\end{align*}
Combining these estimates and choosing $\ge$ sufficiently small yields estimates
(\ref{erind5}) and (\ref{erind6}).
\end{proof}
\end{prop}

 \begin{lemma} \label{ermoser} (\cite{TV} Lemma 3.9) Given $0 \leq u \in
L^2(B_r)$ and $f,h$ smooth nonnegative functions satisfying
 \begin{align*}
  \gD u \geq&\ - uf - h,\\
  \int_{B_r(p)} f^4  \leq&\ \frac{C_1}{r^4}.
 \end{align*}
If $\Vol(B_r(p)) \leq C_2 r^4$, then there exists a constant $C$ depending on
$C_1,C_2,C_S$ so that
\begin{align*}
 \sup_{B_{\frac{r}{2}}(p)} u \leq \frac{C}{r^2} \nm{u}{L^2(B_r(p))} + C r
\nm{h}{L^4(B_r(p))}.
\end{align*}
\end{lemma}

\begin{proof}[Proof of Theorem \ref{er}] With the estimates we have shown in
place, the proof is identical to what appears in \cite{TV}.  We include it here
for convenience.  Using (\ref{riemanncommute}) one obtains the inequality
\begin{align*}
 \gD \brs{\N^k \Rm} \geq&\ - C \sum_{l=0}^k \brs{\N^l \Rm} \brs{\N^{k-l} \Rm} -
C \brs{\N^{k+2} \Rm}\\
 =&\ - C \brs{\N^k \Rm} \brs{\Rm} - C \sum_{l=1}^{k-1} \brs{\N^l \Rm}
\brs{\N^{k-l}
\Rm} - C \brs{\N^{k+2} \Rc}.
\end{align*}
Now combining this inequality, the estimates of Proposition \ref{erind}, and
Lemma \ref{ervol} we may apply Lemma \ref{ermoser} inductively to obtain the
required pointwise estimates.
\end{proof}

\section{Thick-thin decomposition} \label{ttsec}
\subsection{A point-picking result} \label{ppsec}

A crucial issue in understanding local regions of curvature flows is to produce
points ``near" singular points which are at the centers of controlled parabolic
balls.  By exploiting the tubular neighborhood technique developed in the proof
of Theorem \ref{mainthm}, we establish such a point-picking result below in
Proposition \ref{ppp}.  There is a natural algorithm consisting of picking
points in parabolic neighborhoods where the curvature norm has doubled relative
to the center (cf. Lemma \ref{ppl}) which produces a natural candidate for a
point of ``large" curvature which is at the center of a controlled parabolic
ball.  To ensure the existence of this parabolic ball requires controlling the
distance decay of two points which roughly speaking are farther
apart than their respective curvature scales.  For the Ricci flow, this central
issue is overcome by exploiting the second
variation of length  (\cite{P1} Lemma 8.3(b) cf. \cite{KL} Lemma 27.8).  This
technique seems to not apply to fourth-order curvature flows, and so we apply
the tubular neighborhood technique.

\begin{lemma} \label{ppl} Let $(M^n, g_t)$ be a complete solution to the $L^2$
flow with bounded curvature on
$[0,K^{-2}]$, and suppose $(x,t) \in M$ satisfies $f_3(x,K^{-2}) = K$.  Given $A
\geq 1$ and $0 < \eta < 1$, there exists a point $(\bar{x},\bar{t})$ and a
constant $C = C(\eta)$
such
that $d(x,\bar{x},\bar{t}) \leq C A f_3(x,t)^{-\eta}$ and for all
$(x',t')$ satisfying
\begin{align} \label{maxprop}
 d(x,x',t') \leq&\ d(x,\bar{x},\bar{t}) + A
f_3(\bar{x},\bar{t})^{-\eta}, \qquad \bar{t} -
\frac{f_3(\bar{x},\bar{t})^{-2}}{2} \leq t' \leq \bar{t}
\end{align}
one has
\begin{align} \label{maxprop2}
 f_3(x',t') \leq 2 f_3(\bar{x},\bar{t}).
\end{align}
\begin{proof} Let $(x_1,t_1) = (x,t)$.  Given
$(x_k,t_k)$, if the required property fails for this point, then by hypothesis
there is a point $(x_{k+1},t_{k+1})$ such that 
\begin{align*}
t_k -  f_3(x_k,t_k)^{-2}<&\ t_{k+1} \leq t_k,\\
 d(x,x_{k+1},t_{k+1}) \leq&\ d(x,x_k,t_k) + A f_3(x_k,t_k)^{-\eta},\\
f_3(x_{k+1},t_{k+1}) >&\ 2 f_3(x_k,t_k).
\end{align*}
Observe that at each step, by induction we have
\begin{align*}
t_k - f_3(x_k,t_k)^{-2} >&\ t_{k-1} - \frac{f_3(x_{k-1},t_{k-1})^{-2}}{2} -
\frac{f_3(x_k,t_k)^{-2}}{2}\\
 >&\ \dots > t_1 - \frac{1}{2} \sum_{j=1}^{k} f_3(x_j,t_j)^{-2}\\
>&\ t_1 - \frac{f_3(x_1,t_1)^{-2}}{2} \sum_{j=1}^{\infty}  2^{-2(j-1)}\\
=&\ K^{-2} - \frac{2 K^{-2}}{3}\\
>&\ 0,
\end{align*}
thus the choices of $t_k$ remain in our time interval of definition.  Since the
solution is smooth this process must terminate after finitely many
steps.  Call this final point $(\bar{x},\bar{t}) = (x_k,t_k)$.  It remains to
show the distance claim.  Observe the estimate
\begin{align*}
 d(x,\bar{x},\bar{t}) =&\ d(x,x_{k},t_{k})\\
\leq&\ d(x,x_{k-1},t_{k-1}) + A f_3(x_{k-1},t_{k-1})^{-\eta}\\
\leq&\ \dots \leq d(x,x_1,t_1) + A \sum_{j=1}^{k-1} f_3^{-\eta}
(x_j,t_j)\\
\leq&\ A f_3(x_1,t_1)^{-\eta} \sum_{j=1}^{\infty} 2^{-\eta j}\\
\leq&\ C(\eta) A  f_3^{-\eta} (x_1,t_1).
\end{align*}
\end{proof}
\end{lemma}

\begin{rmk} The reason for stating the above lemma with an arbitrary choice of
$\eta$ is due to the next proposition, where we show that the point
$(\bar{x},\bar{t})$ constructed above has a parabolic neighborhood of a
controlled size, provided $\eta$ is chosen properly with respect to the
dimension of the manifold $n \leq 4$.  One would hope to be able to take $\eta =
\frac{1}{2}$ in every dimension, but we are not able to show this.  For the next
proposition we set
\begin{align} \label{etadef}
\eta := \begin{cases}
\frac{1}{4} & n = 3,\\
\frac{1}{10} & n = 4
\end{cases}
\end{align}
\end{rmk}

\begin{prop} \label{ppp} Assume the setup of Lemma \ref{ppl}, assume $n \leq 4$,
and let $\bar{K} =
f_3(\bar{x},\bar{t}) \geq 1$.  Fix some $\rho > 0$ and suppose that for all
$(x',t')$ such that (\ref{maxprop}) holds, one has
\begin{align} \label{pppnc}
\nu_{3}(x',t') \geq \mu.
\end{align}
There exists a constant $\gd = \gd(\mu, \FF(g_0),A) > 0$ so that one has
$\bar{\gs}_3(\bar{x},\bar{t}) \geq \gd \bar{K}^{-\frac{1}{2}}$.
\begin{proof} Let $\gg$ denote the minimizing geodesic connecting $x$ to
$\bar{x}$ in the metric $g_{\bar{t}}$.  Our goal is to show, by a continuity
method, that every point $p \in \gg$ satisfies $\bar{\gs}_3(\bar{x},\bar{t})
\geq
\gd \bar{K}^{-\frac{1}{2}}$ for appropriately chosen $\gd$.  Parameterize $\gg$
by unit speed, i.e. $\gg :
[0,d(x,\bar{x},\bar{t})] \to M$, and let
\begin{align*}
 \Omega := \left\{ s \in [0,d(x,\bar{x},\bar{t})]\ | \
\bar{\gs}_3(\gg(s),\bar{t})
> \gd \bar{K}^{-\frac{1}{2}} \right\}.
\end{align*}
First we show that $0 \in \Omega$.  Given $(x',t') \in P_{\gd
\bar{K}^{-\frac{1}{2}}}(x,\bar{t})$, then
\begin{align*}
 d(x,x',t') \leq \gd \bar{K}^{-\frac{1}{2}} \leq d(x,\bar{x},\bar{t}) + A
\bar{K}^{-\eta},
 \end{align*}
for $\gd \leq 1 \leq A$, since $\bar{K} \geq 1$.  Thus (\ref{maxprop}) holds,
and so by construction one has
$f_3(x',t') \leq 2 \bar{K}$, and the claim follows.  Next, observe that since
the defining inequality for $\Omega$ is strict and the one-parameter family of
metrics is smooth, it follows that $\Omega$ is open.

Lastly we show that $\Omega$ is closed.  We assume $[0,s) \subset \Omega$, and
show that this implies $s \in \Omega$.  This proceeds by an independent
continuity method.  Let $x' = \gg(s)$, and let
\begin{align*}
 \Xi = \left\{ t \in [\bar{t} - \gd^4 \bar{K}^{-2},\bar{t}] \ |\ \forall y \in
B_{\gd
\bar{K}^{-\frac{1}{2}}}(x',t), \ d(x,y,t) < d(x,\bar{x},\bar{t}) + A
\bar{K}^{-\eta} \right\}.
\end{align*}
Note that, if $\Xi = [\bar{t} - \gd^4 \bar{K}^{-2}, \bar{t}]$ then every point
in $P_{\gd \bar{K}^{-\frac{1}{2}}}(x',\bar{t})$ satisfies (\ref{maxprop}) and so
(\ref{maxprop2}) applies, finishing the claim that $s \in \Omega$, thus
finishing the proposition.
First we show that $\bar{t} \in \Xi$.  Note that since the geodesic $\gg$ is
minimizing in the metric $g_{\bar{t}}$ and $x'$ is a point on $\gg$ it follows
that $d(x,x',\bar{t}) \leq d(x,\bar{x},\bar{t})$.  Thus if $y \in B_{\gd
\bar{K}^{-\frac{1}{2}}}(x',\bar{t})$ then by the triangle
inequality one has that
\begin{align*}
 d(x,y,\bar{t}) \leq&\ d(x,x',\bar{t}) + d(x',y,\bar{t}) \leq
d(x,\bar{x},\bar{t}) + \gd \bar{K}^{-\frac{1}{2}} < d(x,\bar{x},\bar{t}) + A
\bar{K}^{-\eta}.
\end{align*}
Thus $\bar{t} \in \Xi$. As above, openness of $\Xi$ follows from openness of its
defining inequality.

Lastly we show that $\Xi$ is closed.  We assume that $(\bar{t} - \ge
\gd^4 \bar{K}^{-2},\bar{t}] \subset \Xi$ for some $\ge \leq 1$, and show
that $\bar{t} - \ge \gd^4 \bar{K}^{-2} \in \Xi$.  As the solution is smooth on
the interval $[\bar{t} - \ge \gd^4 \bar{K}^{-2}, \bar{t}]$, there exists some
$\gb > 0$ such that $\gb$-quasi geodesics in this time interval admit foliated
disc neighborhoods as in Lemma \ref{tubularfocal}.  With this choice of $\gb$
construct a $\gb$-quasi geodesic $\gg$ connecting $x$ to $x'$, restricted to the
time interval $[\bar{t} - \ge \gd^4 \bar{K}^{-2},\bar{t}]$, according to Lemma
\ref{almgeodlemma}.  This construction comes with a decomposition of the time
interval into finitely many subintervals $I_j := [a_j,a_{j+1}]$ which
partition $[\bar{t} - \ge \gd^4 \bar{K}^{-2}, \bar{t}]$.  Restrict attention to
one such interval, on which the curve $\gg$ is fixed in time.

Fix a real number $0 < R < \gd \bar{K}^{-\frac{1}{2}}$. Observe that, in the
context of the continuity method, (\ref{maxprop}) will hold at all points of
$D_R (\gg)$ for $R$ chosen as above.  In particular, for these points, by the
hypothesis
$\nu_3 \geq \mu$ we obtain by Lemma \ref{cheegers} a uniform lower bound on the
injectivity
radius of the form $\inj \geq c(\mu)\bar{K}^{-\frac{1}{2}}$.  If
$\gd$ is chosen sufficiently small with respect to this constant $c(\mu)$ it follows that Lemma
\ref{tubularfocal} applies to $D_R(\gg)$.

Fix a point $p \in \gg$ and a time $t' \in I_j$.
Given $q \in D_R(p)$, note that by the defining property for $\Xi$ we have that
(\ref{maxprop}) holds at $q$, and so we have $f_3(q,t') \leq 2 \bar{K}$.  This
provides us with a uniform estimate $\brs{\N \grad \FF} \leq C
\bar{K}^{\frac{5}{2}}$, which can be integrated along a geodesic to yield
\begin{align*}
\brs{\grad \FF}(p) \leq \brs{\grad \FF}(q) + C R \bar{K}^{\frac{5}{2}}.
\end{align*}
By arguing as in (\ref{pntest}) we can average over $D_R(p)$ to yield
\begin{align*}
\brs{\grad \FF}(p,t) \leq&\ \Area(D_R(p))^{-\frac{1}{2}} \left[ \int_{D_R(p)}
\brs{\grad \FF}^2(q) dA(q) \right]^{\frac{1}{2}} + C R \bar{K}^{\frac{5}{2}}.
\end{align*}
Now arguing similarly to (\ref{lengthestimate}) we yield
\begin{align*}
\frac{d}{dt} L(\gg) \leq C R^{\frac{1-n}{2}} \left[\int_M \brs{\grad \FF}^2
\right]^{\frac{1}{2}} L(\gg)^{\frac{1}{2}} + C R \bar{K}^{\frac{5}{2}} L(\gg).
\end{align*}
Observe that, by the construction of Lemma \ref{ppl}, since we assume $\bar{K}
\geq 1$ we have at least $L(\gg) \leq CA$, where the constant $C$ is universal. 
Thus, as long as $L(\gg) \leq CA + 1$ we can integrate this differential
inequality over our time interval $I_j = [a_j,a_{j+1}]$ to yield
\begin{align*}
d(x,x',t) \leq d(x,x',a_j) + C R^{\frac{1-n}{2}} \int_{a_j}^{a_{j+1}} \left[
\int_M \brs{\grad \FF}^2 dV_g \right]^{\frac{1}{2}} dt + C R
\bar{K}^{\frac{5}{2}}
\int_{a_j}^{a_{j+1}} dt .
\end{align*}
Thus after an iterative application of these estimates we obtain for any $t \in
[\bar{t} - \ge \gd^4 \bar{K}^{-2} ,\bar{t}]$ we have
\begin{align*}
d(x,x',t) \leq&\ d(x,x',\bar{t}) + C R^{\frac{1-n}{2}} \int_{\bar{t} - \ge \gd^4
\bar{K}^{-2}}^{\bar{t}} \left[ \int_M \brs{\grad \FF}^2 dV_g
\right]^{\frac{1}{2}} dt
+ C R \bar{K}^{\frac{5}{2}} \left[ \ge \gd^4 \bar{K}^{-2} \right]\\
\leq&\ d(x,\bar{x},\bar{t}) + C R^{\frac{1-n}{2}} \left[ \int_{\bar{t} - \ge
\gd^4 \bar{K}^{-2}}^{\bar{t}} dt \right]^{\frac{1}{2}} \left[ \int_{\bar{t} -
\ge
\gd^4 \bar{K}^{-2}}^{\bar{t}} \int_M \brs{\grad \FF}^2 dV_g dt
\right]^{\frac{1}{2}} + C \ge \gd^4 R
\bar{K}^{\frac{1}{2}}\\
\leq&\ d(x,\bar{x},\bar{t}) + C \ge^{\frac{1}{2}} \FF(g_0)^{\frac{1}{2}} \gd^2
R^{\frac{1-n}{2}} \bar{K}^{-1} + C \ge \gd^4 R \bar{K}^{\frac{1}{2}}\\
\leq&\ d(x,\bar{x},\bar{t}) + C \FF(g_0)^{\frac{1}{2}} \gd^2 R^{\frac{1-n}{2}}
\bar{K}^{-1} + C \gd^4 R \bar{K}^{\frac{1}{2}}\\
=&\ d(x,\bar{x},\bar{t}) + I + II.
\end{align*}
It remains to estimate the two terms $I$ and $II$ by an appropriate choice of
constants.  First restrict attention to the case $n=4$.  Set
$R = \gd \bar{K}^{-\frac{3}{5}}$.  Since $\bar{K} \geq 1$ this choice satisfies
$0 < R < \gd \bar{K}^{-\frac{1}{2}}$ as required by the construction.  Then, by
choosing $\gd$ appropriately small with respect to universal constants, $A$ and
$\FF(g_0)$, we obtain
\begin{align*}
I \leq&\ C \FF(g_0)^{\frac{1}{2}}  \gd^2 \left( \gd \bar{K}^{-\frac{3}{5}}
\right)^{-\frac{3}{2}} \bar{K}^{-1} = C \FF(g_0)^{\frac{1}{2}} \gd^{\frac{1}{2}}
\bar{K}^{-\frac{1}{10}} \leq \frac{A}{3} \bar{K}^{-\frac{1}{10}}.
\end{align*}
Similarly we estimate
\begin{align*}
II = C \gd^4 (\gd \bar{K}^{-\frac{3}{5}}) \bar{K}^{\frac{1}{2}} = C \gd^5
\bar{K}^{-\frac{1}{10}} \leq \frac{A}{3} \bar{K}^{-\frac{1}{10}}.
\end{align*}
Putting these estimates together and choosing further $\gd \leq \frac{A}{3}$ we
obtain for $t = \bar{t} - \ge \gd^4 \bar{K}^2$ and any $y \in B_{\gd
\bar{K}^{-\frac{1}{2}}}(x',t)$ the estimate
\begin{align*}
d(x,y,t) \leq d(x,x',t) + d(x',y,t) < d(x,\bar{x},\bar{t}) + \frac{2A}{3}
\bar{K}^{-\frac{1}{10}} + \frac{A}{3} \bar{K}^{-\frac{1}{2}} \leq
d(x,\bar{x},\bar{t}) + A \bar{K}^{-\frac{1}{10}},
\end{align*}
as required.

In the case $n=3$, we set $R = \gd \bar{K}^{-\frac{3}{4}} < \gd
\bar{K}^{-\frac{1}{2}}$.  Then by choosing $\gd$ appropriately small with
respect to universal constants, $A$ and $\FF(g_0)$, we obtain
\begin{align*}
I \leq&\ C \FF(g_0)^{\frac{1}{2}}  \gd^2 \left( \gd \bar{K}^{-\frac{3}{4}}
\right)^{-1} \bar{K}^{-1} = C \FF(g_0)^{\frac{1}{2}} \gd
\bar{K}^{-\frac{1}{4}} \leq \frac{A}{3} \bar{K}^{-\frac{1}{4}}.
\end{align*}
Similarly we estimate
\begin{align*}
II = C \gd^4 (\gd \bar{K}^{-\frac{3}{4}}) \bar{K}^{\frac{1}{2}} = C \gd^5
\bar{K}^{-\frac{1}{4}} \leq \frac{A}{3} \bar{K}^{-\frac{1}{4}}.
\end{align*}
Putting these estimates together and choosing further $\gd \leq \frac{A}{3}$ we
obtain for $t = \bar{t} - \ge \gd^4 \bar{K}^2$ and any $y \in B_{\gd
\bar{K}^{-\frac{1}{2}}}(x',t)$ the estimate
\begin{align*}
d(x,y,t) \leq d(x,x',t) + d(x',y,t) < d(x,\bar{x},\bar{t}) + \frac{2A}{3}
\bar{K}^{-\frac{1}{4}} + \frac{A}{3} \bar{K}^{-\frac{1}{2}} \leq
d(x,\bar{x},\bar{t}) + A \bar{K}^{-\frac{1}{4}},
\end{align*}
as required.
\end{proof}
\end{prop}

\subsection{Energy concentration in thick piece}

In this subsection we prove a result on the concentration of energy around
singular points of the flow, assuming a noncollapsing bound along the sequence. 
We begin with a useful definition.

\begin{defn}
 Let $(M^n, g_t)$ be a solution to the $L^2$
flow on $[0,T)$ such
that every time slice is complete.  Call a sequence $\{x_i,t_i\}$ in $M \times
[0,T)$ a \emph{$(\gd,\mu)$-thick blowup sequence} if
 \begin{gather} \label{thicklocalblowup}
  \begin{split}
 \lim_{i \to \infty} f_3(x_i,t_i) =&\ \infty\\
\bar{\gs}_3 (x_i,t_i) \geq&\ \gd f_3(x_i,t_i)^{-\frac{1}{2}}\\
 \liminf_{i \to \infty} \nu_3(x_i,t_i) \geq&\ \mu,   
  \end{split}
 \end{gather}
for some constants $\gd,\mu > 0$.
\end{defn}

\begin{prop} \label{thickblowup} Let $(M^n, g_t)$ be a solution to the $L^2$
flow on $[0,T)$ such
that every time slice is complete.
\begin{enumerate}
 \item If $n=3$, there does not exist a $(\gd,\mu)$-thick blowup sequence.
 \item If $n=4$, there is a constant $\ge = \ge(\gd,\mu)$ such that any
$(\gd,\mu)$-thick blowup sequence satisfies
 \begin{align} \label{thickblowupenergy}
  \liminf_{i \to \infty} \int_{B_{\frac{\bar{\gs}_2(x_i,t_i)}{2}}(x_i,t_i)}
\brs{\Rm}^2 dV >
\ge.
 \end{align}
 In particular, if $\FF(g_0) \leq \ge$ then there does not exist a
$(\gd,\mu)$-thick blowup sequence.
\end{enumerate}
\begin{proof}  Let $\gl_i := f_3(x_i,t_i)$, and consider the blowup sequence of
pointed solution to the $L^2$ flow given  by $\{(M, g^i_t,x_i) \}$, where
\begin{align*}
 g^i_t(x,t) := \gl_i g\left( x, t_i + \frac{t}{\gl_i^2} \right).
\end{align*}
By construction, for every $i$ one has $f_3^i(x_i,0) = 1$, and moreover by the
second equation of (\ref{thicklocalblowup}) there is a uniform bound (depending
on $\gd$) on $f_3^i$ on a parabolic ball of radius $\gd$ around $(x_i,0)$. 
Lastly, by the third equation of (\ref{thicklocalblowup}), in the rescaled
metric there is a uniform lower bound
\begin{align*}
 \Vol B_{\gd}(x_i,0) \geq \mu.
\end{align*}
Since $f_3$ has a uniform bound on this ball, in particular there is a uniform
curvature estimate on this ball, and it follows from Lemma \ref{cheegers} that
there is a constant $i(\gd,\mu)$ such that
\begin{align*}
 \inj_{g^i_0}(x_i) \geq i.
\end{align*}
By Theorem \ref{locacomp}, we obtain a local
limiting solution $(M_{\infty}, g^{\infty}_t,x_{\infty})$ to the $L^2$ flow,
which moreover satisfies.
\begin{align} \label{thicklimit}
 f_3(x_{\infty},0) = 1.
\end{align}

First suppose $n=3$.  Then one observes that
\begin{align*}
 \int_{B(x_{\infty},1)} \brs{\Rm}^2 \leq&\ \limsup_{i \to \infty} \FF(g^i(0))\\
 \leq&\ \limsup_{i \to \infty} \FF(\gl_i g(t_i))\\
 =&\ \limsup_{i \to \infty} \gl_i^{-\frac{1}{2}} \FF(g(t_i))\\
 \leq&\ \limsup_{i \to \infty} \gl_i^{-\frac{1}{2}} \FF(g(0))\\
 =&\ 0.
\end{align*}
This implies that $g_{\infty}$ is flat, contradicting (\ref{thicklimit}).

Now suppose $n=4$.  We show that the metric $g^{\infty}_t \equiv g^{\infty}_0$
is critical.  Using the scaling-invariance of $\FF$ in dimension $4$ we obtain
\begin{align*}
\int_{-1}^0 \int_{M_{\infty}} \brs{\grad \FF_{\infty}}^2 dV_{\infty} dt \leq&\
\lim_{j \to \infty} \int_{\frac{-1}{\gl_j^2} + t_j}^{t_j} \int_M \brs{\grad
\FF_j}^2 dV_j dt\\
=&\ \lim_{j \to \infty} \mathcal F\left(g \left(\frac{-1}{\gl_j^2} + t_j\right)
\right) - \mathcal F(g(t_j))\\
=&\ 0,
\end{align*}
where the last line follows since $\FF$ is monotonically decreasing and bounded
below.  Thus $g^{\infty}_0$ is critical.  Now note that, if
(\ref{thickblowupenergy}) was not satisfied by the original sequence, we could
choose an appropriate subsequence so that, in the limiting space, we have
\begin{align*}
\int_{B_{\frac{\gd}{2}}(x_{\infty},0)} \brs{\Rm}^2 dV \leq \ge.
\end{align*}
Lastly, observe that our uniform curvature and injectivity radius estimates for
the limit space imply that there is a uniform Sobolev constant bound on
$B_{\frac{3 \gd}{4}}(x_{\infty},0)$.  Thus, by Theorem \ref{er}, we conclude
that $f_3(x_{\infty},0) \leq C(\gd,\mu) \ge$.  For $\ge$ chosen sufficiently
small this violates
(\ref{thicklimit}), yielding a contradiction and finishing the proof.
\end{proof}
\end{prop}

\begin{rmk} Again note that such a phenomenon is certainly restricted to
dimension $n \leq 4$: choosing any point on the homogeneous shrinking sphere of
dimension $n\geq5$, and any sequence of times approaching the singular time,
such a sequence is $(\gd,\mu)$-thick, but on the other hand the total energy
goes to zero at the singular time.
\end{rmk}

Observe that, in the statement of Theorem \ref{thickthin1}, we do not claim an
upper bound on the number of thick blowup points which occur, as one might
expect since each one carries a definite amount of energy and the total energy
is finite.  This is due to the fact that we cannot ensure local stability of
energy along the flow as the curvature is blowing up.  In other words, our
construction yields some sequence of times approaching the singular time where
the energy is concentrated near a given point, but we cannot improve this to say
that actually \emph{every} time has concentrated energy as one approaches the
singularity.  Some further local estimates would be required to obtain this
statement.  Assuming one could obtain such estimates of course a bound on the
number of thick blowup regions would follow, as detailed in the next corollary.

\begin{cor} Let $(M^4, g_t)$ be a complete solution to the $L^2$ flow with
finite energy on $[0,T)$.  Given $\gd,\mu > 0$ there exists a constant $N =
N(\gd,\mu,\FF(g_0))$ such that the following holds:
 Given a sequence $t_i \to T$ and $(\gd,\mu)$-thick blowup sequences
$\{(x_i^j,t_i)_{i=1}^{\infty}\}_{j=1}^N$, there exists $\{k_1,k_2\} \in
\{1,\dots,N\}$ such that
 \begin{align} \label{singconv}
  \limsup_{i \to \infty} d_{t_i}(x_i^{k_1},x_i^{k_2}) = 0.
 \end{align}
 \begin{proof} Let $\ge(\gd,\mu)$ denote the constant from Proposition
\ref{thickblowup}, and let $N =\ge^{-1} \FF(g_0)$.  Suppose we had $N$
$(\gd,\mu)$-thick blowup sequences as in the statement, but that
(\ref{singconv}) did not hold.  Observe that, in the notation of Proposition
\ref{thickblowup}, the radii $\rho_i^j \to 0$ for each $j$ as $i \to \infty$. 
In particular, this means we can choose a subsequence of times $t_i$ (not
relabeled) so that sufficiently far along the sequence we have that the balls
$\left\{ B_{\frac{\rho^j_i}{2}}(x^j_i,t_i) \right\}$ are all disjoint.  Then the
estimate (\ref{thickblowupenergy}) implies that, for such times $t_i$
sufficiently close to $T$ we have
\begin{align*}
\FF(g_0) \geq \FF(g_{t_i}) \geq \sum_{j=1}^N \int_{B_{\frac{\rho^j_i}{2}}}
\brs{\Rm}^2 > N \ge = \FF(g_0),
\end{align*}
a contradiction.
\end{proof}
\end{cor}

\begin{rmk} It is reasonable to expect that singularities to the $L^2$ flow in
the thick piece of the manifold occur via the formation of orbifold points. 
With this in mind, one does not expect the topology around a blowup point, after
rescaling, to be trivial, as occurs in the case of Yang-Mills flow where
singular models correspond to Yang-Mills connections on $S^4$.  Thus it is
natural to phrase ``how many'' singularities have occurred in terms of whether
or
not singular areas maintain a positive distance as time approaches the singular 
time.
\end{rmk}

\subsection{Proofs of Theorems}

\begin{proof}[Proof of Theorem \ref{thickthin1}] Fix $\mu > 0$.  Let $(M^4,
g_t)$ be a solution to the $L^2$ flow as in the statement.  Suppose there exists
a point where the curvature blows up, and the second condition fails, i.e.
\begin{align*}
\liminf_{t \to T} d(x,T_{\mu}(g_t),t) > 0.
\end{align*}
Choose a sequence of times $\{t_i\} \to T$ such
that
\begin{align*}
 \lim_{i \to \infty} \brs{\Rm}(x,t_i) = \infty.
\end{align*}
For $i \in \mathbb N$ sufficiently large, we can verify the setup of Lemma
\ref{ppl} with initial point $(x,t_i)$ and $A = 1$ to obtain a new point
$(\bar{x}_i,\bar{t}_i)$.  Moreover, by (\ref{ttt10}) we see that for
sufficiently large $i$, condition (\ref{pppnc}) is verified, and so by Proposition \ref{ppp} there exists
a constant $\gd > 0$ so that
$\bar{\gs}_3(\bar{x}_i,\bar{t}_i) \geq \gd
f_3(\bar{x}_i,\bar{t}_i)^{-\frac{1}{2}}.$  We have thus verified that
$(\bar{x}_i,\bar{t}_i)$ is a $(\gd,\mu)$-thick blowup sequence.  By Proposition
\ref{thickblowup} we obtain that
 \begin{align*}
  \liminf_{i \to \infty}
\int_{B_{\frac{\bar{\gs}_2(\bar{x}_i,\bar{t}_i)}{2}}(\bar{x}_i,\bar{t}_i)}
\brs{\Rm}^2 dV >
\ge,
 \end{align*}
for $\ge$ depending only on $\gd$ and $\mu$, which in turn only depends on
$\FF(g_0)$ and $\mu$.  It is clear by construction that, for any fixed $r > 0$,
for sufficiently large $i$ we have
$B_{\frac{\bar{\gs}_2(\bar{x}_i,\bar{t}_i)}{2}}(\bar{x}_i,\bar{t}_i) \subset
B_r(x,\bar{t}_i)$, and so the theorem follows.
\end{proof}

\begin{proof}[Proof of Theorem \ref{threefoldtt}] Fix $(M^3, g_t)$ a solution to
the $L^2$ flow as in the statement, and fix $\mu > 0$.  If the claim of the
theorem was false, there exists $x$  such that the curvature blows up but
\begin{align} \label{ttt10}
 \liminf_{t \to T} d(x,T_{\mu}(g_t),t) > 0.
\end{align}
Choose a sequence of times $\{t_i\} \to T$ such
that
\begin{align*}
 \lim_{i \to \infty} \brs{\Rm}(x,t_i) = \infty.
\end{align*}
For $i \in \mathbb N$ sufficiently large, we can verify the setup of Lemma
\ref{ppl} with initial point $(x,t_i)$ and $A = 1$ to obtain a new point
$(\bar{x}_i,\bar{t}_i)$.  Moreover, by (\ref{ttt10}) we see that for
sufficiently large $i$, condition (\ref{pppnc}) is verified, and so there exists
a constant $\gd > 0$ so that
$\bar{\gs}_2(\bar{x}_i,\bar{t}_i) \geq \gd
f_2(\bar{x}_i,\bar{t}_i)^{-\frac{1}{2}}.$  We have thus verified that
$(\bar{x}_i,\bar{t}_i)$ is a $(\gd,\mu)$-thick blowup sequence, contradicting
Proposition \ref{thickblowup}.
\end{proof}

\begin{proof} [Proof of Corollary \ref{singularitygap}] To obtain this
corollary, given $\mu$ simply choose $\ge$ smaller than the constant $\ge(\mu)$
given in the statement of Theorem \ref{thickthin1}.  This rules out the first
possibility in the statement of Theorem \ref{thickthin1}, leaving only the
second possibility, as claimed.
\end{proof}

\begin{proof} [Proof of Corollary \ref{singularityconc}] First fix a constant
$\ga > 0$ and define the $\ga$-cutoff $k$-curvature radius to be
\begin{align*}
 \bar{\gs}_k^{\mbox{cut}}(x) := \min \{\bar{\gs}_k(x),\ga\}.
\end{align*}
With this radius one can define associated quantities $\nu^{\mbox{cut}}$ and
$T_{\mu}^{\mbox{cut}}$ as in \S \ref{bckgrnd}, and repeat the arguments of this
section with these new quantities.  Since our concern is only with points of
large curvature this cutoff has no effect on the proofs.  

If $\ga$ is chosen sufficiently small with respect to $\Omega$, by the
discussion in Corollary \ref{sobflatcor}, we know that there is a certain lower
bound $\mu$ on the volume ratio of balls whose radius is no bigger than $\ga$. 
In particular, for $\mu$ chosen sufficiently small we have that
$T_{\mu}^{\mbox{cut}} = \emptyset$.  Thus as in the proof of Theorem
\ref{thickthin1} we can obtain a $(\gd,\mu)$-thick blowup sequence, where $\gd$
depends on $\Omega$ on $\FF(g_0)$.  Arguing as in the proof of Theorem
\ref{thickthin1} using Proposition \ref{thickblowup} yields the corollary.
\end{proof}

\section{Calabi Flow} \label{cfsec}

As stated in the introduction, all of the results of this paper apply to the
Calabi flow on complex surfaces as well as the $L^2$ curvature flow.  Below we
summarize these
proofs.  First, we address the proof of Theorem \ref{smoothing1}, and the
corollaries in \S \ref{smoothingsec}.  The main tool exploited in these
arguments is the local smoothing of curvature derivatives in $L^2$, combined
with local covering arguments using the exponential map.  These $L^2$ smoothing
estimates for Calabi flow are exhibited in \cite{ChenHe} Theorem 3.1.  Strictly
speaking that theorem only applies to show global $L^2$ derivative estimates,
but by employing a cutoff function as in (\cite{SL21} Corollary 5.2,
\cite{SL2LTB} Theorem 4.4) one can localize these estimates.  Implicitly the
evolution equations for curvature and derivatives depend on the complex
structure $J$ as well as the metric $g$, but this presents no difficulty as one
can simply pull $J$ back as well in defining the solutions to Calabi flow on
local covers.  Next, the required $\ge$-regularity result analogous to Theorem
\ref{erthm} is \cite{CW} Theorem 1.1.  Lastly,
the point-picking result of Proposition \ref{ppp} exploits only the local
smoothing and energy monotonicity properties, and so carries over for Calabi
flow on complex surfaces.  With these tools in place, all of the discussion in
\S \ref{ttsec} follows identically to establish the concentration-collapse
decomposition.

As we have seen, the strongest results for the $L^2$ flow are necessarily
restricted to dimension $n \leq 4$.  However, for Calabi flow one expects much
stronger properties.  In particular Chen has conjectured long time existence of
the Calabi flow with arbitrary initial data in all dimensions.  The proofs here
rely on ``energy methods," which are well-suited to the $L^2$ flow and naturally
exploit the necessary dimensionality restriction.  New methods will likely be
required to obtain characterizations of Calabi flow singularities in higher
dimensions.

Also, one would like to generalize previous low-energy results for the Calabi
flow on surfaces \cite{ChenHe,ChenHe2} to a more general setting.  However, 
the techniques used to prove Theorem \ref{einv} do not immediately extend to
prove a corresponding gap theorem for K\"ahler manifolds with small Calabi
energy and a nearly Euclidean lower bound on the volume of sufficiently small
balls.  The reason is because of the failure of the Gap Lemma (\cite{Anderson}
Lemma 3.1) for complete scalar flat K\"ahler metrics.  If one assumes that
sufficiently small metric balls are also topological balls, then using rigidity
results for ALE scalar flat K\"ahler metrics it is possible to prove a gap
theorem.  This assumption is too difficult to ensure in practice though.  A
similar discussion applies to potentially using the $L^2$ flow to prove a gap
theorem near
scalar-flat, half-conformally flat manifolds.

\bibliographystyle{hamsplain}

\end{document}